\documentclass[11pt]{amsart}
\usepackage{graphicx}
\usepackage{eufrak}
\usepackage{amssymb, amsmath, amscd}
\usepackage{bm}
\usepackage{float}

\newtheorem{theorem}{Theorem}[section]
\newtheorem{corollary}[theorem]{Corollary}
\newtheorem{lemma}[theorem]{Lemma}
\newtheorem{prop}[theorem]{Proposition}

\newtheorem{defi}[theorem]{Definition}

\theoremstyle{definition}

\newcommand{\sst}{$\mbox{Spin}^{\mbox{\scriptsize{c}}}$}
\newcommand{\spc}[1]{\ensuremath{ \mbox{Spin}^{\mbox{\scriptsize{c}}}({#1})}}
\newcommand{\rspc}[1]{\ensuremath{ \underline{\mbox{Spin}}^{\mbox{\scriptsize{c}}}({#1})}}
\newcommand{\halbox}{$\mbox{\begin{picture}(0,0)\put( 2,0){\framebox(5,5)}\end{picture} }$}
\newcommand{\al}{\ensuremath{\alpha}}
\newcommand{\be}{\ensuremath{\beta}}
\newcommand{\ga}{\ensuremath{\gamma}}
\newcommand{\de}{\ensuremath{\delta}}
\newcommand{\la}{\ensuremath{\lambda}}
\newcommand{\bsa}{\ensuremath{\boldsymbol{\alpha}}}
\newcommand{\bsb}{\ensuremath{\boldsymbol{\beta}}}
\newcommand{\bsg}{\ensuremath{\boldsymbol{\gamma}}}
\newcommand{\bsd}{\ensuremath{\boldsymbol{\delta}}}
\newcommand{\bse}[1]{\ensuremath{\boldsymbol{\eta}^{#1}}}

\newcommand{\cfh}[1]{\ensuremath{\widehat{CF}({#1})}}

\newcommand{\m}[1]{\ensuremath{\mathfrak{#1} }}
\newcommand{\um}[1]{\ensuremath{\underline{\mathfrak{#1}} }}

\newcommand{\vs}{\vspace{10pt}}
\newcommand{\bline}[1]{\noindent \textbf{#1}}

\newcommand{\myfig}[1]{{#1}}

\newcommand{\1}{\mathbb{I}}

\begin{document}

\title[Twisted Floer homology of mapping tori]{On the twisted Floer homology of mapping tori of periodic diffeomorphisms}

\author[Evan Fink]{$\mbox{Evan Fink}^1$}

\address{Department of Mathematics,
Columbia University, 
New York, NY 10027}
\email{fink@math.columbia.edu}

\begin{abstract}

Let $K \subset Y$ be a knot in a three manifold which admits a longitude-framed surgery such that the surgered manifold has first Betti number greater than that of $Y$.  We find a formula which computes the twisted Floer homology of the surgered manifold, in terms of twisted knot Floer homology.  Using this, we compute the twisted Heegaard Floer homology $\underline{HF}^+$ of the mapping torus of a diffeomorphism of a closed Riemann surface whose mapping class is periodic,  giving an almost complete description of the structure of these groups.   When the mapping class is nontrivial, we find in particular that in the ``second-to-highest'' \sst\ structure, this is isomorphic to a free module (over a certain ring) whose rank is equal to the Lefschetz number of the diffeomorphism.  After taking a tensor product with $\mathbb{Z}/2\mathbb{Z}$, this agrees precisely with the symplectic Floer homology of the diffeomorphism, as calculated by Gautschi.

\end{abstract}

\maketitle

\footnotetext[1]{The author was supported by NSF grant DMS-0739392.}

\section{Introduction}
\label{sect:intro}
Since its introduction in the 1980's, Floer homology has become a vast subject, its many variants proving useful to problems all over geometry and topology.  In three manifold topology and knot theory,  Heegaard Floer homology has been particularly successful. It has, for example, been applied to questions of unknotting numbers, lens space surgeries, and fiberedness of knots and manifolds; and it has strong connections with Seiberg-Witten theory and Khovanov homology.  There are also connections with the symplectic Floer homology that it initially grew out of, but these are perhaps slightly less developed.  

The main motivating purpose of this paper is to furnish one connection: for a periodic orientation-preserving diffeomorphism $\phi$ of a closed Riemann surface $\Sigma$ of genus $g_{\Sigma}$, we compute the twisted coefficient Heegaard Floer homology of its mapping torus, and compare this with the symplectic Floer homology of $\phi$.  We find that the ``second-to-highest level'' of the twisted Floer homology is $\mathbb{Z}^{\Lambda(\phi)}$ tensored with a certain module depending only on the genera of $\Sigma$ and of the underlying surface of the quotient orbifold of $\Sigma$ by the action of $\phi$, where $\Lambda$ denotes Lefschetz number.  (Here, ``$n$th-to-highest level'' means the direct sum of the twisted Floer homologies for those \sst\ structures $\m{t}$ for which $\langle c_1(\m{t}), \Sigma \rangle = 2n-2g_{\Sigma}.$)  This matches with the computation of the symplectic Floer homology of $\phi$ over $\mathbb{Z}_2$, which is computed in \cite{G} to be $\mathbb{Z}_2^{\Lambda(\phi)}$.  
In fact, this computation is part of a wider pattern encompassing all the levels of the homology, which is described in Theorem \ref{thm:1.1}.

Before detailing our results, we review some facts about Seifert fibered spaces and the mapping tori in question.  The following is described in \cite{Scott}. The mapping torus of any periodic, orientation-preserving diffeomorphism of a closed Riemann surface is an orientable Seifert fibered space with orientable base orbifold, of degree 0.  The latter specification is equivalent to saying that the space itself has odd first Betti number.  In fact, any Seifert fibered space of this type can be realized as a mapping torus for such a diffeomorphism. 

Any oriented Seifert fibered space orientable base can be realized also as a surgery on a connect sum of knots of the following two types.  First, on the Borromean rings, perform $0$-surgery on two of the components; then the third component is the Borromean knot $B_1 \subset \#^2S^1 \times S^2$, and we write $B_g$ for the $g$-fold connect sum of copies of $B_1$.  Second, on the Hopf link, perform surgery with coefficient $-p/q$ on one component; then the other component is the $O$-knot $O_{p,q} \subset L(p,q)$.  We always assume that $0 < q < p$ and that $p$ and $q$ are relatively prime.  
Let $K = B_g \#_{\ell =1}^n O_{p_{\ell}, q_{\ell}} \subset Y = \#_{i=1}^g \#^2 S^1 \times S^2 \#_{\ell =1}^n L(p_{\ell}, q_{\ell})$.  Then if 
$$\sum_{\ell=1}^n \frac{q_{\ell}}{p_{\ell}} \in \mathbb{Z},$$
$K$ admits a longitude $\la$ (unique up to isotopy) such that $\la$-framed surgery on $K$ yields a manifold $Y_{\la}(K)$ with odd Betti number, which is therefore a mapping torus of the type we are interested in.  The base orbifold will have genus $g$, and the genus $g_{\Sigma}$ of the Riemann surface being mapped will be given by 
$$g_{\Sigma} = 1 + d\left(g - 1 + \frac{1}{2}\sum_{\ell=1}^n \left(1 - \frac{1}{p_{\ell}} \right)\right),$$
where $d$ is the order of $K$ in $H_1(Y)$.  It is not hard to see that $d$ is the least common multiple of the $p_{\ell}$ values.  

We can take a Seifert surface for $dK$ in $Y$, and then cap this off in the obvious manner in $Y_{\la}(K)$ (hereafter denoted as $Y_0$) to get an element $[\widehat{dS}] \in H_2(Y_0)$.  Such an element will not be unique, but all results that make reference to this class are true for all such choices.

We now explain our main results more precisely.  For $Y_0$ as above, we compute $\underline{HF}^+(Y_0)$, where the underscore denotes totally twisted coefficients (as described in Section 8 of \cite{OSPA}).  
First, let us describe $\underline{HF}^+\left(Y_0\right)$ crudely, neglecting some of the finer structure (e.g. $U$-actions, relative gradings).

Let $\mu$ be a meridian of $K$, thought of as an element of $H_1(Y_0)$.
If $\m{t}_0$ is a \sst\ structure on $Y_0$ for which $c_1(\m{t}_0)$ goes to an element of $\mathbb{Q}\cdot\mbox{PD}[\mu]$ in $H^2(Y_0; \mathbb{Q})$, we say that $\m{t}_0$ is \emph{$\mu$-torsion}.  Let $\mu\mathcal{T}_K$ denote the set of $\mu$-torsion elements of $\spc{Y_0}$.

For $D,E \in \mathbb{Z}$, let $\mathcal{N}(D,E)$ denote the number of solutions $(i_1, \ldots, i_n)$ to the equation
$$\sum_{\ell=1}^n \frac{i_{\ell}}{p_{\ell}} = \frac{E}{d} + (D-g+1)$$
for which $0\leq i_{\ell} < p_{\ell}$ for all $\ell$.  

The wider pattern alluded to above is given by the following.

\begin{theorem}
\label{thm:1.1} 
There are groups $\Omega^g(k)$, which depend only on $k$ and $g$ (and not on $Y_0$), and which are trivial for $k < 1$ and $k > 2g - 1$, such that the following holds.
For $0 \leq i \leq g_{\Sigma}-2$, let 
$$\underline{HF}^+\left(Y_0, [-i]\right) = \bigoplus_{\left\{\m{t} \in \mu\mathcal{T}_K | \langle c_1\left(\m{t}\right), [\widehat{dS}] \rangle = 2(g_{\Sigma}-i-1) \right\}} \underline{HF}^+\left(Y_0, \m{t}\right),$$
where the summands are thought of as ungraded $\mathbb{Z}[H^1(Y_0)]$-modules (and we forget about the $U$-action).
Then we have a short exact sequence  
$$ 0\rightarrow \left(\bigoplus_{k}\left(\Omega^g(k) \right)^{b_{i,k}} \right) \otimes \mathbb{Z}[T, T^{-1}] \rightarrow \underline{HF}^+\left(Y_0, [-i]\right) \rightarrow  \mathbb{Z}^{a_i} \otimes \mathbb{Z}[T, T^{-1}] \rightarrow 0 $$ where 
$$a_i = \sum_{D \in \mathbb{Z}} \mathrm{max}\{0,-D, \lfloor \frac{g-D+1}{2} \rfloor\}\cdot \mathcal{N}(D,i)$$
and
$$b_{i,k} = \mathcal{N}(k-g,i),$$
where $T \in \mathbb{Z}[H^1(Y_0)]$ represents the Poincar{\'e} dual of a fiber of $Y_0$ thought of as a mapping torus.  If we consider the terms of the above sequence only as abelian groups rather than as $\mathbb{Z}[H^1(Y_0)]$-modules, then the sequence splits.
\end{theorem}

Together with conjugation invariance, this describes $\underline{HF}^+$ for all those \sst\ structures where it is non-trivial, aside from the torsion ones. In fact, the same description works for the torsion structures, with $i = g_{\Sigma} - 1$, except that there is an extra summand, $\mathcal{T}^+$, which is described below.  

The following Corollary gives the mentioned connection with symplectic Floer homology.

\begin{corollary}
\label{thm:1.2} 
Let $\phi:\Sigma \rightarrow \Sigma$ be a periodic diffeomorphism of a closed Riemann surface, whose mapping class is not trivial.  Let its mapping torus be $Y_0$, and set $R = \mathbb{Z}[H^1(Y_0)]$.
Then, as $R$-modules, 
$$\underline{HF}^+\left(Y_0, [-i]\right) \cong 
\left\{ \begin{array}{ll}
R, & i=0\\
R^{\Lambda(\phi)}, & i=1\\
\end{array} \right.$$
where $\Lambda$ denotes Lefschetz number.  Furthermore, we have:
the $U$-action is trivial in each; for $i=1$, each copy of $R$ lies in a different \sst\ structure, and each copy of $R$ is supported in a single relative $\mathbb{Z}$-grading; and if $T$ represents the Poincar{\'e} dual of a fiber in $R$, then $T$ lowers this relative grading by $2d(g_{\Sigma} - 1 - i)$.  
\end{corollary}

We extend this to result to the third-to-highest level in Section 9.
In particular, it is expected that this level is related to the periodic Floer homology developed by Hutchings \cite{H}.  

On a different front, it is also interesting to compare the wider structure of Theorem \ref{thm:1.1} with the computations of \cite{MOY}, which tackles the analogous computation for Seiberg-Witten-Floer homology and finds results that have at least some similarity to ours.  On yet another front, in \cite{Wu1} and \cite{Wu2}, computations of perturbed Floer homology -- that is, Heegaard Floer homology with a special type of twisted coefficients -- are carried out for some mapping tori, including $S^1 \times \Sigma_g$. 

The groups $\Omega^g(k)$ are given, in Definition \ref{def:8.5}, in terms of the twisted filtered chain complex of $B_g$.
We do not describe the structure of these groups explicitly, leaving them as mystery subgroups.   We would like a better description of them, but their presence ends up as only a minor distraction here.  In many instances, they don't show up at all: when we look at a mapping torus with first Betti number 1; in Corollary \ref{thm:1.2}; and generally, in many relative grading levels of any space of the type we examine. 

Now, we describe the full structure of $\underline{HF}^+\left(Y_0\right)$ (to the extent that we can), from which we extract Theorem \ref{thm:1.1} and Corollary \ref{thm:1.2}.  To do this, we have to introduce some machinery.

First, we introduce some notation to keep track of $\mu$-torsion \sst\ structures.
Define
$$\widetilde{\mathcal{MT}_K} = \mathbb{Z} \times \bigoplus_{\ell=1}^n \mathbb{Z}/p_{\ell}\mathbb{Z}.$$
We write elements of $\widetilde{\mathcal{MT}_K}$ as pairs $(Q;r_1, \ldots, r_n)$, where $r_{\ell}$ is an integer satisfying $0 \leq r_{\ell} \leq p_{\ell}-1$; we usually shorten this and just write elements as $(Q;r)$ (with $r$ understood as denoting an $n$-tuple), or simply as $A$.  Let
\begin{equation}
\label{eq:1}
S\ell (Q;r) = \sum_{\ell=1}^n \left(1 - \frac{1}{p_{\ell}}\right) + 2 \left(Q + \sum_{\ell=1}^n \frac{r_{\ell}}{p_{\ell}}\right).
\end{equation}

Define an equivalence relation $\sim$ on $\widetilde{\mathcal{MT}_K}$, by setting $(Q;r) \sim (Q';r')$ if and only if $S\ell (Q;r) = S\ell (Q';r')$ and $r$ and $r'$ descend to the same element of the quotient group $\left(\bigoplus_{\ell=1}^n \mathbb{Z}/p_{\ell}\mathbb{Z}\right) / \mathbb{Z}(q_1 \oplus \ldots \oplus q_n)$.  Then, let 
$$\mathcal{MT}_K = \widetilde{\mathcal{MT}_K}/\sim ;$$
we write equivalence classes as $[A]$.
The function $S\ell$ obviously extends to $\mathcal{MT}_K$, as does the function $\epsilon: \widetilde{\mathcal{MT}_K} \rightarrow \mathbb{Q}$ given by
$$\epsilon(A) = g_{\Sigma}-1 - \frac{d}{2}\cdot S\ell(A).$$

\begin{lemma}
\label{thm:1.3} 
There is a bijective map 
$$\theta_K: \mathcal{MT}_K \rightarrow \mu\mathcal{T}_K,$$
which satisfies 
\begin{equation}
\label{eq:2}
S\ell([A]) = -\frac{\langle c_1\left(\theta_K([A])\right), [\widehat{dS}] \rangle}{d}
\end{equation}
for $[A] \in \mathcal{MT}_K$.
\end{lemma}

Next, to describe our answers neatly, we use (a slightly altered version of) the concept of wells introduced in \cite{OSRS}, based on ideas in \cite{N} and \cite{OSP3}.   In our version, for a function $H:\frac{1}{2}\mathbb{Z} \rightarrow \mathbb{Z}$ and an \textbf{odd} integer $n$, we define a \emph{well at height $n$ for $H$} to be a pair of integers $(i, j)$, $i<j$, which satisfy $n < H(i)$, $n < H(j)$, and $n \geq H(s)$ for $i + \frac{1}{2} \leq s \leq j - \frac{1}{2}$.  We also write such a pair as $(i, j)_n$ to denote the height of the well.  Let $W_n(H)$ be the set of wells at height $n$ for $H$;  and let $\mathbb{W}_n(H)$ be the free abelian group generated by $W_n(H)$.  Then we define 
$$ \mathbb{W}_*(H) = \underset{l \in \mathbb{Z}}{\oplus} \mathbb{W}_{2l+1}(H). $$
We write $(i', j')_{n-2} < (i,j)_n$ if $(i', j')_{n-2} \in W_{n-2}(H)$ and 
$i \leq i' < j' \leq j$.
Then we can define an action of $U$ on $\mathbb{W}_*(H)$ by
$$ U\cdot (i,j)_n = \sum_{\{w \in W_{n-2}(H)| w < (i,j)_n\}} w,$$
and extending linearly.  The movitation for the name ``wells'' can be seen by plotting the graph of $H$.  We endow $\mathbb{W}_*(H)$ with the $\mathbb{Z}$-grading given by height.

The well functions we need are given as follows.
For $A \in \widetilde{\mathcal{MT}_K}$, we define a function
$$\eta_A(x) = \sum_{\ell=1}^n \left\{ \frac{q_{\ell}x - r_{\ell}}{p_{\ell}} \right\} - \frac{1}{d}\epsilon(A) + g -1, $$
where the curly braces denote fractional part, $\{x\} = x -\lfloor x \rfloor$.
We then define a function $G_A:\frac{1}{2}\mathbb{Z} \rightarrow \mathbb{Z}$ by
$$G_A(0) = 1,$$
$$G_A(x) = G_A(x-1) + 2\eta_A(x) \mbox{ for } x \in \mathbb{Z},$$
$$G_A\left(x+ \frac{1}{2}\right) = \frac{1}{2}\left(G_A(x) + G_A(x+1) \right) \mbox{ for } x \in \mathbb{Z}.$$
Finally, define $F_A:\frac{1}{2}\mathbb{Z} \rightarrow \mathbb{Z}$ by
$$F_A(x) = G_A(x) + \left\{ \begin{array}{ll} g, & x \in \frac{1}{2} + \mathbb{Z} \\ 0, & x \in \mathbb{Z}. \end{array} \right. $$
The function $F_A$, very roughly, describes relative gradings of certain elements in a certain set of chain complexes, this set being parametrized by $x$.

Recall that $\mathcal{T}^+_{(s)}$ denotes the $\mathbb{Z}[U]$-module $\mathbb{Z}[U, U^{-1}]/U\cdot\mathbb{Z}[U]$, equipped with a $\mathbb{Z}$-grading so that $U^{-i}$ lies in level $s + 2i$ for $i \geq 0$.

\begin{theorem}
\label{thm:1.4}
For $A \in \widetilde{\mathcal{MT}}_K$, let
$$ B_{s}(F_A) \cong \bigoplus_{ \{p\in \mathbb{Z} | F_A(p+\frac{1}{2}) = s + 1\} } \Omega^g\left(F_A(p+\frac{1}{2}) - F_A(p)\right),$$
and let $B_*(F_A) = \bigoplus_{s \in \mathbb{Z}} B_s(F_A)$ (where $\Omega^g(k)$ is the group from the statement of Theorem \ref{thm:1.1}).  Equip this group with trivial $U$-action.
Let $b_A \in 2\mathbb{Z} \cup \{\infty\}$ be the least even upper bound of the function $F_A$. 

Then, if $\m{t}_0 = \theta_K([A])$ is $\mu$-torsion, the relative $\mathbb{Z}$-grading on 
$\underline{HF}^+(Y_0, \m{t}_0)$ lifts to an absolute $\mathbb{Z}$-grading, such that there are short exact seqeunces of graded $\mathbb{Z}[H^1(Y_0)]\otimes \mathbb{Z}[U]$-modules
$$ 0 \rightarrow B_*(F_A) \oplus \mathcal{T}^+_{(b_A)} \rightarrow \underline{HF}^+(Y_0, \m{t}_0) \rightarrow \mathbb{W}_*(F_A) \rightarrow 0 $$
if $b_A \ne \infty$ and
$$ 0 \rightarrow B_*(F_A) \rightarrow \underline{HF}^+(Y_0, \m{t}_0) \rightarrow \mathbb{W}_*(F_A) \rightarrow 0 $$
if $b_A = \infty$; and otherwise $\underline{HF}^+(Y_0, \m{t}_0)$ is trivial.  Furthermore, $b_A \ne \infty$ precisely when $\theta_K([A])$ is torsion. 

The function $F_A$ is the sum of a periodic function with a linear one.  
If we think of $Y_0$ as a mapping torus, and $T \in \mathbb{Z}[H^1(Y_0)]$ represents the Poincar{\'e} dual of a fiber, then $T$ acts on $\mathbb{W}_*(F_A)$ by moving a well to the corresponding one a period to the right. 
\end{theorem}

The short exact sequences above are not necessarily split over $\mathbb{Z}[H^1(Y_0)] \otimes \mathbb{Z}[U]$, but we have the following.

\begin{corollary}
\label{thm:1.5}
If $Y_0$ is the mapping torus of a periodic diffeomorphism, then $\underline{HF}^+(Y_0)$ contains no torsion as an abelian group; and so the short exact sequences of Theorem \ref{thm:1.4} are split over $\mathbb{Z}[U]$.  In particular, $\underline{HF}^+(S^1 \times \Sigma_g)$ contains no torsion as a group.  
\end{corollary}

This corollary is also to be compared with \cite{JMX}, where it is found that $HF^+(S^1 \times \Sigma_g)$ contains torsion for large enough $g$.

Together with the results of \cite{OSRS}, Theorem \ref{thm:1.4} in a sense completes the calculation of the Heegaard Floer homology of Seifert-fibered spaces.  The qualification comes from two sources: the fact that we don't describe the subgroups $\Omega^g(k)$ explicitly, and the fact that by ``the calculation of Heegaard Floer homology'', we mean the calculation of one of $HF^+$ or $\underline{HF}^+$.

Theorem \ref{thm:1.4} is shown using a twisted surgery formula, akin to the formulas of \cite{OSIS} and \cite{OSRS}.  Before outlining the formula, we say a brief word about our use of twisted coefficients.  Both our formula and those of \cite{OSIS} and \cite{OSRS} come about by relating knot Floer homology of a knot with the cobordism maps induced by attaching a two-handle along the knot, eventually arriving at the Floer homology of a three-manifold obtained by surgery along the knot. The twisted coefficient setting is useful for sorting out how the cobordism-induced maps break down into summands for each \sst\ structure on the cobordism, especially when we have surgeries that raise the first Betti number, as we encounter here.  
Indeed, an untwisted version of the formula we use, if it exists, would likely be much less user-friendly. 

We now give a quick description of the formula, which computes $\underline{HF}^+(Y_{\la}(K), \m{t}_0)$ when $\la$ is a longitude of $K$ such that $b_1(Y_{\la}(K)) = b_1(Y) + 1$, and $\m{t}_0$ is $\mu$-torsion with respect to the meridian $\mu$ of $K$ (i.e., $c_1(\m{t}_0)$ goes to an element of $\mathbb{Q}\cdot\mbox{PD}[\mu]$ in $H^2(Y_{\la}(K); \mathbb{Q})$). 
Consider the two handle cobordism $W_0$ obtained by attaching a $\la$-framed 2-handle to $K$.  Let $\m{t}^0_{\infty} \in \spc{Y}$ be some \sst\ structure cobordant to $\m{t}_0$, and let $\m{t}^i_{\infty}$ be $\m{t}^0_{\infty} - i\mbox{PD}[K]$ for $i\in \mathbb{Z}$.  If $d$ is the order of $K$ in $H_1(Y)$, then $\m{t}^{i+d}_{\infty}$ will be the same as $\m{t}^i_{\infty}$, but we nonetheless treat them as distinct in what follows.

We recall that there is a notion of relative \sst\ structures on $(Y,K)$, which we denote by $\rspc{Y,K}$; there is a natural map $G_K: \rspc{Y,K} \rightarrow \spc{Y}$, whose fibers are the orbits of a $\mathbb{Z}\cdot\mbox{PD}[\mu]$ action on $\rspc{Y,K}$.  To each relative \sst\ structure $\xi$ we may associate the group $CFK^{\infty}(Y,K, \xi)$, which is generated over $\mathbb{Z}$ by elements of the form $[\mathbf{x}, i, j]$, where $\mathbf{x}$ is a generator of $\widehat{CF}\left(Y, G_K(\xi)\right)$ and $i, j$ are integers, required to satisfy a certain condition which depends on $\mathbf{x}$.  Then, as a group, $\underline{CFK}\{i \geq 0 \mbox{ or }j \geq 0\}(Y,K, \xi)$ is generated by elements of the form $[\mathbf{x}, i, j] \otimes r$, where $[\mathbf{x}, i, j]$ is one of the generators of $CFK^{\infty}(Y,K, \xi)$ for which $i \geq 0$ or $j \geq 0$, and $r$ is an element of the group ring $\mathbb{Z}[H^1(Y)]$.  The differential is analogous to that used in ordinary twisted Floer homology.  We give more detail in Section 5.

There is a chain map
$$ v_{\xi}: \underline{CFK}\{i \geq 0 \mbox{ or }j \geq 0\}(Y,K, \xi) \rightarrow \underline{CF}^+\big(Y, G_K(\xi) \big), $$
which simply takes the generator $[\mathbf{x}, i, j] \otimes r$ to $[\mathbf{x}, i] \otimes r$.  
There is also a map 
$$ h_{\xi}: \underline{CFK}\{i \geq 0 \mbox{ or }j \geq 0\}(Y,K, \xi) \rightarrow \underline{CF}^+\big(Y, G_K(\xi) - \mbox{PD}[K] \big), $$
which takes the generator $[\mathbf{x}, i, j] \otimes r$ to $[\mathbf{x}, j] \otimes r$, the latter now belonging to the same Heegaard diagram but with different basepoint (and hence representing a different \sst\ structure).

\begin{theorem}
\label{thm:1.6}
Let $K, \la,$ and $\m{t}_0$ be as above.  There are elements $\xi_i \in \rspc{Y,K}$ for each $i \in \mathbb{Z}$ such that $G_K(\xi_i) = \m{t}^i_{\infty}$, $\xi_{i+d} = \xi_i$, and so that the following holds. Let 
$$\underline{f}^+_{K, \m{t}_0}: \bigoplus_{i \in \mathbb{Z}} \underline{CFK}\{i \geq 0 \mbox{ or }j \geq 0\}(Y,K, \xi_i) \rightarrow \bigoplus_{i \in \mathbb{Z}} \underline{CF}^+\big(Y, \m{t}_i)  $$
be the map given by the direct sum of the maps $v_{\xi_i}$ and $h_{\xi_i}$ over all $i$, where each $\xi_i$ and $\m{t}^i_{\infty}$ is treated as distinct, and $v_{\xi_i}$ and $h_{\xi_i}$ are considered to take summand $i$ on the left to summands $i$ and $i+1$ on the right, respectively.  Then, there is a quasi-isomorphism from $M(\underline{f}^+_{K, \m{t}_0})$ to $\underline{CF}^+(Y_0, \m{t}_0)$, where $M$ denotes the mapping cone. Furthermore, $M(\underline{f}^+_{K, \m{t}_0})$ admits a relative $\mathbb{Z}$-grading and a $U$-action which the quasi-isomorphism respects.
\end{theorem}

We give a more precise version of the formula in Theorem \ref{thm:6.1}.  In particular, we will describe the structures $\xi_i$ precisely.
We believe that the details of \cite{OSRS} can be mimiced along the lines of this paper, to get a twisted surgery formula that applies when $\la$ is an arbitrary longitude, but we don't undertake this herein.

At least at the level of concept, twisted knot Floer homology is a more or less straightforward combination of twisted Floer homology and untwisted knot Floer homology.  For this reason, twisted knot Floer homology has already been used in, for example, \cite{JM} and \cite{Ni}, despite having had only the barest of definitions written down (as far as the author can tell).  As we rely on this more extensively, and for fairly delicate computations leading to Theorem \ref{thm:1.4}, we give a slightly fuller treatment here.

\subsection{Organization}
\label{intro:org}
In Section 2, we introduce special Heegaard diagrams that we use throughout, and treat relative \sst\ structures.  We put some of the more tedious results of this section in Appendix A.  In Section 3, we make some observations about certain triangles in our diagrams, and the \sst\ structures they represent.  In Section 4, we prove a twisted coefficient long exact sequence.  In  Section 5, we introduce twisted knot Floer homology, and give analogues of the large $N$ surgeries formula and the K\"unneth formula.  In Section 6, we state and prove the twisted surgery formula.  In Section 7, we make basic computations for the Borromean knots and $O$-knots, reducing the work for Theorem \ref{thm:1.4} to algebra, which is then carried out in Section 8.   We prove Theorem \ref{thm:1.1} and Corollary \ref{thm:1.2} in Section 9.  Examples are presented in Section 10.

\subsection{Acknowledgement}
I would like to thank my advisor, Peter Ozsv{\'a}th, for suggesting this problem, for his guidance and support during the preparation of this paper, and for his generosity of time and patience with me as a student.

\section{Standard Heegaard Diagrams and Relative \sst\ Structures}
We introduce a special type of Heegaard diagram associated to a knot in a three manifold.  We make reference to these diagrams when we develop twisted knot Floer homology, and when we prove the long exact sequence of Section 4.  In the present section, we also use these to solidify a connection
between relative \sst\ structures on a knot complement and \sst\ structures on a cobordism gotten by attaching a two-handle to the knot.

Throughout this paper, all knots we consider are oriented rationally nullhomologous knots $K \subset Y$, always implicitly equipped with a distinguished longitude $\la$.  When $K$ and $\lambda$ are understood, we henceforth write $Y_0$ for $Y_{\la}(K)$, and more generally $Y_N$ for $Y_{N\mu + \la}(K)$.   

Much of what we say will be applicable to all such knots, but the main results concern knots of the following type.

\begin{defi}
\label{thm:2.1}
If an oriented rationally nullhomologous knot $K \subset Y$ admits a longitude \la\ satisfying $b_1\big(Y_{\la}(K)\big) = b_1(Y) + 1$, then we call the knot (and the distinguished longitude) \emph{special}.
\end{defi}

The reason that we tend to not explicitly mention longitudes lies in the fact that any knot admits at most one special longitude, so that for special knots the choice of longitude is canonical.

\subsection{Standard Heegaard diagrams and cobordisms} 
For an (oriented rationally nullhomologous) knot $K \subset Y$ and a positive integer $N$, we can form a weakly admissable, doubly pointed Heegaard quadruple $(\Sigma, \bsa, \bsb, \bsg, \bsd, w, z)$ on a genus $g$ Riemann surface $\Sigma$ such that
$Y_{\al\be} = Y_0$, $Y_{\al\ga} = Y_N$, and $Y_{\al\de} = Y$, and such that $(\Sigma, \bsa, \bsd, w, z)$ is a Heegaard diagram for $(Y,K)$.  We
restrict attention to certain diagrams.

\begin{defi}
\label{thm:2.2}
Suppose that we have a weakly admissable, doubly pointed Heegaard quadruple as above, for which the following holds.  A portion of the diagram can be drawn as in Figure \myfig{1}, which takes place in a punctured torus.  Also, the vertically-drawn curve \de$_g$ is a meridian for $K$; the curve \be$_g$, oriented as shown, is a special longitude for $K$; the points $w$ and $z$ flank \de$_g$ as shown, with $w$ on the left.  Finally, the \be, \ga\ and \de\ curves not shown are all small isotopic translates of each other. Then, we call such a diagram \emph{standard}.  We will often forget parts of the data (e.g., we consider a Heegaard triple $(\Sigma, \bsa, \bsg, \bsd, w)$, ignoring $\bsb$ and $z$); we will also refer to such diagrams as standard.
\end{defi}

It is not hard to see that for any oriented knot in a three manifold and any $N$, there exist standard diagrams.   

We have usual ``highest'' intersection points $\Theta_{\be\de} \in \mathbb{T}_{\be} \cap \mathbb{T}_{\de}$ and $\Theta_{\ga\de} \in \mathbb{T}_{\ga} \cap \mathbb{T}_{\de}$, which have components as shown in Figure \myfig{1}.  We also choose a point $w_{\be\ga}$ on $\be_g \cap \ga_g$, which determines a ``highest'' intersection point $\Theta_{\be\ga}$ in $\cfh{Y_{\be\ga}} = \cfh{L(N,1)\#^{g-1}S^1 \times S^2}$, as follows.    $L(N,1)$ is realized as the boundary of a disk bundle $V$ over a sphere with Euler number $N$, and there is a unique \sst\ structure on $L(N,1)$ that extends to some $\m{s}$ on $V$ such that $\langle c_1(\m{s}), V \rangle = N$.  We choose the point $w_{\be\ga}$ so that $\m{s}_w(w_{{\be}{\ga}})$ is the unique \sst\ structure that is torsion and restricts to this \sst\ structure on $L(N,1)$.  

\begin{figure}[t!]
\label{fig:1}
\centering \includegraphics[scale=.50]{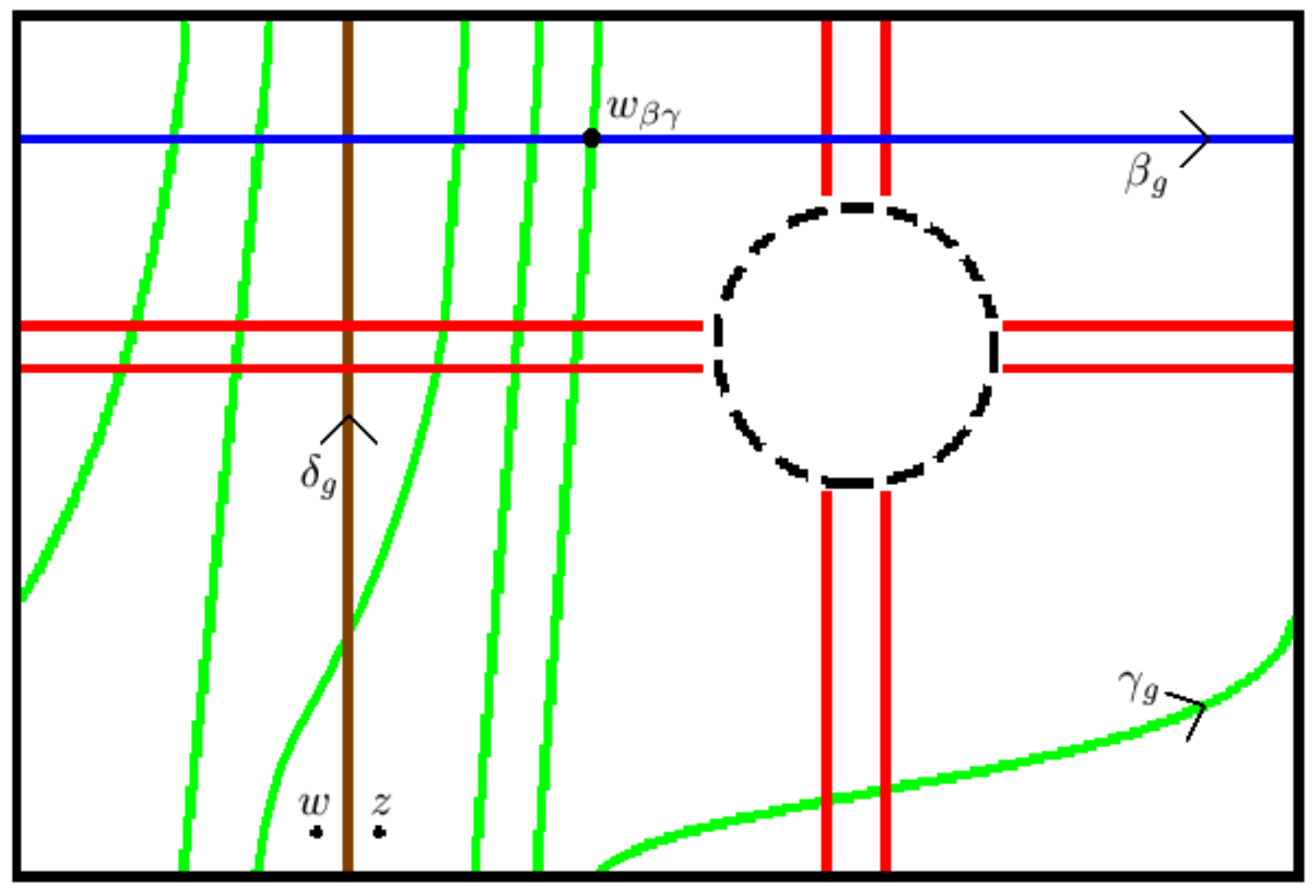}
\caption {Local picture of a standard diagram inside a punctured torus region, where the sides of the rectangle are identified in the usual way, and the dashed circle represents the puncture.  The red curves are strands of $\alpha$ circles.  We supply $\be_g$, $\ga_g$ and $\de_g$ with orientations for later reference.  The winding region is the portion of the above picture that lies to the left of the puncture.}
\end{figure}

For a standard diagram, we will also fix iterated small isotopies \bsb$_{(i)}$, \bsg$_{(i)}$, and \bsd$_{(i)}$ for positive integers $i$, which will play a role in the proof of the long exact sequence.  Thinking of \bsb, \bsg\ and \bsd\ as \bsb$_{(0)}$, \bsg$_{(0)}$, and \bsd$_{(0)}$, we will sometimes write $\bse{0} = \bsa$, $\bse{3i+1} = \bsb_{(i)}$, $\bse{3i + 2} = \bsg_{(i)}$, and $\bse{3i + 3} = \bsd_{(i)}$.  Likewise, we will have points $\Theta_{\eta^i\eta^j}$ for all $j > i > 0$, defined in the obvious manner.

The Figure also shows a natural partition of $\mathbb{T}_{\al} \cap \mathbb{T}_{\ga}$ into points that are \emph{supported in the winding region} -- those for which the point on $\ga_g$ is on one of the horizontal $\al$ strands -- and those that are not.  The \emph{winding region} on $\ga_g$ itself is the portion of that circle that is visible in the Figure, running between the leftmost and rightmost intersections with the horizontal $\al$ strands. For $\mathbf{x} \in \mathbb{T}_{\al} \cap \mathbb{T}_{\ga}$ and $\mathbf{y} \in \mathbb{T}_{\al} \cap \mathbb{T}_{\de}$, we say a triangle in $\pi_2(\mathbf{x}, \Theta_{\ga\de}, \mathbf{y})$ is \emph{small} if its boundary has component on $\ga_g$ contained entirely in the winding region.

Each point $\mathbf{x} \in \mathbb{T}_{\al} \cap \mathbb{T}_{\ga}$ that is supported in the winding region has a unique \emph{nearest} point $\mathbf{y}$ in $\mathbb{T}_{\al} \cap \mathbb{T}_{\de}$, for which there exists a small triangle in $\pi_2(\mathbf{x}, \Theta_{\ga\de}, \mathbf{y})$.  We also say that $\mathbf{x}$ is a nearest point of $\mathbf{y}$.

Let $W_N$ be the cobordism gotten by attaching a 2-handle to $Y$ along the longitude $N\mu + \la$, with boundary $-Y \coprod Y_N$.  Note that $H_2(Y_N)$ and $H_2(Y)$ inject into $H_2(W_N)$;  furthermore, if $N\mu + \la$ is a longitude that is \emph{not} special, they in fact have the same image in $H_2(W_N)$, since the inclusion of $Y \setminus K$ into $Y$ and $Y_N$ induces isomorphisms on second homology. 

We write $W'_N$ for $-W_N$, thought of as a cobordism from $Y_N$ to $Y$.  Then a standard diagram for $Y,K$ yields a Heegaard triple $(\Sigma, \bsa, \bsg, \bsd, w)$ for $W'_N$.  
The above observation then has consequences for the periodic domains appearing in a standard diagram, which are identified with elements of second homology (with respect to the basepoint $w$).  Indeed, each $\bsa\bsg-$periodic domains (that is, domains whose boundary is composed of an integral linear combination of entire $\al_i$ and $\ga_i$ circles) will correspond to a $\bsa\bsd-$periodic domain; the latter will have the same boundary as the former, except with $\ga$ circles replaced by the corresponding $\de$ translates. In particular, none of the $\bsa\bsg-$ or $\bsa\bsd-$periodic domains will have boundary including a nonzero multiple of $\ga_g$ or $\de_g$.  

Likewise, there will also be $\bsa\bsb-$periodic domains corresponding to the $\bsa\bsg-$periodic domains (again gotten by comparing boundaries).  However, for a special longitude, we will need to add one extra $\bsa\bsb-$periodic domain to generate the set of all $\bsa\bsb-$periodic domains.    

We can of course say corresponding things about the periodic domains between the $\al$ and $\eta^i$ circles.

\subsection{Relative \sst\ structures}  
There are a couple of slightly different (but equivalent) descriptions of relative \sst\ structures on three-manifolds with torus boundary.
We review the description we use.

There is a unique homotopy class of nonwhere vanishing vector fields on the torus such that if we take a covering map of $\mathbb{R}^2$ to the torus, any representative vector field lifts to one that is homotopic through nowhere vanishing vector fields to a constant vector field (identifying the tangent space of a point on $\mathbb{R}^2$ with $\mathbb{R}^2$ itself).  We refer to this as the canonical vector field on the torus.  Given any Dehn filling of the torus, the canonical vector field extends to a nonwhere vanishing vector field on the filling.

Given an oriented rationally nullhomologous knot $K$ in a three-manifold $Y$, let $\rspc{Y,K}$ be the set of homology classes of nowhere-vanishing vector fields on $Y \setminus N(K)$ that agree with the canonical vector field on the boundary (and in particular point along the boundary), where $N(K)$ is a regular neighborhood of $K$; we refer to these as \emph{relative} \sst\ \emph{structures}.  Here, two vector fields are homologous if they are homotopic through non-vanishing vector fields in the complement of a ball in the interior of $Y \setminus N(K)$.  
The set $\rspc{Y,K}$ is an affine space for $H^2(Y,K)$, in the same way that absolute \sst\ structures form an affine space for $H^2(Y)$.  

If we take the two-plane field $\vec{v}^{\perp}$ of vectors orthogonal to $\vec{v}$ representing $\xi \in \rspc{Y,K}$, we have a well-defined global section of $\vec{v}^{\perp}$ along the boundary given by unit normal vectors pointing outwards, and hence a trivialization $\tau$ of $\vec{v}^{\perp}$ along the boundary; this then gives a well-defined relative first Chern class $c_1(\xi) = c_1(\vec{v}^{\perp}, \tau) \in H^2(Y, K)$.  

There is a natural projection map $G_K: \rspc{Y,K} \rightarrow \spc{Y}$, given as follows.  
Think of $D^2$ as the unit circle in the complex plane, and view $N(K)$ as $S^1 \times D^2$, where $K$ is $S^1 \times \{0\}$, with direction of increasing angle agreeing with the orientation of $K$.  Also, for $S \subset [0,1]$, let $D_S$ be the set $\{z \in D^2| |z| \in S \}$.  Extend the vector field over $S^1 \times D_{[1/2, 1]}$ so that the vector field points inward nowhere and so that it points in the positive direction on the $S^1$ factor on $S^1 \times D_{\{1/2\}}$.  Then, extend over the rest of $N(K)$ so that the vector field is transverse to the $D^2$ factor on $S^1 \times D_{[0, 1/2]}$, and so that on $K \approx S^1 \times \{0\}$ the vector field traces out a closed orbit whose orientation agrees with that of $K$.    

The projection is $H^2(Y,K)$-equivariant with respect to the natural map $j^*:H^2(Y,K)$ to $H^2(Y)$.  In \cite{OSLI}, it is shown that for $\um{s} \in \rspc{Y,K}$, we have 
\begin{equation}
\label{eq:3}
c_1\big(G_K(\um{s})\big) = j^*\big(c_1(\um{s})\big) + \mbox{PD}[K].
\end{equation}
The fibers of $G_K$ are the orbits of the $\mathbb{Z}\cdot \mbox{PD}[\mu]$ action on $\rspc{Y,K}$, where $\mbox{PD}[\mu]$ is the class in $H^2(Y,K)$ corresponding to the oriented meridian.  Also, just like in the absolute case, we have that for $x \in H^2(Y,K)$
$$ c_1(\xi + x) = c_1(\xi) + 2x .$$

\subsection{Intersection points and relative \sst\ structures}
Consider a doubly-pointed Heegard diagram $(\Sigma, \bsa, \bsd, w, z)$ associated to an oriented knot $K$, which is part of a standard diagram.  Recall the conventions for this.  In $\Sigma$, we draw two arcs connecting $w$ and $z$: a small one $\eta_{\al}$ crossing $\de_g$ once and none of the other $\al$ or $\de$ circles, and another one $\eta_{\de}$ which does not cross any of the $\de$ circles.  The former can be pushed into the $\al$-handlebody, and the latter can be pushed into the $\de$-handlebody.  The union of these two arcs (which intersect at the common boundary of the handlebodies) should then give $K$, and the orientation on $K$ should be such that $\eta_{\al}$ goes from $w$ to $z$.
In particular, if we push a small segment of $\de_g$ near the basepoints into the $\al$-handlebody, we should have a meridian for $K$.  

\vs

\bline{Remark}  We can draw an honest oriented meridian in $\Sigma$ by taking a small counterclockwise circle around $z$.  Confusingly, this appears to be at odds with the fact that, say, $\lambda + 2\mu$ surgery is given by drawing $\ga_g$ with slope $+2$ in a standard diagram, rather than slope $-2$.  This is because when we draw our Heegaard diagram for the surgered manifold, we pretend that the knot passes through the handle attached by $\de_g$, even though according to the above discussion it does not.  If one is bothered by this discrepancy, then he or she can think of the results contained herein as being the result of comparing a doubly-pointed Heegaard 2-tuple with a singly-pointed Heegaard triple with an extra ``reference point" that just happens to look very similar.  

\vs

Analogously to the constructions used for the three-manifold invariants, we can assign a relative \sst\ structure $\um{s}_{w,z}(\mathbf{x})$ to a point $\mathbf{x} \in \mathbb{T}_{\al} \cap \mathbb{T}_{\de}$. Specifically, we take a Morse function compatible with the pointed Heegaard diagram $(\Sigma, \bsa, \bsd, w)$, and take its gradient vector field; we then remove the portions over neighborhoods of the flowlines corresponding to $w$ and to the components of $\mathbf{x}$, replacing them with non-vanishing ones.  We  replace the field over the neighborhood of the flowline for $w$ as depicted in Figure \ref{fig:2}; together with the unaltered portion of the flowline for $z$, this will give a periodic orbit of the vector field which should coincide with the oriented knot $K$.  After a homotopy, we can assume that there is a tubular neighborhood of $K$ over which the vector field is of the form described on $S^1 \times D^2$ in the construction of the projection map $G_K$.  Then $\um{s}_{w,z}(\mathbf{x})$ is the homology class of this vector field restricted to the complement of $S^1 \times D_{[0, 1/2]}$ as in that construction.

\begin{figure}[t!]
\label{fig:2}
\centering \includegraphics[scale=.60]{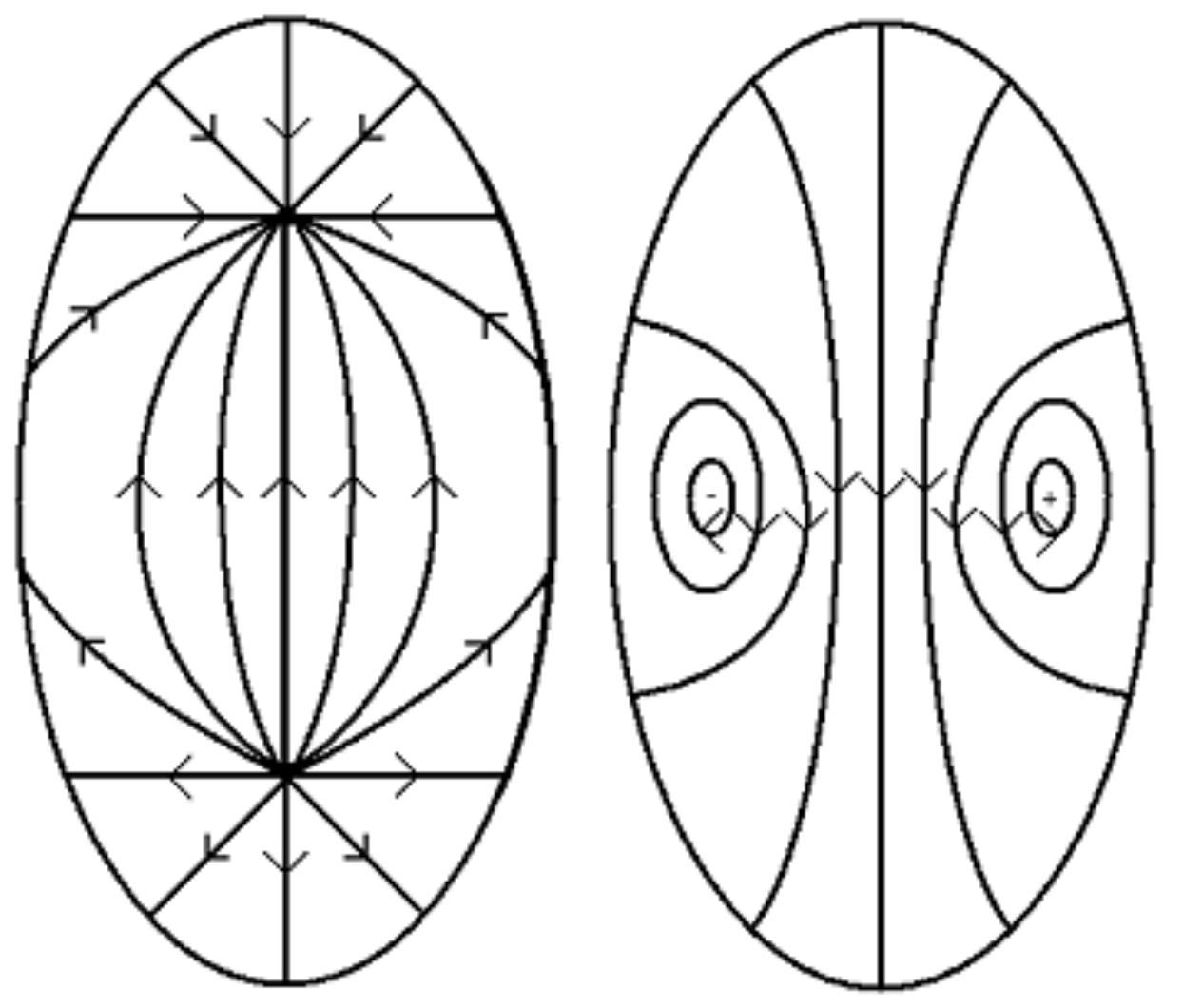}
\caption{The left side (when rotated about a vertical axis) shows the original vector field in a neighborhood of the flowline from the index 0 critical point to the index 3 critical point; the right shows what we replace it with.  At the points marked $+$ and $-$ on the right, the vector field points out of and into the page, respectively.}
\end{figure}

As in the absolute case, we have a straightforward way to calculate $\um{s}_{w,z}(\mathbf{x}_2) - \um{s}_{w,z}(\mathbf{x}_1)$ for two intersection points $\mathbf{x}_1, \mathbf{x}_2$.  We take arcs from $\mathbf{x}_1$ to $\mathbf{x}_2$ along the $\bsa$ circles, and from $\mathbf{x}_2$ to $\mathbf{x}_1$ along the $\bsd$ circles, yielding a closed path in $\Sigma$ and hence a homology class in $H_1(\Sigma \setminus \{w, z\})$.  The difference will then be the image of this homology class in $H_1(Y \setminus K)$, whose Poincar{\'e} dual lives in $H^2\big(Y \setminus K, \partial(Y\setminus K)\big)$.  Then $\um{s}_{w,z}(\mathbf{x}_2) - \um{s}_{w,z}(\mathbf{x}_1)$ is this Poincar{\'e} dual, thought of as an element of $H^2(Y,K)$.  From this, it is shown in \cite{OSRS} that if there is a disk $\phi \in \pi_2(\mathbf{x}, \mathbf{y})$, then
\begin{equation}
\label{eq:4}
\um{s}_{w,z}(\mathbf{y}) - \um{s}_{w,z}(\mathbf{x}) = \big(n_z(\phi) - n_w(\phi)\big)\cdot\mbox{PD}[\mu].
\end{equation}

We also want to be able to evaluate the Chern class of the relative \sst\ structure of an intersection point on an arbitrary class in $H_2\big(Y \setminus N(K), \partial(Y \setminus N(K))\big)$, and so we
will construct geometric representatives of elements of this group.  

To start, consider a standard Heegaard diagram $(\Sigma, \bsa, \bsg, \bsd, w, z)$ for $W'_N$ for any value of $N$. 
To a $\bsa\bsg\bsd$-periodic domain $\mathcal{P}$ such that $\partial \mathcal{P}$ has no components in $\bsg \setminus \{\ga_g\}$, the usual construction of representives of $H_2(W'_N)$ yields an orientation-preserving map $\Phi: S \rightarrow W'_N$ (for some oriented surface $S$ with boundary) of the following form.  There is a disjoint union of disks $D \subset S$, such that $\Phi$ maps each component of $D$ diffeomorphically to the core of the 2-handle attached to $Y$ to form $W'_N$, each component of $\partial D$ into $\ga_g$, and $S \setminus D$ into $Y$.

Choose open balls $B_w$ and $B_z$ around $w$ and $z$ in $\Sigma$, such that $\ga_g$ is tangent to $\partial B_w$ and $\partial B_z$ at precisely one point each.  Then, we may choose our regular neighborhood $N(K)$ carefully so that $B_w$ and $B_z$ are the intersections of $N(K)$ with $\Sigma$, and so that $\ga_g \subset \Sigma$ lies in $\partial N(K)$.  (Roughly, the portion of $\ga_g$ lying between the two tangencies is thought of as being on the top of a tunnel in the $\al$-handlebody, and the rest is thought of as being on the bottom of a tunnel in the $\be$-handlebody.)  So, let $D' \subset S$ be $D$ union with $\Phi^{-1}\big(\{B_w \cup B_z\}\big)$; then, we think of $\mathcal{P}$ as corresponding to the map $\Phi_0:(S\setminus D', \partial(S\setminus D')) \rightarrow \big(Y\setminus N(K), \partial(Y\setminus N(K))\big)$ defined by restricting $\Phi$ to $S \setminus D'$, which gives the desired homology class.

It is easy to see that all classes in $H_2\big(Y \setminus N(K), \partial(Y \setminus N(K))\big)$ can be represented by triply-periodic domains via this construction; and triply periodic domains in turn are in correspondence with classes of $H_2(W'_N)$.  
We can make this identification more explicit as follows.
Write the handle attached to form $W_N$ as $D^2 \times D^2$, where the first factor is the core of the handle.  Let $U$ be the $B \times D^2$, where $B$ is a small neighborhood of the center of the disk, and think of $U$ as sitting in $W'_N$.  Then the pair $(W'_N \setminus U, \partial U)$ is homotopy equivalent to the pair $(Y,K)$; and by excision, the (co)homology of $(W'_N \setminus U, \partial U)$ is the same as that of $(W'_N, U)$.  Thus, for $i \ne 0$, there are canonical isomorphisms $\phi_*: H_i(Y, K) \rightarrow H_i(W'_N)$ and $\phi^*: H^i(W'_N) \rightarrow H^i(Y,K)$.  
Then, it is clear that in fact $\phi_*\big({\Phi_0}_*([S \setminus D'])\big) = \Phi_*([S])$.   

The following largely follows as in Proposition 7.5 of \cite{OSPA}.
Recall that in that paper quantities $\widehat{\chi}(\mathcal{D})$ and $\mu_{\mathbf{y}}(\mathcal{D})$ are defined for two-chains $\mathcal{D}$. 

\begin{prop}
\label{thm:2.3}
Fix some value of $N$, and a standard (doubly-pointed) Heegaard diagram $(\Sigma, \bsa, \bsg, \bsd, w, z)$ for $W'_N$.  Let $\mathbf{y} \in \mathbb{T}_{\al} \cap \mathbb{T}_{\de}$ and $h \in H_2(Y,K)$.  Set $\mathcal{P}$ to be the unique $\bsa\bsg\bsd$-periodic domain with no boundary components in $\bsg \setminus \{\ga_g\}$ (and no local multiplicity at $w$) that represents $\phi_*(h)$.  Then
$$ \langle c_1\big(\um{s}_{w,z}(\mathbf{y})\big), h \rangle = \widehat{\chi}(\mathcal{P}) + 2\mu_{\mathbf{y}}(\mathcal{P}) - n_w(\mathcal{P}) - n_z(\mathcal{P}).$$
\end{prop}

\begin{proof}
Let $v$ be a vector field representing $\um{s}_{w,z}(\mathbf{y})$.  By construction, this vector field should point along $\ga_g$, and $v^{\perp}$ is trivialized along $\ga_g$ by vectors pointing out of $\Sigma$ (after a homotopy). Calling this trivialization $\tau$, we want to calculate $\langle e(v^{\perp}|_{S \setminus D'}, \tau), [S \setminus D', \partial(S \setminus D')] \rangle. $
This proceeds almost identically to the proof of Proposition 7.5 of \cite{OSPA}.  There it is found that 
$$  \langle c_1\big(\m{s}_w(\mathbf{y})\big), h \rangle = \widehat{\chi}(\mathcal{P}) + 2\mu_{\mathbf{y}}(\mathcal{P}) - 2n_w(\mathcal{P})$$
for a $\bsa\bsd$-periodic domain $\mathcal{P}$ representing $h \in H_2(Y)$.  For us, the boundary components of $\mathcal{P}$ along $\ga_g$ don't affect the calculation.  However, we must subtract 1 for each point in $\Phi^{-1}(\{z\})$ and subtract -1 for each point in $\Phi^{-1}(\{w\})$ due to the fact that $v$ is trivialized in a neighborhood of $K$.  Hence, we get 
$\widehat{\chi}(\mathcal{P}) + 2\mu_{\mathbf{y}}(\mathcal{P}) - n_w(\mathcal{P}) - n_z(\mathcal{P}). $
\end{proof}

\subsection{\sst\ structures on cobordisms and relative \sst\ structures}
Given $K$, $N>0$, and $\m{t}_0 \in \spc{Y_0}$, let $\m{S}^N_0(\m{t}_0)$ denote the set $\m{t}_0 + \mathbb{Z}\cdot\mbox{PD}[N\mu]$, thinking of $\mu$ as an element of $H_1(Y_0)$.  In a standard diagram for $K$, consider the set of \sst\ structures on $X_{\al\be\ga}$ that restrict to an element of $\m{S}^N_0(\m{t}_0)$ on $Y_{\al\be}=Y_0$ and to the canonical \sst\ structure on $Y_{\be\ga} = L(N,1)$.  Define $\m{S}_N(\m{t}_0)$ to be the set of restrictions of these structures to $Y_{\al\ga} = Y_N$.  

Also, let $\m{S}_{N\infty}(\m{t}_0) \subset \spc{W_N}$ be the set of structures that restrict to an element of $\m{S}_N(\m{t}_0)$, and let $\m{S}_{\infty}(\m{t}_0) \subset \spc{Y}$ be the restrictions of $\m{S}_{N\infty}(\m{t}_0)$ to $Y$.  

We banish the proof of the following to the Appendix.

\begin{prop}
\label{thm:2.4}
Assume that $N\mu +\lambda$ is not special.  Then the sets $\m{S}_{\infty}(\m{t}_0)$ and $\m{S}_N(\m{t}_0)$ are finite, and independent of the family of diagrams we use.  The former is even independent of $N$: precisely, it is the set of restrictions to $Y$ of structures on $W'_0$ that also restrict to $\m{t}_0$.   Furthermore, if there exist $i, j \in \mathbb{Z}$ such that $ic_1(\m{t}_0) + j\mathrm{PD}[\mu]$ is torsion, then both $\m{S}_{\infty}(\m{t}_0)$ and $\m{S}_N(\m{t}_0)$ consist entirely of torsion \sst\ structures; otherwise, they both consist entirely of non-torsion ones. 
\end{prop}

If $ic_1(\m{t}_0) + j\mathrm{PD}[\mu]$ is torsion for some $i, j \in \mathbb{Z}$, let us call $\m{t}_0$ \emph{$\mu$-torsion}.  Obviously, a torsion structure is $\mu$-torsion.

We wish to speak of the sets $\m{S}_N(\m{t}_0)$ on a common ground for all $N$  -- specifically, we want to identify them all with subsets of $\rspc{Y,K}$.  As a first step toward this end, we make the following definition.  For $\mathbf{x} \in \mathbb{T}_{\al} \cap \mathbb{T}_{\ga}$, $\mathbf{y} \in \mathbb{T}_{\al} \cap \mathbb{T}_{\de}$, and $\psi \in \pi_2(\mathbf{x}, \Theta_{\ga\de}, \mathbf{y})$, define 
\begin{equation}
\label{eq:5}
E_{K,N}(\psi) = \um{s}_{w,z}(\mathbf{y}) + \big(n_w(\psi) - n_z(\psi)\big)\mathrm{PD}[\mu].
\end{equation}
We will show that this in fact gives a well-defined, diagram-independent map 
$$E_{K,N}: \spc{W'_N} \rightarrow \rspc{Y,K}.$$

Recall the definition of the spider number $\sigma(\psi, \mathcal{P})$.  Orient the $\bsa, \bsg$ and $\bsd$ circles so that every circle appears in $\partial\mathcal{P}$ with nonnegative multiplicity. 
Let $\partial'_{\al}\mathcal{P}$ be gotten by taking an inward translate of each $\bsa$ circle with respected to the endowed orientation, and then taking a linear combination of these circles with multiplicities given by the corresponding multiplicities in $\partial\mathcal{P}$.

Think of $\psi \in \pi_2(\mathbf{x}, \mathbf{y}, \mathbf{w})$ as a map from $\Delta$ to $\Sigma$, where $\Delta$ is a triangle depicted as in Figure \myfig{3}.  A spider is a point $u$ in $\Delta$ together with three segments $a, b$ and $c$ from $u$ to the $\al$, $\ga$ and $\de$ portions of the boundary of $\Delta$, respectively, each oriented outward from $u$.  Then set
$$\sigma(\psi, \mathcal{P}) = n_{\psi(u)}(\mathcal{P}) + \#\big(\partial'_{\al}\mathcal{P} \cap \psi(a)\big) + \#\big(\partial'_{\ga}\mathcal{P} \cap \psi(b)\big) + \#\big(\partial'_{\de}\mathcal{P} \cap \psi(c)\big),$$
where all the intersection numbers are oriented (with $x$-axis $\cap$ $y$-axis = 1).

\begin{figure}[t!]
\label{fig:3}
\centering \includegraphics[scale=.65]{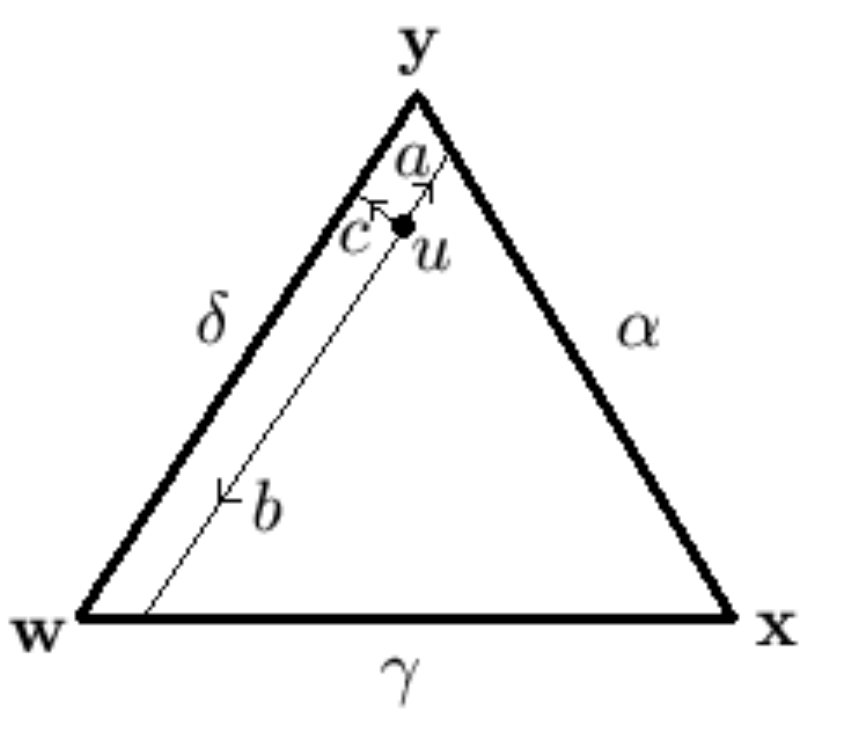}
\caption{A triangle $\Delta$, with sides and vertices marked by where their images lie under a map $\psi$, and a spider.  We will think of our spiders as having the point near $\mathbf{x}$, with two short legs $a$ and $c$ and one long one $b$ running parallel to $\delta$.}
\end{figure}

\begin{prop}
\label{thm:2.5}
Given a knot $K$, choose $N$ such that $N\mu + \la$ is not a special longitude, and form a standard diagram.  Let $h \in H_2(Y,K)$, and let $\mathcal{P}$ be the unique $\bsa\bsg\bsd$-periodic domain with no boundary components in $\bsg \setminus \{\ga_g\}$ (and no local multiplicity at $w$) that represents $\phi_*(h)$.  Then, writing $\la_N$ for $N\mu + \la \in H_1(Y \setminus K)$, we have
$$\langle c_1\big(\m{s}_w(\psi)\big), \phi_*(h) \rangle = \langle c_1\big(E_{K,N}(\psi)\big) + \mathrm{PD}[\la_N - \mu], h \rangle.$$
\end{prop}

\begin{proof}
The Chern class formula from Section 6 of \cite{OS4M} gives
\begin{equation}
\label{eq:6}
\langle c_1\big(\m{s}_w(\psi)\big), \phi_*(h) \rangle = \widehat{\chi}(\mathcal{P}) + \#\partial\mathcal{P} - 2n_w(\mathcal{P}) + 2\sigma(\psi, \mathcal{P}).
\end{equation}
Let us calculate 
$$\langle c_1\big(\m{s}_w(\psi)\big), \phi_*(h) \rangle - \langle c_1\big(\um{s}_{w,z}(\mathbf{y})\big), h \rangle = \#\partial\mathcal{P} - n_w(\mathcal{P}) + n_z(\mathcal{P})
 + 2\sigma(\psi, \mathcal{P}) - 2\mu_{\mathbf{y}}(\mathcal{P}),$$
with $\mathbf{y}$ a corner of $\psi$ in $\mathbb{T}_{\al} \cap \mathbb{T}_{\de}$ as above.

For the calculation of $\sigma$, choose the spider in our triangle $\psi$ as depicted in Figure \myfig{3}, with the spider point near the intersection of the $\alpha$ and $\delta$ edges, and the leg $b$ running parallel and close to the $\delta$ edge.  Then it is easy to check that 
$$2\mu_{\mathbf{y}}(\mathcal{P}) = \#\partial\mathcal{P} - \#\partial_{\ga}\mathcal{P} + 
2\sigma(\psi, \mathcal{P}) - 2(\#\partial'_{\ga}\mathcal{P} \cap b)$$
where $\#\partial_{\ga}\mathcal{P}$ is the number of $\ga$ circles in $\partial\mathcal{P}$.  So, we need only calculate $$\#\partial_{\ga}\mathcal{P} + 2(\#\partial'_{\ga}\mathcal{P} \cap b) - n_w(\mathcal{P}) + n_z(\mathcal{P}).$$   

Suppose that $\partial\mathcal{P} = k\ga_g + L$, where $L$ is a linear combination of circles in $\bsa$ and $\bsd$, and we break our convention and orient $\ga_g$ to agree with the orientation of $K$.  Assume for now that $k\geq 0$.  Then examination of a standard diagram shows that $\#\partial_{\ga}\mathcal{P} = k$ and $2(\#\partial'_{\ga}\mathcal{P} \cap b) = -2k(n_z(\psi)-n_w(\psi) + 1)$.  To see the latter, note that we assume that the only $\bsg$ boundary component of $\mathcal{P}$ is $\ga_g$, and it is easy to see that if a component of the image of $b$ intersects $\ga_g$, it must run parallel to $\de_g$, and we just count the number of times it circles the meridian.  

It is not hard to see from our construction of a surface representing the homology class $h$ that $k = \langle \mbox{PD}[\mu], h\rangle$ and that $n_z(\mathcal{P}) - n_w(\mathcal{P}) = \langle \mbox{PD}[\la_N], h\rangle$.  So, we conclude that 
$$\langle c_1\big(\m{s}_w(\psi)\big), \phi_*(h) \rangle = \langle c_1\Big(\um{s}_{w,z}(\mathbf{y}) + \big(n_w(\psi) - n_z(\psi)\big)\mbox{PD}[\mu]\Big) + \mbox{PD}\big[\la_N - \mu], h \rangle. $$
The case where $k<0$ is similar, and we get the same class. 
\end{proof}

As a map from $\spc{W'_N}$ to $\rspc{Y,K}$, let $E_{K,N}$ be defined as follows.  We claim that given $\m{s} \in \spc{W'_N}$, there is a unique element $\xi \in \rspc{Y,K}$ such that $G_K(\xi) = 
\m{s}|_Y$ and such that $c_1(\xi) = \phi^*\big(c_1(\m{s})\big) - \mbox{PD}[\la_N - \mu]$; then let $E_{K,N}(\m{s}) = \xi$.

\begin{prop}
\label{thm:2.6}
The map described above is indeed well-defined, and $E_{K,N}\big(\m{s}_w(\psi)\big) = E_{K,N}(\psi)$ for any triangle $\psi$ in a standard diagram, where the right hand side uses the definition of $E_{K,N}$ as given in Equation \ref{eq:5}.  This map is $H^2(W'_N)$-equivariant, where the action on the right is via the map $\phi^*$.
\end{prop}

\begin{proof}
First note that the condition that $G_K(\xi) = \m{s}|_Y$ determines $\xi$ up to adding some multiple of $\mbox{PD}[\mu]$.  Examining the long exact sequence
$$ H^1(Y) \rightarrow H^1(K) \rightarrow H^2(Y,K) \overset{j^*}{\rightarrow} H^2(Y),$$
we see that the first arrow is the zero map, since $K$ is rationally nullhomologous; hence, the second map is an injection of $H^1(K) \cong \mathbb{Z}$ onto $\mbox{Ker}(j^*) \subset H^2(Y,K)$.  Thus, there is at most one $\xi$ that satisfies the first conditions and the demand that $c_1(\xi) - \phi^*\big(c_1(\m{s}) \big) + \mbox{PD}[\la_N - \mu]$ is torsion.  
Assuming existance of such a $\xi$, then Equation \ref{eq:3} and the fact that $j^*(\mbox{PD}[\la_N- \mu]) = \mbox{PD}[K]$ then imply that this class is actually 0. 

Proposition \ref{thm:2.5} ensures that such a $\xi$ does exist if there is a triangle $\psi$ representing $\m{s}$ in some standard diagram.  It is easy to see that if $\xi$ works for $\m{s}$, then $\xi + \phi^*(x)$ works for $\m{s} + x$, which establishes the map as well-defined on all of $\spc{W'_N}$ as well as the $H^2(W'_N)$-equivariance. 
\end{proof}

\subsection{Squares of Chern classes}
The intersection form on the cobordism $W_N$ is defined via the cup product $H^2(W_N, \partial W_N) \otimes H^2(W_N, \partial W_N) \rightarrow H^4(W_N, \partial W_N) \cong \mathbb{Z}$, where the latter isomorphism is evaluation on a fundamental class.  If $j$ denotes the composition of maps $H^2(W_N; \mathbb{Z}) \rightarrow H^2(\partial W_N; \mathbb{Z}) \rightarrow H^2(\partial W_N;\mathbb{Q})$, then the intersection form can be extended to a pairing $\mbox{Ker }j \otimes \mbox{Ker }j \rightarrow \mathbb{Q}$ as follows.  If $x \in \mbox{Ker }j$, then $x \otimes 1 \in H^2(W_N; \mathbb{Z}) \otimes \mathbb{Q}$ goes to an element in $H^2(W_N; \mathbb{Q})$ which clearly lifts to some $\widetilde{x} \in H^2(W_N, \partial W_N; \mathbb{Q})$.  The square of $\widetilde{x}$ is easily seen to be independent of the choice of lift $\widetilde{x}$, and so we define $x^2 = \widetilde{x}^2$.  

We can also square elements of $H_2(W_N)$ in the usual manner.  The group $H_2(W_N, Y)$ is isomorphic to $\mathbb{Z}$, generated by the oriented core $F$ of the attached handle, oriented to agree with $-K$ on the boundary; $H_2(W_N)$ splits (non-canonically) as the direct sum of $H_2(Y)$ and $\mathbb{Z}\cdot [\widetilde{dF}]$, where $[\widetilde{dF}]$ is gotten by taking $d$ copies of $F$ (recalling that $d$ is the order of $K$ in $H_1(Y)$) and then capping off the boundaries with a Seifert surface for $dK$.  

We write $F'$ or $[\widetilde{dF'}]$ when thinking of the above as elements of $H_2(W'_N, Y)$ or $H_2(W'_N)$; this distinction only matters when we consider intersection forms on cobordisms.
Note that for any rationally nullhomologous knot $K$ equipped with longitude $\la$, there are unique relatively prime integers $p$ and $q$ with $p > 0$ such that surgery on the framing $p\mu + q\la$ increase the first Betti number; let $\kappa = -\frac{p}{q}$.  In particular, if $K$ is special, then $\kappa = 0$.
The proof of the following is also given in the Appendix.

\begin{prop}
\label{thm:2.7}
The order of $\mu$ in $H_1(Y_0)$ is $|d\kappa |$, and for any lift $[\widetilde{dF'}]$ of $F'$, we have that 
$$\frac{[\widetilde{dF'}]^2}{d^2} = -\kappa - N.$$
In particular, for special $K$, $\mu$ is not torsion in $H_1(Y_0)$ and the square of 
$[\widetilde{dF'}]$ is $-d^2N$.
\end{prop}

We restrict attention now to \sst\ structures on $W'_N$ that restrict to torsion or (if $N\mu + \la$ is special) $\mu$-torsion structures on the boundary; we refer to these \sst\ structures, as well as the relative structures on $(Y, K)$ over them, as \emph{boundary-torsion}.  Our interest in these stems from the fact that the generators of $\underline{HF}^+(Y)$ and $\underline{HF}^+(Y_N)$ (for non-special $N\mu + \la$) that represent them will inherit absolute $\mathbb{Q}$-gradings from their untwisted counterparts.  

For torsion $\m{t}_N \in \mbox{Spin}^{\mbox{\scriptsize{c}}}(Y_N)$, choose some fixed $\m{s}_0 \in \mbox{Spin}^{\mbox{\scriptsize{c}}}(W'_N)$ that restricts to $\m{t}_N$.  Then the subset of elements in $\mbox{Spin}^{\mbox{\scriptsize{c}}}(W'_N)$ that restrict to $\m{t}_N$ will be $\m{s}_0 + \mathbb{Z}\cdot\big(\mbox{PD}[F']|_{W'_N}\big)$. (Here, $\mbox{PD}[F'] \in H^2(W'_N, Y_N)$ is the Poincar{\'e} dual of $F'$ thought of as an element of $H_2(W'_N, Y)$.)  For boundary-torsion $\m{s}$, the evaluation of $c_1(\m{s})$ on any element of the kernel of the map from $H_2(W'_N)$ to $H_2(W'_N, Y)$ vanishes; thus, the quantity $\frac{\langle c_1(\m{s}), [\widetilde{dF'}] \rangle}{d}$ is the same for any choice of $[\widetilde{dF'}]$ in the preimage of $d[F']$.
Define  
$$ Q_K(x; \m{s}_0) = x^2\frac{[\widetilde{dF'}]^2}{d^2} + x\left(\frac{\langle c_1(\m{s}_0), [\widetilde{dF'}] \rangle}{d} \right) = \frac{c_1^2(\m{s}_0 + x\mbox{PD}[F']|_{W'_N}) - c_1^2(\m{s}_0)}{4}.$$
For a fixed $\m{s}_0$ which restricts to $\m{t}_N$, we can interpret $Q_K(x; \m{s}_0)$ as a quadratic function of $x \in \mathbb{Q}$; let $x_* \in \mathbb{Q}$ be the value at which this function achieves its maximum, and set $x_0 = \lfloor x_* \rfloor$.
We write $\m{s}_{K+}(\m{t}_N) = \m{s}_0 + x_0\mbox{PD}[F']|_{W'_N}$ and $\m{s}_{K-}(\m{t}_N) = \m{s}_0 + (x_0 + 1)\mbox{PD}[F']|_{W'_N}$.  These structures depend on $K$, $N$ and $\m{t}_N$, but clearly not on $\m{s}_0$, since they are simply the maximizers of the square of the Chern class among those \sst\ structures which restrict to $\m{t}_N$.

We are interested also in the quantity 
\begin{equation}
\label{eq:7}
q_K(\xi) = Q_K\big(1; E^{-1}_{K, N}(\xi)\big) = \frac{[\widetilde{dF'}]^2}{d^2} + \left(\frac{\langle c_1\left(E^{-1}_{K, N}(\xi)\right), [\widetilde{dF'}] \rangle}{d} \right)
\end{equation} 
for $\xi \in \underline{\mbox{Spin}}^{\mbox{\scriptsize{c}}}(Y,K)$.  
We also abuse notation and sometimes write $q_K$ for $q_K \circ E_{K,N} = Q_K(1; \cdot)$, which makes sense for any boundary-torsion element of $\spc{W'_N}$ for any $N$.  
In light of the following, we sometimes write a relative \sst\ structure $\xi$ in the form $[G_K(\xi), q_K(\xi)]$.

\begin{prop}
\label{thm:2.8}
The quantity $q_K(\xi)$ is independent of $N$, for $\xi \in \underline{\mbox{Spin}}^{\mbox{\scriptsize{c}}}(Y,K)$, and the map $\xi \mapsto \big(G_K(\xi), q_K(\xi)\big)$ is injective.
\end{prop}

We defer the proof to Section 3.

We would hope that the above makes the set of relative \sst\ structures corresponding to $\m{S}_{N\infty}(\m{t}_0)$ independent of $N$.  This is not the case, but when $K$ is special we can say something that will be of use to us later.

\begin{prop}
\label{thm:2.9}
Fix some oriented knot $K$ and some $\mu$-torsion $\m{t}_0 \in \spc{Y_0}.$ If $K$ is special, we have 
$$ q_K\big(\m{s}_{K+}(\m{t}_N)\big) = -\frac{\langle c_1(\m{t}_0), [\widehat{dS}] \rangle}{d} $$
for any $\m{t}_N$ in $\m{S}_N(\m{t}_0)$, where $[\widehat{dS}]$ is any Seifert surface for $dK$ capped off in $Y_0$; and
$$ E_{K, N}\circ \m{s}_{K+}\big(\{\m{S}_N(\m{t}_0)\}\big) = \big\{ [\m{s}_{K+}(\m{t}_N)|_Y - i\mathrm{PD}[K],   -\frac{\langle c_1(\m{t}_0), [\widehat{dS}] \rangle}{d}] \big| 0 \leq i < d \big\}. $$
\end{prop}

We give the proof in the Appendix.

\section{Families of Standard Heegaard diagrams and Small Triangles} 
In this section, we make a convenient generalization of the standard Heegaard diagrams, and note some facts that will be useful in proving Theorem \ref{thm:5.2}, a twisted version of the large-$N$ surgery formula that is the first step towards the more specific formula of Theorem \ref{thm:6.1}.

Note that the winding region of a standard diagram, as depicted in Figure \myfig{1}, has a fairly rigid form: essentially, the picture is determined by the number of horizontal $\bsa$ strands, the number of turns in $\ga_g$, and the location of $\de_g$ (along with the flanking basepoints).  Given a particular standard diagram, we can then cut out the winding region, and replace it with any other possible winding region that has the same number of $\bsa$ strands.  Such a replacement has little effect on the manifolds $Y_0$ and $Y$ that the diagram represents, or the intersection points of $\mathbb{T}_{\al} \cap \mathbb{T}_{\be}$ or $\mathbb{T}_{\al} \cap \mathbb{T}_{\de}$.  The new diagram will represent $Y_N$ and $W'_N$ for new values of $N$, of course.

Likewise, the small triangles appearing in a standard diagram have a rigid form; specifically, they are determined by their corner $\mathbf{y} \in \mathbb{T}_{\al} \cap \mathbb{T}_{\de}$ and the value of $n_w(\psi) - n_z(\psi)$.

\begin{defi}
\label{thm:3.1}
A \emph{family} of special Heegaard diagrams for $(Y,K)$ is the set of all diagrams gotten from a given one by making this type of replacement.  
\end{defi}

If two small triangles appearing in two Heegaard diagrams in the same family have the same corner $\mathbf{y} \in \mathbb{T}_{\al} \cap \mathbb{T}_{\de}$ and the same value of $n_w(\psi) - n_z(\psi)$, we call them \emph{similar}. 

\vs

\noindent \textit{Proof of Proposition \ref{thm:2.8}.} We show that raising $N$ by 1 doesn't change $q_K(\xi)$.  When $N$ is increased by 1, the quantity $\frac{[\widetilde{dF'}]^2}{d^2}$ decreases by 1 by Proposition \ref{thm:2.7}. 

As for $\frac{\langle c_1\big(E^{-1}_{K, N}(\xi)\big), [\widetilde{dF'}] \rangle}{d}$, we appeal to the Chern class evaluation formula, Equation \ref{eq:6}.  Any lift $[\widetilde{dF'}]$ in $H_2(W'_N)$ has a representative periodic domain in a standard diagram whose boundary is of the form $d(\ga_g - N\de_g) + L$, where $L$ is a linear combination of $\al$ circles and $\ga_1, \ldots, \ga_{g-1}$.  If we take a diagram for $W'_{N+1}$ from the same family, we will have a periodic domain with boundary $d(\ga_g - (N+1)\de_g) + L$ for the same value of $L$, which will also represent some class $[\widetilde{dF'}] \in H_2(W'_{N+1})$.
Likewise, for some $\xi$, we have similar triangles in the two diagrams that respectively represent $E^{-1}_{K, N}(\xi)$ and $E^{-1}_{K, N+1}(\xi)$.  So, we can calculate the difference using the Chern class formula; we find that $\langle c_1\big(E^{-1}_{K, N+1}(\xi)\big), [\widetilde{dF'}] \rangle - \langle c_1\big(E^{-1}_{K, N}(\xi)\big), [\widetilde{dF'}] \rangle = -d$ for such values of $\xi$.  By the characterization of $E_{K,N}(\xi)$ given in Proposition \ref{thm:2.6}, it is then clear that this equation must hold for all $\xi$.  

To see that $\xi$ is determined by $G_K(\xi)$ and $q_K(\xi)$, simply note that $q_K(\xi + i\mbox{PD}[\mu]) = q_K(\xi) + 2i$. \halbox

\vs

Fix a family $\mathcal{F}$ of standard diagrams $\big\{(\Sigma, \bsa, \bsg, \bsd, w, z)\big\}$ for $(Y,K)$.
If we forget about the basepoints, we can still talk about equivalence classes of intersection points in $\mathbb{T}_{\al} \cap \mathbb{T}_{\ga}$: two intersection points are equivalent if and only if, upon adding basepoints, they represent the same \sst\ structure.  Of course, the intersection points in 
$\mathbb{T}_{\al} \cap \mathbb{T}_{\de}$ for any two of these diagrams are the same.  Equivalence classes of points still form an affine set for the action of $H_1(Y_N)$.

Over all the diagrams in $\mathcal{F}$, the total number of intersection points not supported in the winding region is constant, since the components of these points lie in the portions of the diagram that don't change throughout the family.  Hence, the number of equivalence classes that contain points not supported in the winding region is bounded independent of $N$.  

\begin{lemma}
\label{thm:3.2}
There is an integer $\epsilon > 0$ such that if $N$ is sufficiently large, then for all $\m{t} \in \spc{Y_N}$, there is a diagram in $\mathcal{F}$ such that $\m{t}$ is represented by an equivalence class of points which are all supported in the winding region, and such that we have 
$$2|n_w(\psi) - n_z(\psi)| \leq N + \epsilon$$ 
for all small $\psi$ with a corner representing $\m{t}$. 
\end{lemma}

\begin{proof}
We assume as always that $K$ is rationally nullhomologous in $Y$, say of order $d$.  Then, it is not hard to show that there is a constant $c$ such that $\mu$ will be of order $|dN - c|$ in $H_1(Y_N)$. 

There are $N + 1$ members of $\mathcal{F}$ that represent $W'_N$, corresponding to the $N + 1$ placements of $\de_g$ and accompanying placements of the basepoint $w$.  Consider the $\epsilon$ innermost placements of $\de_g$ -- that is, disregard the $\frac{N - \epsilon}{2}$ leftmost and rightmost placements.  We claim that if $N$ is large enough, then amongst these $\epsilon$ different placements, a given \sst\ structure $\m{t}$ will be represented by 
$\epsilon$ distinct equivalence classes of intersection points in $\mathbb{T}_{\al} \cap \mathbb{T}_{\ga}$.  To see this, suppose we have two adjacent placements of $\de_g$, with corresponding basepoints $w_1$ and $w_2$; then if $e$ is the equivalence class of $\m{t}$ with respect to $w_1$, then $e + \mbox{PD}[\mu]$ is the equivalence class with respect to $w_2$.  Then for large enough $N$, $|dN-c| \geq \epsilon$, so that the $\epsilon$ equivalence classes will therefore all be distinct.

Hence, if $\epsilon$ is larger than the number of equivalence classes that contain points not supported in the winding region, then with respect to one of these placements of $\de_g$, $\m{t}$ must be represented by an equivalence class of points which all are supported in the winding region.  Furthermore, for these placements, it is easy to see that 
$$|n_w(\psi) - n_z(\psi)| \leq \epsilon + \frac{N-\epsilon}{2}$$
(since for small triangles, at most one of $n_w(\psi)$ and $n_z(\psi)$ is nonzero),  
from which the claim follows. 
\end{proof}

Call a value of $\epsilon$ valid if it makes Lemma \ref{thm:3.2} hold.  Given $\m{t} \in \spc{Y_N}$, we call a standard diagram $\m{t}$\emph{-proper} if it is of the type described in the statement of Lemma \ref{thm:3.2} with respect to some valid $\epsilon$.

\begin{lemma}
\label{thm:3.3}
For any small triangle $\psi$ in \emph{any} diagram in $\mathcal{F}$, the quantity 
$$|q_K\big(\m{s}_w(\psi)\big) - 2(n_w(\psi) - n_z(\psi))|$$ 
is bounded independent of the particular diagram. In particular, there exists a constant $C_q$ depending only on $\mathcal{F}$ such that 
$$|q_K\big(\m{s}_w(\psi)\big)| \leq C_q + N$$
for any small triangle $\psi$ in a $\m{t}$-proper standard diagram in $\mathcal{F}$ representing $W'_N$, where $\m{t} = \m{s}_w(\psi)|_{Y_N}$.
\end{lemma}

\begin{proof}
Fix a point $\mathbf{y} \in \mathbb{T}_{\al} \cap \mathbb{T}_{\de}$, and consider some small triangle $\psi$ with $\mathbf{y}$ as a corner in a diagram in $\mathcal{F}$.  If we choose a different small triangle $\psi'$ in the same diagram, with $\mathbf{y}$ still a corner but such that $n_w(\psi') - n_z(\psi')$ increased by 1, then we can use the Chern class formula to show that $\langle c_1\big(\m{s}_w(\psi)\big), [\widetilde{dF'}] \rangle$ increases by $2d$, so that $q\big(\m{s}_w(\psi)\big)$ is increased by 2.  On the other hand, if we add a turn to $\ga_g$ and consider the similar triangle $\psi''$ to $\psi$ in this diagram, Proposition \ref{thm:2.8} shows that $q_K\big(\m{s}_w(\psi)\big)$ doesn't change.  Thus, $|q_K\big(\m{s}_w(\psi)\big) - 2 (n_w(\psi) - n_z(\psi))|$ depends only on $\mathbf{y}$; since the number of points in $\mathbb{T}_{\al} \cap \mathbb{T}_{\de}$ is fixed, the first claim follows.  

The second claim follows from the first together with Lemma \ref{thm:3.2}. 
\end{proof}

\begin{prop}
\label{thm:3.4}
For a family $\mathcal{F}$ of standard diagrams for $K$, there is a constant $N(\mathcal{F})$ such that the following holds.
Take any $\m{t} \in \spc{Y_N}$, and let $\psi$ be some small triangle with a corner representing $\m{t}$, in a standard $\m{t}$-proper diagram for $W'_N$.  Then if $N \geq N(\mathcal{F})$, we have 
$$\m{s}_w(\psi) = \m{s}_{K+}(\m{t})$$
and
$$\m{s}_z(\psi) = \m{s}_{K-}(\m{t}).$$
\end{prop}

\begin{proof}
First, note that $\m{s}_z(\psi) - \m{s}_w(\psi) = \mbox{PD}[F']|_{W'_N}$. To see this, note that the Chern class formula gives that for $\mathcal{P}$ representing $h \in H_2(W'_N)$, $\langle c_1\big(\m{s}_z(\psi)\big) - c_1\big(\m{s}_w(\psi)\big), h \rangle = 2n_w(\mathcal{P}) - 2n_z(\mathcal{P})$.  Thus, we can easily show that $c_1\big(\m{s}_z(\psi)\big) - c_1\big(\m{s}_w(\psi)\big) - 2\mbox{PD}[F']|_{W'_N}$ is trivial as an element of $\mbox{Hom}(H_2(W'_N), \mathbb{Z})$, hence torsion in $H^2(W'_N)$, hence $0$ since $\m{s}_z(\psi)|_{Y_N} =  \m{s}_w(\psi)|_{Y_N}$ and the torsion subgroup of $H^2(W'_N)$ injects into $H^2(Y_N)$.

We claim that if $N$ is large enough, we must have
$$c_1^2\big(\m{s}_w(\psi)\big) \geq c_1^2\big(\m{s}_w(\psi) - \mbox{PD}[F']|_{W'_N}\big)$$
and 
$$c_1^2\big(\m{s}_w(\psi) + \mbox{PD}[F']|_{W'_N}\big) \geq c_1^2\big(\m{s}_w(\psi) + 2\mbox{PD}[F']|_{W'_N}\big)$$
for $\psi$ having a corner representing any $\m{t} \in \spc{Y_N}$.  Assuming this, the fact that $c_1^2\big(\m{s}_w(\psi) + x\mbox{PD}[F']|_{W'_N}\big)$ depends quadratically on $x$ means that this function is greater at $x=0$ and $x=1$ than at any other values of $x$, from which it follows that $\m{s}_w(\psi) = \m{s}_{K+}(\m{t})$ and $\m{s}_z(\psi) = \m{s}_w(\psi) + \mbox{PD}[F']|_{W'_N} = \m{s}_{K-}(\m{t})$.  

To show the claim, note that the two inequalities above are respectively equivalent to
$$ Q_K\big(0; \m{s}_w(\psi_{\m{t}})\big) \geq Q_K\big(-1; \m{s}_w(\psi_{\m{t}})\big)$$ 
and
$$ Q_K\big(1; \m{s}_w(\psi_{\m{t}})\big) \geq Q_K\big(2; \m{s}_w(\psi_{\m{t}})\big).$$
In turn, these can be reduced further, to 
\begin{equation}
\label{eq:8}
q_K\big(\m{s}_w(\psi)\big) - 2\frac{[\widetilde{dF'}]^2}{d^2} \geq 0 
\end{equation}
and
\begin{equation}
\label{eq:9}
q_K\big(\m{s}_w(\psi)\big) + 2\frac{[\widetilde{dF'}]^2}{d^2} \leq 0. 
\end{equation}
 
Lemma \ref{thm:3.3} says that there is a number $C_q$ depending only on $\mathcal{F}$ such that 
$$ -C_q - N \leq q_K\big(\m{s}_w(\psi)\big) \leq C_q + N $$
always holds for all small triangles $\psi$.  By Proposition \ref{thm:2.8}, $2\frac{[\widetilde{dF'}]^2}{d^2}$ will be equal to $-2\kappa - 2N$.  So we have 
$$ q_K\big(\m{s}_w(\psi)\big) - 2\frac{[\widetilde{dF'}]^2}{d^2} \geq 2\kappa - C_q + N, $$
$$ q_K\big(\m{s}_w(\psi)\big) + 2\frac{[\widetilde{dF'}]^2}{d^2} \leq -2\kappa + C_q - N; $$
for large enough $N$, these imply that inequalities (\ref{eq:8}) and (\ref{eq:9}) hold, which proves the claim. \end{proof}

\section{Twisted Coefficients and a Long Exact Sequence}

We wish to prove a twisted analogue of the surgery long exact sequence relating $Y$, $Y_0$, and $Y_N$ when $K$ is a special knot.  To accomplish this, we introduce a system of coefficients adapted to a given standard diagram.

\subsection{Additive functions on polygons}
Suppose that we have a standard Heegaard diagram with translates $\mathcal{H} = (\Sigma, \{\bse{i}\}, w)$ for $W'_N$ with $N>0$, recalling that we denote the tuples $\bsa, \bsb, \bsg$ and $\bsd$ associated to a standard diagram, as well as their translates, by $\bse{i}$ for $i \geq 0$.  Again, we assume that $K$ is special.

To properly relate the twisted Floer homologies of the various manifolds $Y_{\eta^i\eta^j}$, we find it useful to have the following.  Let $C(\mathcal{H})$ be the free abelian group generated by the set $\bigcup_{i\geq 0} \{\eta^i_1, \ldots, \eta^i_g\}$.  Define $L(\mathcal{H})$ to be the quotient of $C(\mathcal{H})$ by the equivalence relation $\sim$, where $\sim$ is generated by $\eta^i_k \sim \eta^{j}_k$ when $i, j \ne 0$ and $k \ne g$;   $\eta^i_g \sim \eta^{j}_g$ when $i, j \ne 0$ and $i \equiv j$ mod 3; and $\ga_g \sim \be_g + N\de_g$.  Informally, we are identifying elements that are \emph{a priori} equivalent in $H_1(\Sigma)$.  

Then, let $K(\mathcal{H})$ be the kernel of the obvious homomorphism from $L(\mathcal{H})$ to $H_1(\Sigma)$.

Denote by $\Delta(\mathcal{H})$ the set of all homotopy classes of polygons in this diagram, i.e., the disjoint union of $\pi_2(\mathbf{x}_1, \ldots, \mathbf{x}_k)$ for $k \geq 2$, over all tuples of points such that $\pi_2(\mathbf{x}_1, \ldots, \mathbf{x}_k)$ makes sense.  We wish to define a function $A_K: \Delta(\mathcal{H}) \rightarrow K(\mathcal{H})$ that is additive under splicing, and has appropriate equivariance properties (which we make precise later).

To begin the construction, we first make the following choices:
\begin{itemize}

\item points $\mathbf{p}_i \in  \mathbb{T}_{\eta^0} \bigcap \mathbb{T}_{\eta^i}$ for each $i$;

\item for each $\mathbf{x} \in \mathbb{T}_{\eta^i} \bigcap \mathbb{T}_{\eta^j}$, an oriented multiarc $q_i(\mathbf{x})$ from $\mathbf{x}$ to $\mathbf{p}_i$ along $\boldsymbol{\eta}^i$ (and similarly a multiarc $q_j(\mathbf{x})$); and

\item multiarcs $m_i$ from $\mathbf{p}_i$ to $\mathbf{p}_0$ along $\bse{0}$, letting $m_0$ be the trivial multiarc.

\end{itemize}

\noindent We choose the multiarcs $q_i(\mathbf{x})$ so that if $\mathbf{x} \in \mathbb{T}_{\eta^i} \bigcap \mathbb{T}_{\eta^j}$ and $\mathbf{x}' \in \mathbb{T}_{\eta^{i'}} \bigcap \mathbb{T}_{\eta^{j'}}$ with $i \equiv i'$ and $j \equiv j'$ mod 3 are corresponding points, then $q_i(\mathbf{x})$ and $q_{i'}(\mathbf{x}')$ are corresponding multiarcs.

For any point $\mathbf{x} \in \mathbb{T}_{\eta^i} \bigcap \mathbb{T}_{\eta^j}$ with $j>i$, define $\ell_0(\mathbf{x}) = q_{j}(\mathbf{x}) - q_{i}(\mathbf{x})   + m_{j} - m_{i}.$
Note that this realizes a closed oriented multiarc in $\Sigma$, and if $\mathbf{x}_1, \ldots, \mathbf{x}_k$ are such that $\mathbf{x}_k \in \mathbb{T}_{\eta^i_k} \bigcap \mathbb{T}_{\eta^i_{k+1}}$ with $i_{k+1} = i_1 < i_2 < \ldots < i_k$, then the sum $\ell_0(\mathbf{x}_1) + \ldots + \ell_0(\mathbf{x}_k) - \ell_0(\mathbf{x}_{k+1})$ is homotopic (within the circles of $\mathcal{H}$) to a multiarc supported along circles in $\boldsymbol{\eta}^{i_n}$ for $n = 1, \ldots, k$. 

Let $L(i, j)$ denote $\ell_0(\Theta_{i,j})$.  It is not hard to see for any point $\mathbf{x} \in \mathbb{T}_{\eta^i} \cap \mathbb{T}_{\eta^j}$ with $j > i > 0$, that $L(i, i+1) + L(i+1, i+2) + \ldots + L(j-1,j) - \ell_0(\mathbf{x})$ is homologous in $H_1(\Sigma)$ to some element $\ell_1(\mathbf{x}) \in C(\mathcal{H})$.  Then, define $\ell(\mathbf{x})$ to be the closed oriented multiarc $\ell_0(\mathbf{x}) - \ell_1(\mathbf{x})$. 

Say that two points $\mathbf{x} \in \mathbb{T}_{\al} \cap \mathbb{T}_{\eta^i}$ and $\mathbf{y} \in \mathbb{T}_{\al} \cap \mathbb{T}_{\eta^{i+k}}$ are \emph{homologous} if \linebreak $\pi_2(\mathbf{x}, \Theta_{i,i+1}, \ldots, \Theta_{i+k-1, i+k}, \mathbf{y})$ is nonempty, and extend this to an equivalence relation on $\bigcup_{i} \mathbb{T}_{\al} \cap \mathbb{T}_{\eta^i}$.  Of course, generators that are \sst -equivalent will be homologous.
For each such homology class $c$, choose a representative point $\mathbf{x}_c \in \mathbb{T}_{\al} \cap \mathbb{T}_{\eta^1}$.  Then, it is also not hard to see for any other point $\mathbf{x} \in \mathbb{T}_{\al} \cap \mathbb{T}_{\eta^i}$ for any $i$ representing $c$, that $\ell_0(\mathbf{x}_c) + L(1,2) + \ldots + L(i-1, i)  - \ell_0(\mathbf{x})$ is homologous in $H_1(\Sigma)$ to some element $\ell_1(\mathbf{x}) \in C(\mathcal{H})$; so let $\ell(\mathbf{x})$ be $\ell_0(\mathbf{x}) - \ell_1(\mathbf{x})$.

The upshot of the above is that for any polygon $\psi \in \pi_2(\mathbf{x}_1, \mathbf{x}_2, \ldots, \mathbf{x}_k, \mathbf{y})$, we have $\partial \psi = \ell(\mathbf{x}_1) + \ldots + \ell(\mathbf{x}_k) - \ell(\mathbf{y}) + Z$ if $\psi$ is not a bigon or $\partial \psi = \ell(\mathbf{y}) - \ell(\mathbf{x}_1) + Z$ if $\psi$ is a bigon,  where $Z \in C(\mathcal{H})$ vanishes in $H_1(\Sigma)$.
So, choose points $p_{i, s}$ on $\eta^i_s$, in a suitably generic position (away from intersections between isotopes), and orient each curve.  Let $A_0(c) = \sum_{i,s} m_{p_{i,s}}(c)\cdot \eta^i_s$ for any closed multiarc $c$, where $m_{p_{i,s}}$ is the oriented intersection number; and let 
$$A_0(\psi) = A_0\left(\partial \psi - \ell(\mathbf{x}_1) - \ldots - \ell(\mathbf{x}_k) + \ell(\mathbf{y})\right)$$
if $\psi$ is not a bigon, and
$$A_0(\psi) = A_0\left(-\partial \psi - \ell(\mathbf{x}_1) + \ell(\mathbf{y})\right)$$
if $\psi$ is a bigon. 
Letting $A_K$ denote the composition of $A_0$ with the map taking $C(\mathcal{H})$ to $L(\mathcal{H})$, the image of $A_K$ actually lies in $K(\mathcal{H})$.  This map will clearly be additive under splicing.  Furthermore, note that $H^1(Y_{0,j})$ naturally embeds into $K(\mathcal{H})$ (via a choice of basepoint in our Heegaard diagram).  With respect to this, the map $A_K$ is also $H^1(Y_{0,j})$-equivariant.

We also define one more function: let $M(\mathcal{H})$ be the $\mathbb{Z}/N\mathbb{Z}$-module freely generated by $\{\eta^{3i}_g | i \geq 1\}$ (i.e., $\de_g$ and all its isotopic translates).  For $\psi \in \Delta(\mathcal{H})$, define $A_M(\psi)$ to be the summands of $A_0(\partial \psi)$ corresponding to these circles.  Again, this is clearly additive under splicing.  

Going forward, we write $M_i$ for $\eta^{3i}_g$; and we write $A(\psi)$ for $A_K(\psi) \oplus A_M(\psi) \in K(\mathcal{H} \oplus M(\mathcal{H})$.

\subsection{Standard diagram coefficients}
Let $R_K = \mathbb{Z}[K(\mathcal{H})]$ and $R_M = \mathbb{Z}[M(\mathcal{H})]$; set $R = R_K \otimes R_M$, which will be equal to the group ring  of $K(\mathcal{H}) \oplus M(\mathcal{H})$.  Of course, $R_K$ and hence $R$ will also be an algebra over $\mathbb{Z}[H^1(Y_{i,j})]$.

We now define the \emph{chain complex with standard diagram coefficients,} $\underline{CF}^+(Y_{i,j}; R)$, to be the group $CF^+(Y_{i,j}) \otimes R$ equipped with the differential given by
$$\underline{\partial}^+\big([\mathbf{x}, i] \otimes r\big) =  \sum_{\mathbf{y} \in \mathbb{T}_{{\eta}^i} \cap \mathbb{T}_{{\eta}^j}} \sum_{\{\phi \in \pi_2(\mathbf{x}, \mathbf{y})|\mu(\phi) = 1\}} 
\#\widehat{\mathcal{M}}(\phi)\cdot [\mathbf{y}, i-n_w(\phi)] \otimes (e^{A(\phi)}\cdot r),$$ 
where as usual we use exponential notation for elements of the group ring.  The fact that $A$ is additive under splicing ensures that this indeed defines a chain complex.  This chain complex is not \emph{a priori} an invariant, but rather depends on the diagram and the function $A$. However, the relationship with the twisted coefficient chain complex should be not difficult to see.

When $i, j \ne 0$, we will be interested in $\underline{CF}^{\leq 0}(Y_{i,j}; \mathbb{Z}[[U]] \otimes R) = CF^{\leq 0}(Y_{i,j}) \otimes \mathbb{Z}[[U]] \otimes R$, where $\mathbb{Z}[[U]]$ denotes a ring of formal power series, and where $H^1(Y_{i,j})$ acts trivially on $\mathbb{Z}[[U]] \otimes R$.  The differential is defined in the same way as above.  It is not hard to find an isomorphism from the trivially-twisted complex $CF^{\leq 0}(Y_{i,j}; \mathbb{Z}[[U]]) \otimes R$ to $\underline{CF}^{\leq 0}(Y_{i,j}; \mathbb{Z}[[U]] \otimes R)$.

For a set of $g$-tuples of circles $\boldsymbol{\eta}^{i_1}, \ldots, \boldsymbol{\eta}^{i_k}$ with $k$ equal to 3 or 4, we define maps 
$$\underline{f}^+_{i_1,\ldots, i_k}: \underline{CF}^+(Y_{i_1,i_2}; R) \bigotimes_{n=2}^{k-1} \underline{CF}^{\leq 0}(Y_{i_n,i_{n+1}}; \mathbb{Z}[[U]] \otimes R) \rightarrow \underline{CF}^+(Y_{0,i_k}; R) $$
if $i_1=0$.  If $i_1 \ne 0$, we also define a similar map, except replacing $\underline{CF}^+(Y_{i_1,i_j}; R)$ with $\underline{CF}^{\leq 0}(Y_{i_1,i_j}; \mathbb{Z}[[U]] \otimes R)$ for $j = 2$ and $j=k$.  
In both cases, the map is given by
$$\begin{array}{llllllllllllllllllll} 
\underline{f}^+_{i_1,\ldots, i_k}\Big(\bigotimes_{n=1}^{k-1} ([\mathbf{x}_n, j_n] \otimes r_n) \Big) = & \mbox{  } & \mbox{  } & \mbox{  } & \mbox{  } & \mbox{  } & \mbox{  } & \mbox{  } & \mbox{  } & \mbox{  } & \mbox{  } & \mbox{  } & \mbox{  } & \mbox{  } & \mbox{  } & \mbox{  } & \mbox{  } & \mbox{  } & \mbox{  } & \mbox{  }\\
\end{array}$$
$$\mbox{             } \sum_{\mathbf{w} \in \mathbb{T}_{{\eta}^{i_1}} \cap \mathbb{T}_{{\eta}^{i_k}}} 
\sum_{\substack{\psi \in \pi_2(\mathbf{x}_1, \ldots, \mathbf{x}_{k-1}, \mathbf{w})\\ \mu(\psi) = k-3}} 
\#\mathcal{M}(\psi)\cdot \big[\mathbf{w}, \sum_{n=1}^{k-1} j_n - n_w(\psi)\big] \otimes \big(e^{A(\psi) }\cdot \prod_{n=1}^{k-1} r_n\big).$$
In general, for rectangles and pentagons, when we refer to a moduli space of rectangles or pentagons, we will mean the moduli space of those rectangles or pentagons that are pseudoholomorphic with respect to a one-parameter family of almost-complex structures on $\Sigma$; in particular, for $k=4$, we take $\#\mathcal{M}(\psi)$ in the above to be a count of such rectangles.   

Given $\m{s} \in \mbox{Spin}^{\mbox{\scriptsize{c}}}(X_{i_1,\ldots,i_k})$, we also have maps $\underline{f}^+_{i_1,\ldots, i_k, \m{s}}$ defined similarly, except counting only those polygons $\psi$ representing $\m{s}$.

\begin{lemma} \label{thm:4.1} We have for each $i \geq 1$ 
$$\underline{f}^+_{i,i+1,i+2,i+3}(\Theta_{i,i+1} \otimes \Theta_{i+1,i+2} \otimes \Theta_{i+2,i+3}) = \Theta_{i,i+3} \otimes r_i,$$
with $r_i \in \mathbb{Z}[[U]] \otimes R$.  There are constants $c \in \mathbb{Z}/N\mathbb{Z}$ and $k_i \in K(\mathcal{H})$, such that the $U^0$ coefficient of $r_i$ is $e^{k_i + cM_j}$ if $i$ equals $3j-2$ or $3j-1$, and $e^{k_i}\cdot \sum_{n=0}^{N-1} e^{nM_j + (c - n)M_{j+1}}$ if $i = 3j$.
\end{lemma}

\begin{proof} We have three cases to look at, according to the value of $i$ mod 3.  In each case, every holomorphic quadrilateral passing through $\Theta_{i,i+1}$, $\Theta_{i+1,i+2}$, and $\Theta_{i+2,i+3}$ that the map counts has last corner $\Theta_{i,i+3}$, for Maslov index reasons.  

First, consider the case where $i$ equals $3j - 2$ or $3j-1$. Examining periodic domains, we see that precisely one of these quadrilaterals, say $\psi_0$, will have zero multiplicity at the base point $w$. 
Thinking of $\underline{f}^+_{i,i+1,i+2,i+3}(\Theta_{i,i+1} \otimes \Theta_{i+1,i+2} \otimes \Theta_{i+2,i+3})$
as a sum of terms corresponding to each homotopy class of quadrilateral, the term corresponding to $\psi_0$ will be $\Theta_{i,i+3} \otimes e^{A_K(\psi_0)} \otimes e^{A_M(\psi_0)}$, and all the other nonzero summands will be $U^n \cdot \Theta_{i,i+3} \otimes r$ with $n \geq 1$ and $r \in R$.  We immediately see that $A_M(\psi_0)$ is a multiple of $M_{j}$.  It is also not hard to see that this multiple really only depends on the position of the component of $\Theta_{\be\ga}$ in the torus portion of a standard diagram; in particular, this multiple should be the same for all such $i$.

When $i=3j$, things work slightly differently. If we arrange our diagram appropriately, there will now be $N$ holomorphic rectangles with Maslov index -1 and zero multiplicity at $w$ which don't cancel; these can be labelled as $\psi_0, \ldots, \psi_{N-1}$ so that $A_M(\psi_n) = nM_j + (c-n)M_{j+1}$, as these triangles only differ by $\bse{i}\bse{i+3}$-periodic domains.  The rest of the calculation proceeds as before; in particular, $A_K(\psi_n)$ is independent of $n$.  The result follows. \end{proof}

We define chain maps $\underline{f}^+_{(i)}: \underline{CF}^+(Y_{0,i}; R) \rightarrow \underline{CF}^+(Y_{0,i+1}; R)$ by
$$\underline{f}^+_{(i)}([\mathbf{x},j] \otimes r) = 
\underline{f}^+_{0,i,i+1}\big(([\mathbf{x},j] \otimes r) \otimes \Theta_{\eta^i\eta^{i+1}}\big). $$
That these are indeed chain maps follows from the usual untwisted arguments, together with the fact that the quantity $A$ used in the definition of $\underline{f}^+_{0,i,i+1}$ is additive under splicing.  When $i \equiv 0, 2$ mod 3, there are also maps $\underline{f}^+_{(i), \m{s}}$ for each $\m{s}$ in $\spc{W_0}$ if $i \equiv 0$ or $\spc{W'_N}$ if $i \equiv 2$; these only count triangles which represent $\m{s}$ (identifying the appropriate fillings of $X_{0,i,i+1}$ with $W_0$ or $W'_N$).

For each $i > 0$, we have a map $\underline{H}_i: \underline{CF}^+(Y_{0,i}; R) \rightarrow \underline{CF}^+(Y_{0,i+2}; R)$ by
$$ \underline{H}_i([\mathbf{x},j] \otimes r) = \underline{f}^+_{0, i, i+1, i+2}\big(
([\mathbf{x},j] \otimes r) \otimes \Theta_{i,i+1} \otimes \Theta_{i+1,i+2}\big),$$
which is also a chain map.

Furthermore, there are chain maps $\underline{g}_{i}: \underline{CF}^+(Y_{0,i}; R) \rightarrow \underline{CF}^+(Y_{0,i+3}; R)$ given by
$$\underline{g}_{i}([\mathbf{x},j] \otimes r) = \underline{f}^+_{0,i, i+3}\big(([\mathbf{x},j]\otimes r) \otimes (\Theta_{i,i+3} \otimes r_i)\big),$$
with $r_i$ as furnished by Lemma \ref{thm:4.1}.

We need a number of minor results to establish the long exact sequence, as well as for later; we break them up into the next two Propositions.  Henceforth, we write $i \equiv j$ to mean that $i$ and $j$ are equivalent mod 3.  
We say that two generators $x= [\mathbf{x}, j] \otimes e^a$ and $x' = [\mathbf{x}', j'] \otimes e^{a'}$ are \emph{connected by a disk} if there exists $\phi \in \pi_2(\mathbf{x}, \mathbf{x}')$ such that $n_w(\phi) = j - j'$ and $A(\phi) = a' - a$.  There is a similar notion of $\mathbf{x}'$ being connected to $\mathbf{x}$ and some of the $\Theta_{i,j}$ by a polygon.  If it is clear from context which $\Theta_{i,j}$ we mean, we simply say that $\mathbf{x}$ and $\mathbf{x}'$ are \emph{connected by a polygon}.

\begin{prop} \label{thm:4.2}  For $i\geq 1$, there are subcomplexes $C_i$ of $\underline{CF}^+(Y_{0,i}; R)$ and projection maps $\pi_i: \underline{CF}^+(Y_{0,i}; R) \rightarrow C_i$ such that the following hold.

\vspace{3pt}

a) The maps $\pi_i$ are chain maps.

\vspace{3pt}

b) If $i \equiv 0$, then $C_i \cong \underline{CF}^+(Y_{0,i}) \otimes \mathbb{Z}[\mathbb{Z}] \otimes \mathbb{Z}[\mathbb{Z}/N\mathbb{Z}] = \bigoplus_{l,n} C_i(l, n)$, where the sum is over $l \in \mathbb{Z}$ and $n \in \mathbb{Z}/N\mathbb{Z}$, and $C_i(l, n) \cong \underline{CF}^+(Y_{0,i})$.

\vspace{3pt}

c) If $i \equiv 1$, then $C_i \cong \underline{CF}^+(Y_{0,i})$. 

\vspace{3pt}

d) If $i \equiv 2$, then $C_i \cong \underline{CF}^+(Y_{0,i}) \otimes \mathbb{Z}[\mathbb{Z}] = \bigoplus_{l} C_i(l)$, where the sum is over $l \in \mathbb{Z}$, and $C_i(l) \cong \underline{CF}^+(Y_{0,i})$.

\vspace{3pt}

e) If $i \equiv 0$, then $\pi_{i+1} \circ \underline{f}^+_{(i)} \circ \pi_i = \pi_{i+1} \circ \underline{f}^+_{(i)}$, and $\pi_{i+2} \circ \underline{H}^+_{i} \circ \pi_i = \pi_{i+2} \circ \underline{H}^+_{i}$.

\vspace{3pt}

f) If $i\equiv 1$, then $\pi_{i+1} \circ \underline{f}^+_{(i)} = \underline{f}^+_{(i)} \circ \pi_i$. 
 
\vspace{3pt}

g) If $i\equiv 1, 2$, then $\pi_{i+3} \circ \underline{g}_{i}$ takes $C_i$ isomorphically to $C_{i+3}$.
\end{prop}

\begin{proof} Recall the discussion from the end of Section 2.1, from which we cull the following.
The $\bsa\bsg$- and $\bsa\bsd$-periodic domains will each correspond to a subgroup of $K(\mathcal{H})$ (the same for each); 
call this $K_0(\mathcal{H})$.  For periodic domains $\mathcal{P}$, $A_K(\mathcal{P})$ is the image of $\partial \mathcal{P}$ in $K(\mathcal{H})$; thus for these periodic domains, $A_K(\mathcal{P}) \in K_0(\mathcal{H})$.  There is also some periodic domain $P$, whose boundary contains $\be_g$ with multiplicity $d$, such that $K(\mathcal{H}) = K_0(\mathcal{H}) \oplus \mathbb{Z}P$, and any element of $K(\mathcal{H})$ corresponds to a $\bsa\bsb$-periodic domain.  For any of these periodic domains, $A_M(\mathcal{P}) = 0$.  

Therefore, the value of $A_M$ on a bigon is determined by its endpoints; the same is true of $A_K$ up to $K_0$ if $i \equiv 0,1$.  So, there is a function $z_0$ from intersection points in our diagram to $K(\mathcal{H}) \oplus M(\mathcal{H})$, such that $[\mathbf{x}, j] \otimes e^a$ and $[\mathbf{x}', j'] \otimes e^{a'}$ are connected by a disk if and only if $z_0(\mathbf{x}') - z_0(\mathbf{x}) - (a' -a)$ is in the image of the $\bsa\bse{i}$-periodic domains in $K(\mathcal{H}) \oplus 0$ for appropriate $i$.   So, for any fixed element $s$ of $K(\mathcal{H}) \oplus M(\mathcal{H})$, let $C^0_i(s)$ be the subgroup of $\underline{CF}^+(Y_{0,i}; R)$ generated over $\mathbb{Z}$ by 
elements of the form $[\mathbf{x}, j] \otimes e^{z_0(\mathbf{x}) + s+ k}$, for $k \in K(\mathcal{H})$ if $i \equiv 0$, and for $k \in K_0(\mathcal{H})$ otherwise.  Then this group is a subcomplex, isomorphic to $\underline{CF}^+(Y_{0,i})$, and in fact $\underline{CF}^+(Y_{0,i}; R)$ splits as a direct sum of subcomplexes of this form (for varying $s$).

Let $s_j = c\cdot\sum_{l = 1}^{j-1} M_j$, where $c$ is as given by Lemma \ref{thm:4.1}.  
If $i = 3j-2$, let $C_i = C^0_i(s_j)$.  If $i = 3j-1$, let $C_i = \bigoplus_{l \in \mathbb{Z}} C^0_i(s_j + l P)$; denote the summands by $C_i(l)$.  If $i=3j$, let $C_i = \bigoplus_{\substack{l \in \mathbb{Z}\\n \in \mathbb{Z}/N\mathbb{Z}}} C^0_i(s_j + l P + n M_j)$; denote the summands by $C_i(l, n)$.  Let $\pi_i$ be the projection from the full complexes down to these subcomplexes.  It is clear that these maps are chain maps, and that these groups satisfy the isomorphisms of claims b), c), and d).

Let us examine the maps $\underline{f}^+_{(i)}$ and $\underline{H}_{i}$.  If $i =3j$, for any $m \in M(\mathcal{H})$, the image of $C_i  \otimes e^m$ under the former map will lie in $\bigoplus_{n=0}^{N-1} C_{i+1} \otimes e^{m + nM_j}$, since the counted triangles may have boundaries traversing $M_j$, but not any other translates of $\de_g$.  In particular, $\pi_{i+1} \circ \underline{f}^+_{(i)}$ will only be nontrivial on elements of $C_i$.  The image of $C_i \otimes e^m$ under the latter map will likewise lie in $\bigoplus_{n=0}^{N-1} C_{i+2} \otimes e^{m + nM_j}$, so we can say the same for $\pi_{i+1} \circ \underline{H}_{i}$, showing claim e).  

If $i \equiv 1$, then the image of $C_i \otimes e^m$ under $\underline{f}^+_{(i)}$ will lie in $C_{i+1} \otimes e^m$, since none of the triangles this map counts will traverse any of the $\de_g$ translates.  This gives f).  

For each point $\mathbf{x} \in \mathbb{T}_{\al} \cap \mathbb{T}_{\eta^i}$, there is a canonical nearest point $\mathbf{x}' \in \mathbb{T}_{\al} \cap \mathbb{T}_{\eta^{i+3}}$ and small triangle $\psi(\mathbf{x}) \in \pi_2(\mathbf{x}, \Theta_{i,i+3}, \mathbf{x}')$ that admits a single holomorphic representative.  Let $\widetilde{\underline{g}}_{i}$ be the summand of $\underline{g}_{i}$ which counts only this triangle.  If the $U^0$ coefficient of $r_i$ is $r_i^0$, then $\widetilde{\underline{g}}_{i}$ will take $[\mathbf{x},j] \otimes r_{\mathbf{x}}$ to $[\mathbf{x}',j] \otimes r_i^0\cdot r_{\mathbf{x}}$, since the points which measure $A_M(\psi(\mathbf{x}))$ are arranged to lie away from the small triangles.  Therefore, when $i$ equals $3j-2$ or $3j-1$, for each $m \in M(\mathcal{H})$ the map $\widetilde{\underline{g}}_{i}$ gives an isomorphism from $C_i \otimes e^m$ to $C_{i+3} \otimes e^m$, in light of Lemma \ref{thm:4.1} and the definition of $s_j$. In particular $\widetilde{\underline{g}}_{i}$ and hence $\pi_{i+3} \circ \widetilde{\underline{g}}_{i}$ take $C_i$ to $C_{i+3}$ isomorphically. By a standard area-filtration argument, it follows that the same is true for $\pi_{i+3} \circ \underline{g}_{i}$, giving the last assertion.  \end{proof}

\begin{prop} \label{thm:4.3} If $i\equiv 0$, then $\underline{f}^+_{(i), \m{s}} = \underline{f}^+_{(i), \m{s}}|_{C_i(l, n^i_0(\m{s}))}$, for some function $n^i_0: \spc{W_0} \rightarrow \mathbb{Z}/N\mathbb{Z}$. If $\m{s}$ and $\m{s}'$ both have the same restriction to $Y$, this function has the property that $n^i_0(\m{s}) = n^i_0(\m{s}')$ if and only $\m{s}|_{Y_0}$ and $\m{s}'|_{Y_0}$ belong to the same $\mathrm{PD}[N\mu]$-orbit in $\spc{Y_0}$.

If $i \equiv 2$, the image of $C_i(l)$ under $\underline{f}^+_{(i), \m{s}}$ lies in $C_{i+1}(l^i_N(l, \m{s}), n^i_N(\m{s}))$ for some functions $l^i_N: \mathbb{Z} \times \spc{W'_N} \rightarrow \mathbb{Z}$ and $n^i_N: \spc{W'_N} \rightarrow \mathbb{Z}/N\mathbb{Z}$. If $\m{s}$ and $\m{s}'$ both have the same restriction to $Y$, $n^i_N$ has the property that $n^i_N(\m{s}) = n^i_N(\m{s}')$ if and only $\m{s}|_{Y_0}$ and $\m{s}'|_{Y_0}$ belong to the same $\mbox{PD}[N\mu]$-orbit in $\spc{Y_N}$. The function $l^i_N$ satisfies $l^i_N(l + k, \m{s} + m\mathrm{PD}[\widetilde{dF}]) = 
l^i_N(l, \m{s}) + k + m.$
\end{prop}

\begin{proof} If $i =3j$, suppose that generators $x$ and $x'$ of $C_i$ admit triangles $\psi$ and $\psi'$ representing $\m{s} \in \spc{W_0}$ that connect them to respective generators $y$ and $y'$ of $C_{i+1}$.  Then $y$ and $y'$ are certainly connected by a disk, so $x$ will also be connected to $y'$ by $\psi$ spliced with this disk, which represents $\m{s}$.  Clearly $x$ must be connected to $x'\otimes e^{a_K +a_M}$ by a disk for some $a_K \in K(\mathcal{H})$ and $a_M \in M(\mathcal{H})$, and we can splice this disk with $\psi'$ to get another triangle representing $\m{s}$ connecting $x$ and $y'$.  Since any two triangles connecting $x$ and $y'$ that represent the same \sst\ structure will have the same value of $A_M$, it follows quickly that $a_M = 0$.  Thus, if $x \in C_i(\ell, n)$ and $x'\in C_i(\ell', n')$, then $n = n'$, and we set this value to be $n^i_0(\m{s})$.  

Suppose that $\m{s}$ and $\m{s}'$ both have the same restriction to $Y$ and satisfy $\m{s}'|_{Y_0} - \m{s}|_{Y_0} = k\mbox{PD}[\mu]$.  It is not hard to see that if we choose the function $z_0$ of the last proposition carefully, then we may ensure that for $\psi \in \pi_2(\mathbf{x}, \Theta_{i,i+1}, \mathbf{y})$ and $\psi' \in \pi_2(\mathbf{x}, \Theta_{i,i+1}, \mathbf{y}')$ representing these two \sst\ structures, $A_M(\psi') - A_M(\psi) -(z_0(\mathbf{y}') - z_0(\mathbf{y}))$ equals $kM_j$ plus an element of $K(\mathcal{H})$.  It follows that $n^i_0(\m{s}) = n^i_0(\m{s}')$ if and only if $k$ is a multiple of $N$. 

The corresponding claims when $i \equiv 2$ follows along similar lines; the only real difference is that we now note that for two triangles $\psi, \psi'$ connecting the same two points representing the same \sst\ structure, in addition to $A_M(\psi) = A_M(\psi')$, we also have $A_K(\psi') -A_K(\psi) \in K_0(\mathcal{H})$.  Also, note that $\mbox{PD}[\widetilde{dF}]$ is the Poincar{\'e} dual of a class in $H_2(W'_N)$ whose associated periodic domain represents $P$; from this, the claim about $l^i_N$ is clear. \end{proof}

\subsection{The long exact sequence}
We first prove a long exact sequence in terms of the above, and then we translate it into a more invariant result.  In the following, for any map $\underline{f}$ whose source and target are respectively $\underline{CF}^+(Y_{0,i}; R)$ and $\underline{CF}^+(Y_{0,j}; R)$, we write $\underline{\pi f}$ for $\pi_j \circ \underline{f} \circ \pi_i$, and similarly for induced maps on homology.  The precomposition by $\pi_i$ can be thought of as a restriction of the domain.

\begin{theorem} \label{thm:4.4}  There is a long exact sequence 
$$ H_*(C_1) \overset{\underline{\pi F}^+_{(1)}}{\longrightarrow} H_*(C_2) \overset{\underline{\pi F}^+_{(2)}}{\longrightarrow} H_*(C_3) \overset{\underline{\pi F}^+_{(3)}}{\longrightarrow} H_*(C_4) \rightarrow \ldots;$$
furthermore, there exists a quasi-isomorphism 
$$\psi^+: M(\underline{\pi F}^+_{(2)}|_{C_2}) \rightarrow C_4, $$
where $M$ denotes the mapping cone.
\end{theorem}

\begin{proof}
We follow the strategy of \cite{OSBD}.  To do this, we will show that there are chain nullhomotopies $\underline{\pi H}_i: C_i \rightarrow C_{i+2}$ of $\underline{\pi f}^+_{(i+1)} \circ \underline{\pi f}^+_{(i)}$, and that the maps $\underline{\pi f}^+_{(i+2)} \circ \underline{\pi H}_i - \underline{\pi H}_{i+1} \circ \underline{\pi f}^+_{(i)}$ are quasi-isomorphisms.

We consider the moduli space of quadrilaterals $\psi \in \pi_2(\mathbf{x}, \Theta_{i,i+1}, \Theta_{i+1,i+2}, \mathbf{w})$ with $\mu(\psi) = 0$.  This space is compact and oriented of dimension 1 (recall our use of the phrase ``moduli space''), so the signed count of its boundaries vanishes.  Noting that the $\Theta$ points are cycles, we see that with appropriate orientation conventions, a twisted count of the ends yields 
$$\underline{\partial}^+_{{\eta}^0{\eta}^{i+2}} \circ \underline{H}_i + \underline{H}_i \circ \underline{\partial}^+_{{\eta}^0{\eta}^{i}} = \underline{f}^+_{(i+1)} \circ \underline{f}^+_{(i)} 
+ \sum_{\m{s}} \underline{f}^+_{0,i,i+2, \m{s}_1}\big(\mbox{ }\cdot\mbox{ }\otimes \underline{f}^+_{i,i+1,i+2, \m{s}_2}(\Theta_{i,i+1} \otimes \Theta_{i+1,i+2})\big),  $$
where $\m{s}_1$ and $\m{s}_2$ are the appropriate restrictions of $\m{s}$.  
It is not hard to see that 
$$\pi_{i+2} \circ \left(\underline{\partial}^+_{{\eta}^0{\eta}^{i+2}} \circ \underline{H}_i + \underline{H}_i \circ \underline{\partial}^+_{{\eta}^0{\eta}^{i}}\right) \circ \pi_i = \underline{\pi \partial}^+_{{\eta}^0{\eta}^{i+2}} \circ \underline{\pi H}_i + \underline{\pi H}_i \circ \underline{\pi \partial}^+_{{\eta}^0{\eta}^{i}}$$
and that 
$$\pi_{i+2} \circ \left(\underline{f}^+_{(i+1)} \circ \underline{f}^+_{(i)}\right) \circ \pi_i = \underline{\pi f}^+_{(i+1)} \circ \underline{\pi f}^+_{(i)}.$$
So as usual, to show that $\underline{\pi H}_i$ is a chain nullhomotopy, it suffices to show that for each $\m{s}_1 \in \spc{X_{0,i,i+2}}$, we have
\begin{equation}
\label{eq:10}
\sum_{  \{ \m{s} \big| \m{s}|_{X_{0,i,i+2}} = \m{s}_1 \}  } \underline{f}^+_{i,i+1,i+2, \m{s}|_{X_{i,i+1,i+2}}}(\Theta_{i,i+1} \otimes \Theta_{i+1,i+2}) = 0.
\end{equation}

To verify this equation, in fact, the untwisted arguments work with few changes.  Suppose first that we have two triangles $\psi$ and $\psi'$ with $n_w(\psi) = n_w(\psi')$ connecting the same three points $\mathbf{x}_{i+j} \in \mathbb{T}_{\eta^{i+j}} \cap \mathbb{T}_{\eta^{i+j+1}}$ for $i \geq 1$ and $j =0,1,2$.  Then $\partial \psi'$ will equal $\partial \psi$ plus a number of doubly-periodic domains plus some triply-periodic domain $\mathcal{P}$, where $\partial \mathcal{P}$ is a multiple of $\ga_g - \be_g - N\de_g$ (up to replacing circles with corresponding translates).  All of these domains go to $0$ in both $L(\mathcal{H})$ and $M(\mathcal{H})$, the latter because $\de_g$ is of order $N$ in $M(\mathcal{H})$; and so $A_K(\psi') - A_K(\psi) = A_M(\psi') - A_M(\psi) = 0$.  Thus, for triangles through any three such points, $n_w(\psi)$ determines $A_K(\psi)$ and $A_M(\psi)$.

Having shown this, the usual arguments show that for all $i \ne 0$ and each $k > 0$, there are two homotopy classes of triangles $\psi^{\pm}_k \in \pi_2(\Theta_{i,i+1}, \Theta_{i+1,i+2}, \Theta_{i,i+2})$ with $\mu(\psi^{\pm}_k) = 0$ and $n_w(\psi^+_k) = n_w(\psi^-_k)$ (both of which equal a quadratic function of $k$, depending on $N$ and the precise position of $\Theta_{\be\ga}$), and each of these admit a single holomorphic representative.  Since $n_w(\psi^+_k) = n_w(\psi^-_k)$, we know that $A_K(\psi^+_k) = A_K(\psi^-_k)$ and $A_M(\psi^+_k) = A_M(\psi^-_k)$; hence, the two corresponding terms in $\underline{f}^+_{i,i+1,i+2, \m{s}}(\Theta_{i,i+1} \otimes \Theta_{i+1,i+2})$ will appear with the same twisting coefficient. Furthermore, orientations can be arranged so that the terms appear with opposite signs.  Likewise, when $i \equiv 2$, any of the other points $\Theta'$ in $\mathbb{T}_{\ga} \cap \mathbb{T}_{\be}$ will have two homotopy classes of triangles $\psi^{\pm}_k(\Theta') \in \pi_2(\Theta_{i,i+1}, \Theta_{i+1,i+2}, \Theta')$ with $\mu(\psi^{\pm}_k(\Theta')) = 0$ and $n_w(\psi^+_k(\Theta')) = n_w(\psi^-_k(\Theta'))$, and we can ensure that these yield cancelling terms as well.  Thus, $\underline{f}^+_{i,i+1,i+2}(\Theta_{i,i+1} \otimes \Theta_{i+1,i+2}) = 0;$ the left hand side of Equation \ref{eq:10} is equal to sum of terms equal to $\underline{f}^+_{i,i+1,i+2}(\Theta_{i,i+1} \otimes \Theta_{i+1,i+2})$, and so this equation holds, proving that the $\underline{\pi H}_i$ are nullhomotopies.

Next, we examine the desired quasi-isomorphisms.  Consider the moduli space of pentagons $\psi \in \pi_2(\mathbf{x}, \Theta_{i,i+1}, \Theta_{i+1,i+2}, \Theta_{i+2,i+3}, \mathbf{w})$ with $\mu(\psi) = 0$.  We count signed boundary components, noting again that the $\Theta$ points are cycles and that $\underline{f}^+_{i,i+1,i+2}(\Theta_{i,i+1} \otimes \Theta_{i+1,i+2}) = 0$.  Ignoring the terms that vanish due to these observations, and orienting appropriately, we have
\begin{flushleft}
$\begin{array}{l}
\underline{f}^+_{(i+2)} \circ \underline{H}_i - \underline{H}_{i+1} \circ \underline{f}^+_{(i)} =
\underline{\partial}^+_{{\eta}^{i+3}} \circ \underline{J} + 
\underline{J} \circ \underline{\partial}^+_{{\eta}^i} +  \\
\end{array}$
\end{flushleft}
$$\sum_{\m{s}} \underline{f}^+_{0,i,i+3, \m{s}_1}
\big(\mbox{ }\cdot\mbox{ }\otimes \underline{f}^+_{i,i+1,i+2,i+3, \m{s}_2}(\Theta_{i,i+1} \otimes \Theta_{i+1,i+2} \otimes \Theta_{i+2,i+3})\big),$$
where $\underline{J}$ counts pentagons $\psi$ with $\mu(\psi) = -1$.  In fact, we can dispense with the \sst\ structures in the second line.

First, consider the case where $i$ is not $0$.  In these cases, it is easy to see that 
the above equation together with 4.2 e) and f) imply that
\begin{flushleft}
$\begin{array}{l}
\underline{\pi f}^+_{(i+2)} \circ \underline{\pi H}_i - \underline{\pi H}_{i+1} \circ \underline{\pi f}^+_{(i)} =
\underline{\pi \partial}^+_{{\eta}^{i+3}} \circ \underline{\pi J} + 
\underline{\pi J} \circ \underline{\pi \partial}^+_{{\eta}^i} +  \\
\end{array}$
\end{flushleft}
$$\underline{\pi f}^+_{0,i,i+3}
\big(\pi_i( \cdot )\otimes \underline{f}^+_{i,i+1,i+2,i+3}(\Theta_{i,i+1} \otimes \Theta_{i+1,i+2} \otimes \Theta_{i+2,i+3})\big).$$
Hence, in these cases, it suffices to show that the term in the second line of this expression is a quasi-isomorphism when considered as a map from $C_i$ to $C_{i+3}$; but this is immediate from Proposition \ref{thm:4.2} g).

We have only to show the analogous result for $i = 3j$.  This requires a little bit of finesse.  

Any element of $C_i$ is of the form $x = [\mathbf{x}, i] \otimes e^{k + s_j + nM_j}$, for some $k \in K(\mathcal{H})$ and $n \in \mathbb{Z}/N\mathbb{Z}$.  Then $\pi_{i+1} \circ \underline{f}^+_{(i)}(x)$ gives a count of triangles $\psi$ originating at $\mathbf{x}$ with $m_{p^i}(\partial \psi) = c - n$ (where $c$ is as in Lemma \ref{thm:4.1}), and $\pi_{i+2} \circ \underline{H}_i(x)$ does the same for rectangles.

Given $\ell \in \mathbb{Z}/N\mathbb{Z}$, let $\underline{J'}_{\ell}$ be a map which counts holomorphic pentagons \linebreak $\psi \in \pi_2(\mathbf{x}, \Theta_{i,i+1}, \Theta_{i+1, i+2}, \Theta_{i+2, i+3}, \mathbf{w})$ with $\mu(\psi) = -1$, but \emph{only} those pentagons for which $m_{p^i}(\partial \psi) = \ell$.  We bundle these into a single map $\underline{J'}$, by having $\underline{J'}(x) = \underline{J'}_{c -n}(x)$ when $x$ is of the form given above.  Then, examining boundary components of moduli spaces of pentagons $\psi$ with $\mu(\psi) = 0$ and $m_{p^i}(\partial \psi) = c_j -n$, we see that 
$$\pi_{i+3} \circ \underline{f}^+_{(i+2)} \circ \pi_{i+2} \circ \underline{H}_i(x) - \pi_{i+3} \circ \underline{H}_{i+1} \circ \pi_{i+1} \circ \underline{f}^+_{(i)}(x) = \underline{\pi \partial}^+_{{\eta}^{i+3}} \circ \underline{J'}(x) + 
\underline{J'} \circ \underline{\pi \partial}^+_{{\eta}^i}(x) + G(x),$$
where   
$$G(x) = \sum_{\mathbf{w} \in \mathbb{T}_{\al} \cap \mathbb{T}_{\eta^{i+3}}}
\sum_{\substack{\psi \in \pi_2(\mathbf{x}, \Theta_{i,i+3}, \mathbf{w})\\
\mu(\psi) =0}}
C(\psi) \cdot \#\mathcal{M}(\psi) \cdot [\mathbf{w}, i - n_w(\psi)] \cdot e^{k + s_j + nM_j + A(\psi)}, $$
for appropriate constants $C(\psi) \in \mathbb{Z}[[U]] \otimes R$.
Precisely, write
$$\underline{f}^+_{i,i+1,i+2,i+3, \m{s}_2}(\Theta_{i,i+1} \otimes \Theta_{i+1,i+2} \otimes \Theta_{i+2,i+3})\big) = \sum_{q \in \mathbb{Z}/N\mathbb{Z}} e^{qM_j} \cdot P(q)
\cdot \Theta_{i,i+3},$$
where $P(q) \in \mathbb{Z}[[U]] \otimes R$ contains no powers of $e^{M_j}$; then if $m_{p^i}(\partial \psi) = a$, $C(\psi)$ should equal $P(c - n - a)$.  Indeed, the coefficients of $P(q)$ all arise from triangles \linebreak $\psi_1 \in \pi_2(\Theta_{i,i+1}, \Theta_{i+1,i+2}, \Theta_{i+2,i+3}, \Theta_{i,i+3})$ for which $m_{p^i}(\partial \psi_1) = q$, and we essentially want to count pairs $(\psi, \psi_1)$ such that $m_{p^i}\left(\partial (\psi * \psi_1) \right) = c - n$.

Let $\widetilde{G}$ be the count of triangles that $G$ performs, with the second sum restricted to canonical small triangles.  If $\psi_0$ is such a small triangle, recall that we have arranged so that $m_{p^i}(\partial \psi_0) = 0$.  Hence, Lemma \ref{thm:4.1} implies that $C(\psi_0)$ will have $U^0$ coefficient equal to  
$e^{k + (c-n)M_j + nM_{j+1}}$ for some $k \in K(\mathcal{H})$ and $n \in \mathbb{Z}/N\mathbb{Z}$.  Therefore, $\widetilde{G}$ is an isomorphism; as usual, an area filtration argument can then be used to show that $G$ is an isomorphism as well, proving the quasi-isomorphism statement for $i \equiv 0$.

Thus, the mapping cone Lemma of \cite{OSBD} finishes the proof. \end{proof}

We now refine Theorem \ref{thm:4.4}.
Recall that the usual cobordism-induced map $\underline{F}^+_{W'_N, \m{s}}: \underline{HF}^+(Y_N, \m{s}|_{Y_N}) \rightarrow \underline{HF}^+(Y, \m{s}|_Y)$, and indeed the groups themselves, are well-defined only up to a sign and the actions of $H^1(Y_N)$ and $H^1(Y)$.  We may fix these so that the following holds.

\begin{theorem} \label{thm:4.5}  For special $K$, there is a long exact sequence 
$$\begin{array}{ll} 
\ldots \overset{\underline{F}^+_0}{\longrightarrow} \underline{HF}^+(Y_0, \m{S}^N_0(\m{t}_0)) &
\overset{\underline{F}^+_N}{\longrightarrow} 
\underline{HF}^+(Y_N, \m{S}_N(\m{t}_0)) \otimes \mathbb{Z}[\mathbb{Z}]\\
& \overset{\underline{F}^+}{\longrightarrow} 
\underline{HF}^+(Y, \m{S}_{\infty}(\m{t}_0)) \otimes \mathbb{Z}[\mathbb{Z}] 
\rightarrow \ldots, \\
\end{array}$$
where each group is taken with totally twisted coefficients.   We have
$$\underline{F}^+ = \sum_{\m{s} \in \m{S}_{N\infty}(\m{t}_0)} \underline{F}^+_{W'_N, \m{s}} \otimes \1$$
where $\underline{F}^+_{W'_N}$ is the usual twisted  coefficient map induced by $W'_N$, appropriately fixed; and so there exists a quasi-isomorphism 
$$\psi^+: M\left(\sum_{\m{s} \in \m{S}_{N\infty}(\m{t}_0)} \underline{f}^+_{\m{s}}\right)
\rightarrow \underline{CF}^+\left(Y_0, \m{S}^N_0(\m{t}_0)\right), $$
where $\underline{f}^+_{\m{s}}$ is the chain map inducing $\underline{F}^+_{\m{s}}$.

Furthermore, we can choose $\m{t}_N \in \m{S}_N(\m{t}_0)$ and $\m{t}_{\infty} \in \m{S}(\m{t}_0)$, and identifications
$$\underline{CF}^+(Y_N, \m{S}_N(\m{t}_0)) \otimes \mathbb{Z}[\mathbb{Z}] \cong \bigoplus_{i\in\mathbb{Z}} \underline{CF}^+(Y_N, \m{t}_N + i\mathrm{PD}[F']|_{Y_N})$$
and
$$\underline{CF}^+(Y, \m{S}_{\infty}(\m{t}_0)) \otimes \mathbb{Z}[\mathbb{Z}] \cong \bigoplus_{i\in\mathbb{Z}} \underline{CF}^+(Y, \m{t}_{\infty} + i\mathrm{PD}[F']|_Y),$$
where we treat each summand as distinct, so that if $\underline{f}^+_{\m{s}}$ takes summand $i$ to summand $j$, then $\underline{f}^+_{\m{s}+k\mathrm{\scriptsize{PD}}[F']|_{W'_N}}$ takes summand $i$ to summand $j + k$.
\end{theorem}

\begin{proof}  For $\m{t}_0 \in \spc{Y_0}$, let $\m{S}^*_0(\m{t}_0) = \m{t}_0 + \mathbb{Z} \cdot\mbox{PD}[\mu]$ and $\m{S}^*_N(\m{t}_0) = \m{S}_N(\m{t}_0) + \mathbb{Z}\cdot\mbox{PD}[\mu]$.
It follows quickly from Theorem \ref{thm:4.4} that there is a long exact sequence of the form 
$$\begin{array}{ll} 
\ldots \overset{\underline{F}^+_0}{\longrightarrow} \underline{HF}^+(Y_0, \m{S}^*_0(\m{t}_0)) &
\overset{\underline{F}^+_N}{\longrightarrow} 
\underline{HF}^+(Y_N, \m{S}^*_N(\m{t}_0)) \otimes \mathbb{Z}[\mathbb{Z}]\\
& \overset{\underline{F}^+}{\longrightarrow} 
\bigoplus^N\underline{HF}^+(Y, \m{S}_{\infty}(\m{t}_0)) \otimes \mathbb{Z}[\mathbb{Z}]  
\rightarrow \ldots, \\
\end{array}$$ 
by identifying each $C_i$ with one of the above groups and noting that each map of the long exact sequence naturally splits along \sst\ structures as given.

All three groups above further decompose into $N$ subgroups: the first two by breaking into sums indexed by the $N$ differenet $\mathbb{Z}\cdot N\mbox{PD}[\mu]$-suborbits of \sst\ structures, and the last in the obvious manner.
So to show that the exact sequence of the statement exists, it suffices to show that the maps in the long exact sequence respect these decompositions (i.e. each map takes each summand of its source to a distinct summand of its target).
It is not difficult to see that $\underline{F}^+_N$ does, by (for example) examining Lemma \ref{thm:A.1}. 
That the other two do follows from the statement about $n^i_0$ and $n^i_N$ in Proposition \ref{thm:4.3}.
 
The identification of $\underline{F}^+_{W'_N}$ and the mapping cone statement are both clear. Finally, the last statement follows from the second paragraph of Proposition \ref{thm:4.3}. \end{proof}

For an integer $\de > 0$, let $\underline{CF}^{\de}$ denote the subcomplex of $\underline{CF}^+$ consisting of elements in the kernel of $U^{\de}$.  In the statements and proofs of all the above, we can go through line by line and systematically replace $\underline{CF}^+$ and $\underline{HF}^+$ with $\underline{CF}^{\de}$ and $\underline{HF}^{\de}$ (interpreting every map with one of these groups as source as having appropriately restricted domain).  It is  straightforward to then go through and check that everything still makes sense and holds true.  We will specifically need the quasi-isomorphism statement, so we state it precisely, and enhance it a bit.

\begin{corollary} 
\label{thm:4.6} 
If $K$ is special, then there is a quasi-isomorphism
$$\psi^{\de}: M\left(\sum_{\m{s} \in \m{S}_{N\infty}(\m{t}_0)} \underline{f}^+_{W'_N}|_{\underline{CF}^{\de}(Y_N, \m{s}|_{Y_N})} \otimes \1\right) \rightarrow \underline{CF}^{\de}(Y_0, \m{S}^N_0(\m{t}_0)). $$
The mapping cone inherits a $U$ action from the summands of its chain group; with respect to this action, the quasi-isomorphisms are $U$-equivariant.
\end{corollary}

\begin{proof} We explain the last statement quickly.  The chain group of the mapping cone is  $\underline{CF}^{\de}(Y_N, \m{S}_N(\m{t}_0)) \oplus \underline{CF}^{\de}(Y, \m{S}_{\infty}(\m{t}_0))$.
Examining the proof of the mapping cone Lemma as given in \cite{OSBD}, we can express $\phi^{\de}$ in terms of this decomposition by
$$\phi^{\de}(x,y) = \underline{H}_2(x) + \underline{f}^+_0(y);$$ 
these maps are $U$-equivariant, and thus so is $\phi^{\de}$. \end{proof}

\section{Twisted Knot Floer Homology}
Knot Floer homology was originally defined in \cite{OSKI} and \cite{RAS} for nullhomologous knots; in \cite{OSRS} the definition is extended to knots that are only rationally nullhomologous.  We recall the construction here, extending it to the case of twisted coefficients.  We then prove analogues of some of the results in \cite{OSRS} in the twisted setting, which relate the knot filtration with the homologies of large $N$ surgeries on the knot.  We write out a large portion of the details in this, even though most of the results here follow in a similar manner to previous results; we do this mainly so that we may be unambiguous when referring to these results afterwards.

\subsection{The knot filtration}  
Take a doubly-pointed Heegard diagram $(\Sigma, \bsa, \bsd, w, z)$ for an oriented, rationally nullhomologous knot $K \subset Y$, which need not be standard. We can form
the usual chain complex $CF^{\circ}(\Sigma, \bsa, \bsd, w)$ (where $\circ$ denotes any of $\widehat{\mbox{  }}, +, -$, or $\infty$), but the extra point endows this with an additional $\mathbb{Z}$ filtration, via the ordering on the fibers of $\underline{\mbox{Spin}}^{\mbox{\scriptsize{c}}}(Y, K)$.  In \cite{OSRS}, it is asserted that the $\mathbb{Z} \oplus \mathbb{Z}$-filtered chain homotopy type of $CF^{\circ}(\Sigma, \bsa, \bsd, w, z, \xi) = CF^{\circ}(Y, K, \xi)$ is an invariant of $Y, K$ and $\xi \in \underline{\mbox{Spin}}^{\mbox{\scriptsize{c}}}(Y, K)$. 

There is an obvious alteration of this construction to get a filtration on the twisted complex.  One still needs an invariance result, but it is easy to see that the filtration ``cares'' only about the generators of the complex, and not about the coefficients appearing next to them, so that invariance comes from the respective invariance results for the twisted three-manifold invariant and the untwisted knot filtration.  

Let us be more precise about this.
Recall the conventional set up for twisted coefficients: one chooses complete sets of paths in the sense of Section 3 of \cite{OSHD}, which (together with the choice of basepoint $w$) yield a surjective additive assignment $h$ from $\pi_2(\mathbf{x}, \mathbf{y})$ (when it is nonempty) to $H^1(Y_{\al\de})$.  Then we take our universal twisted coefficient ring to be $\mathbb{Z}[H^1(Y_{\al\de})]$, with twisting specified by $h$.  

So, if $\xi \in \underline{\mbox{Spin}}^{\mbox{\scriptsize{c}}}(Y, K)$, we let $\m{T}(\xi)$ be the set of $[\mathbf{x}, i, j] \in (\mathbb{T}_{\al} \cap \mathbb{T}_{\de}) \times \mathbb{Z} \times \mathbb{Z}$ such that 
\begin{equation}
\label{eq:11}
\underline{\m{s}}_{w,z}(\mathbf{x}) - (i - j)\mbox{PD}[\mu] = \xi.
\end{equation}
Then, for any $\mathbb{Z}[H^1(Y_{\al\de})]$-module $M$, define $\underline{CFK}^{\infty}(\Sigma, \bsa, \bsd, w, z, \xi; M)$ to be the abelian group $CFK^{\infty}(\Sigma, \bsa, \bsd, w, z, \xi) \otimes M$, where $CFK^{\infty}(\Sigma, \bsa, \bsd, w, z, \xi)$ is the free abelian group generated by $[\mathbf{x}, i, j] \in \m{T}(\xi)$, with differential 
$$\underline{\partial}^{\infty}([\mathbf{x}, i, j] \otimes m) = 
\sum_{\mathbf{y} \in \mathbb{T}_{{\al}} \cap \mathbb{T}_{{\de}}} \sum_{\{\phi \in \pi_2(\mathbf{x}, \mathbf{y})|\mu(\phi) = 1\}} \#\widehat{\mathcal{M}}(\phi)\cdot [\mathbf{y}, i-n_w(\phi), j-n_z(\phi)] \otimes e^{h(\phi)}\cdot m,$$
which is a finite sum if the Heegaard diagram $(\Sigma, \bsa, \bsd, w)$ is strongly-$G_K(\xi)$ admissable.  The differential takes $\underline{CFK}^{\infty}(\Sigma, \bsa, \bsd, w, z, \xi; M)$ to itself in light of Equation \ref{eq:4}, and the usual arguments show that $(\underline{\partial}^{\infty})^2 = 0$.  If we don't specify $M$, we take $M$ to be $\mathbb{Z}[H^1(Y_{\al\de})]$.

This group is naturally $\mathbb{Z} \oplus \mathbb{Z}$-graded, by declaring $[\mathbf{x}, i, j] \otimes m$ to be in grading $(i,j)$, for $m \in M$; this makes $\underline{CFK}^{\infty}(\Sigma, \bsa, \bsd, w, z, \xi; M)$ into a $\mathbb{Z} \oplus \mathbb{Z}$-filtered chain complex.
    
\begin{theorem}
\label{thm:5.1}  The $\mathbb{Z} \oplus \mathbb{Z}$-filtered chain homotopy type of the $\mathbb{Z}[U] \otimes \mathbb{Z}[H^1(Y_{\al\de})]$-module $\underline{CFK}^{\infty}(\Sigma, \bsa, \bsd, w, z, \xi; M) = \underline{CFK}^{\infty}(Y, K, \xi; M)$ is an invariant of $Y$, $K$, and $\xi$.
\end{theorem}

\begin{proof}  Considering the above remarks, this is a routine adaptation of the invariance arguments from \cite{OSKI}. \end{proof}

We abbreviate $\underline{CFK}^{\infty}(Y, K, \xi; M)$ to $\underline{C}_{\xi}(Y, K; M)$.  If $S$ is a subset of $\mathbb{Z} \oplus \mathbb{Z}$ such that $(i', j') \in S$ whenever $(i,j) \in S$ with $i' \leq i, j' \leq j$ -- or if $S$ is the complement of one such region in another -- then we have a subgroup  $\underline{C}_{\xi}(S)(Y, K; M)$ generated by elements of the form $[\mathbf{x}, i, j] \otimes m$ with $(i,j) \in S$, which is naturally a subquotient chain complex of $\underline{C}_{\xi}(Y, K; M)$.  Such a chain complex will also be an invariant of $Y$, $K$, and $\xi$.  Of particular interest is the set $S = \{(i,j)| i \geq 0 \mbox{ or } j \geq 0\}$; for this $S$, we write $\underline{C}^+_{\xi}(Y, K; M)$.  If we only wish to calculate $\underline{C}^+_{\xi}(Y, K; M)$, we need only have a weakly admissable Heegaard diagram to have a well-defined differential.

\subsection{Relationship with large $N$ surgeries} 
Note that for $\xi \in \underline{\mbox{Spin}}^{\mbox{\scriptsize{c}}}(Y,K)$, if we take the quotient of $\underline{C}^+_{\xi}(Y, K)$ by the subgroup of elements of the form $[\mathbf{x}, i, j] \otimes e^{\ell}$ with $i<0$, we are left with a chain complex naturally isomorphic to  $\underline{CF}^+\big(Y, E^{-1}_{K, N}(\xi)|_Y\big)$; essentially, we forget the filtration.  Let $v_{\xi, K}$ be the quotient map, taking $[\mathbf{x}, i, j] \otimes e^{\ell}$ to $[\mathbf{x}, i] \otimes e^{\ell}$.  

Switching the basepoints $w$ and $z$, we have a doubly-pointed Heegaard diagram representing $K$ with the reverse orientation, which we denote $-K$.  Since $\um{s}_{w,z}(\mathbf{x}) = \um{s}_{z,w}(\mathbf{x})$, the map  $s:\underline{C}^+_{\xi}(Y, K) \rightarrow \underline{C}^+_{\xi}(Y, -K)$ that takes $[\mathbf{x}, i,j]$ to $[\mathbf{x}, j,i]$ is an isomorphism of chain complexes -- of filtered chain complexes, even, although \emph{not} with the usual filtration on $\underline{C}^+_{\xi}(Y, -K)$.   Define $h_{\xi, K}: \underline{C}^+_{\xi}(Y, K) \rightarrow \underline{CF}^+\big(Y, E^{r^{-1}}_{-K, N}(\xi)|_Y\big)$ be equal to $v_{\xi, -K} \circ s$.  Of course, all of this works with an arbitrary coefficient module as well.  We henceforth drop $K$ from the notation and write $v_{\xi}$ and $h_{\xi}$.

Now, fix a knot $K \subset Y$, and a family $\mathcal{F}$ of diagrams for $K$.  
Suppose $N$ is large enough so that there is a $\m{t}$-proper diagram in $\mathcal{F}$ for each $\m{t} \in \spc{Y_N}$; in such a $\m{t}$-proper diagram, choose a small triangle $\psi_{\m{t}}$ with a corner representing $\m{t}$, for each $\m{t}$. 
Recall that for all $N$ such that $N\mu + \la$ is not special, there are canonical identifications of the $\bsa\boldsymbol{\ga}-$periodic domains in the Heegaard triple for $W'_N$ with the $\bsa\boldsymbol{\de}-$periodic domains; and that  we have a canonical identification of $H^1(Y)$ and $H^1(Y_N)$, via the images of $H_2(Y)$ and $H_2(Y_N)$ in $H_2(W'_N)$.  
Any triangle representing $E_{K,N}\big(\m{s}_w(\psi_{\m{t}})\big)$ can be written as $\psi = \psi_{\m{t}} + \phi_{\al\ga} + \phi_{\al\de} + \phi_{\ga\de}$.  We define $h'(\psi)$ to be $h(\phi_{\al\ga}) + h(\phi_{\al\de})$, where $h$ denotes the additive assignments used to define $\underline{CF}^+(Y_N, \m{t})$ and $\underline{CF}^+(Y, \m{s}_w(\psi_{\m{t}})|_Y)$, and both $h(\phi_{\al\ga})$ and $h(\phi_{\al\de})$ are considered as elements of $H^1(Y)$.  

We now can define a map $\Psi^+_{\m{t}, N}: \underline{CF}^+(Y_N, \m{t}; M) \rightarrow \underline{C}^+_{\Xi(\m{t})}(Y, K; M)$ by
$$\begin{array}{lllllllllllllllllllll}
\Psi^+_{\m{t}, N}([\mathbf{x}, i] \otimes m) = & \mbox{ } & \mbox{ } & \mbox{ } & \mbox{ } & \mbox{ } & \mbox{ } & \mbox{ } & \mbox{ } & \mbox{ } & \mbox{ } & \mbox{ } & \mbox{ } & \mbox{ } & \mbox{ } & \mbox{ } & \mbox{ } & \mbox{ } & \mbox{ } & \mbox{ } & \mbox{ } \\
\end{array}$$
$$\sum_{\mathbf{w} \in \mathbb{T}_{{\al}} \cap \mathbb{T}_{{\de}}}
\sum_{\left\{
\begin{scriptsize}
\begin{array}{l|l}
\psi \in \pi_2(\mathbf{x}, \Theta_{\ga\de}, \mathbf{w}) & \mu(\psi) = 0\\
 & \m{s}_w(\psi) =  \m{s}_+(\m{t}) \\
\end{array} \end{scriptsize}\right\}} \#\widehat{\mathcal{M}}(\psi)\cdot [\mathbf{y}, i-n_w(\psi), i-n_z(\psi)] \otimes e^{h'(\psi)}\cdot m,$$
where $M$ can be considered to be a module over both $\mathbb{Z}[H^1(Y)]$ and $\mathbb{Z}[H^1(Y_N)]$.

\begin{theorem}
\label{thm:5.2}  Fix a family of doubly-pointed standard Heegaard triples for $K$, and a $\mathbb{Z}[H^1(Y)]$-module $M$.  Write $\xi(\m{t})$ for $E_{K,N}\big(\m{s}_{K+}(\m{t})\big)$. If $N$ is sufficiently large, then there are $\m{t}$-proper diagrams inducing commutative squares
$$\begin{CD}
\underline{CF}^+(Y_N, \m{t};M) @>{\underline{f}^+_{W'_N, \m{s}_{K+}(\m{t})}}>> \underline{CF}^+(Y, \m{s}_{K+}(\m{t})|_{Y};M)\\
@V{\Psi^+_{\m{t}, N}}VV         @VVV \\
\underline{C}^+_{\xi(\m{t})}(Y, K; M) @>v_{\xi(\m{t})}>> \underline{CF}^+(Y, \m{s}_{K+}(\m{t})|_{Y};M)
\end{CD}$$
and
$$\begin{CD}
\underline{CF}^+(Y_N, \m{t};M) @>{\underline{f}^+_{W'_N, \m{s}_{K-}(\m{t})}}>> \underline{CF}^+(Y, \m{s}_{K-}(\m{t})|_{Y};M)\\
@V{\Psi^+_{\m{t}, N}}VV         @VVV \\
\underline{C}^+_{\xi(\m{t})}(Y, K; M) @>h_{\xi(\m{t})}>> \underline{CF}^+(Y, \m{s}_{K-}(\m{t})|_{Y};M).
\end{CD}$$
In each square, the vertical maps are isomorphisms of relatively $\mathbb{Z}$-graded complexes over $\mathbb{Z}[U] \otimes \mathbb{Z}[H^1(Y)]$; the righthand maps are each multiplication by elements of $H^1(Y)$, which depend on how precisely the maps $\underline{f}^+_{W'_N}$ are fixed.
\end{theorem}

\begin{proof} Let us start with the claim that $\Psi^+_{\m{t}, N}$ is an isomorphism.  It is clear from the definition that it is a chain map.  We also need to check that, as defined, $\Psi^+_{\m{t}, N}$ indeed does take $\underline{CF}^+(Y_N, \m{t}; M)$ to $\underline{C}^+_{\xi(\m{t})}(Y, K; M)$.   For any homotopy class $\psi$ of triangles in $\pi_2(\mathbf{x}, \Theta_{\ga\de}, \mathbf{y})$ representing $\m{s}_{K+}(\m{t})$, 
we have that 
$$ E_{K,N}(\psi) = \underline{\m{s}}_{w,z}(\mathbf{y}) + \big(n_w(\psi) - n_z(\psi)\big)\cdot\mbox{PD}[\mu] = E_{K,N}\big(\m{s}_{K+}(\m{t})\big) = \xi(\m{t}). $$
Then, according to Equation \ref{eq:11}, $[\mathbf{y}, i-n_w(\psi), i-n_z(\psi)] \in \m{T}(\xi(\m{t}))$.  Of course, twisting plays no role here.

For each point $\mathbf{x} \in \mathbb{T}_{\al} \cap \mathbb{T}_{\ga}$ representing $\m{t}$, we have a canonical smallest triangle $\psi_{\mathbf{x}} \in \pi_2(\mathbf{x}, \Theta_{\ga\de}, \mathbf{x}')$ supported in the winding region.  According to Proposition \ref{thm:3.4}, assuming $N$ is sufficiently large, this triangle satisfies $\m{s}_w(\psi_{\mathbf{x}}) = \m{s}_{K+}(\m{t})$ and $\m{s}_z(\psi_{\mathbf{x}}) = \m{s}_{K-}(\m{t})$.

Let $\Psi^+_0$ be the map which takes $[\mathbf{x}, i] \otimes m$ to $[\mathbf{x}', i - n_w(\psi_{\mathbf{x}}), i - n_z(\psi_{\mathbf{x}}) ] \otimes e^{h'(\psi_{\mathbf{x}})} \otimes m$.  This map is an isomorphism, owing to the fact that only one of $n_w(\psi_{\mathbf{x}})$ and $n_z(\psi_{\mathbf{x}})$ is nonzero (indeed, this is why we need to work with $\underline{CFK}\{i \geq 0 \mbox{ or }j\geq 0\}$);  twisting also has to be considered here, but ends up having no effect due to the canonical isomorphism $H^1(Y) \cong H^1(Y_N)$.  Since $\#\mathcal{M}(\psi_{\mathbf{x}}) = 1$, the map $\Psi^+_0$ is also a summand of $\Psi^+_{\m{t}, N}$.  In fact, making the area of the winding region sufficiently small, every other summand of $\Psi^+_{\m{t}, N}$ will be of lower order with respect to the energy filtration, as defined in \cite{OSHD}.  The argument from that paper then applies, showing that $\Psi^+_{\m{t}, N}$ must also be an isomorphism.  Hence, it is clear that the top square commutes, for both the top and bottom horizontal maps count precisely the same triangles, with coefficients differing by a constant factor in $H^1(Y)$.  

Switching the basepoints $w$ and $z$, we can say the same for the bottom square: in the context of this diagram, the map $\Psi^+_{\m{t}, N}$ counts the same triangles, but with respect to the basepoint $z$ these triangles represent $\m{s}_{K-}(\m{t})$ instead, again by Proposition \ref{thm:3.4}.  (Of course, $w$ and $z$ lie in the same component of $\Sigma \setminus \bsa \setminus \bsg$, so that $\m{s}_w(\mathbf{x}) = \m{s}_w(\mathbf{x})$ for $\mathbf{x} \in \mathbb{T}_{\al} \cap \mathbb{T}_{\ga}$.) \end{proof}

\subsection{The K\"unneth Formula} 
Finally, we will need a version of the K\"unneth formula for behavior under connect sums.  Recall that given two relative \sst\ structures $\xi_i \in \underline{\mbox{Spin}}^{\mbox{\scriptsize{c}}}(Y_i, K_i)$ for $i = 1,2$, there is a notion of the connect sum $\xi_1 \# \xi_2 \in \underline{\mbox{Spin}}^{\mbox{\scriptsize{c}}}(Y_1 \# Y_2, K_1 \# K_2)$.  To define this, fill in $\xi_i$ to get a particular vector field on $Y_i$, which is tangent to $K_i$.   Consider points $p_i \in K_i$, and take sufficiently small balls $B_i$ around $p_i$ so that the filled in vector fields are normal to $\partial B_i$ at two points, one ``going in'', the other ``going out''.  Now, remove these balls and glue the complements together along the boundaries so that the vector fields match up; remove a small neighborhood of $K_1 \# K_2$ to get $\xi_1 \# \xi_2$. 

The gluing is equivariant with respect to the inclusion maps on relative second cohomology; e.g., if $a \in H^2(Y_1, K_1)$, and $i^*: H^2(Y_1, K_1) \rightarrow H^2(Y_1 \# Y_2, K_1 \# K_2)$ is induced by inclusion, then $(\xi_1 + a) \# \xi_2 = (\xi_1 \# \xi_2) + i^*(a)$.

Of course, the connect sum of two knots equipped with reference longitudes is canonically equipped with a longitude.  With respect to this, we have the following.

\begin{lemma} \label{thm:5.3} $q_{K_1 \# K_2}(\xi_1 \# \xi_2) = q_{K_1}(\xi_1) + q_{K_2}(\xi_2)$.
\end{lemma}

\begin{proof}  In the definition of $q_K$, Equation \ref{eq:7}, both terms add under connect sums.  That the first term adds can be seen by explicit drawing a diagram for the connect sum; that the second term adds is straightforward. \end{proof}

Note that $H^1(Y_1 \# Y_2) \cong H^1(Y_1) \oplus H^1(Y_2)$ canonically.  Thus, if $M_i$ is a module over $\mathbb{Z}[H^1(Y_i)]$ for $i= 1,2$, then $M_1 \otimes_{\mathbb{Z}} M_2$ is a module over $\mathbb{Z}[H^1(Y_1)] \otimes \mathbb{Z}[H^1(Y_2)] \cong \mathbb{Z}[H^1(Y_1\#Y_2)]$.

\begin{theorem}
\label{thm:5.4}  For rationally null-homologous knots $K_1 \subset Y_1$ and $K_2 \subset Y_2$, and relative \sst\ structures $\xi_1 \in \underline{\mbox{Spin}}^{\mbox{\scriptsize{c}}}(Y_1, K_1)$ and $\xi_2 \in \underline{\mbox{Spin}}^{\mbox{\scriptsize{c}}}(Y_2, K_2)$, there is a $\mathbb{Z} \oplus \mathbb{Z}$-filtered chain homotopy equivalence of complexes over $\mathbb{Z}[U] \otimes \mathbb{Z}[H^1(Y_1\#Y_2)]$
$$\underline{CFK}^{\infty}(Y_1, K_1, \xi_1; M_1) \otimes_{\mathbb{Z}[U]} \underline{CFK}^{\infty}(Y_2, K_2, \xi_2; M_2) \rightarrow \underline{CFK}^{\infty}(Y_1 \# Y_2, K_1 \# K_2, \xi_1 \# \xi_2; M_1 \otimes_{\mathbb{Z}} M_2).$$
\end{theorem}

\begin{proof} This is a routine adaptation of the argument from \cite{OSKI}. \end{proof}

\section{Twisted Surgery Formula}
We now combine the results of the previous two sections, to achieve results akin to Theorem 6.1 of \cite{OSRS}.  

Let $K$ be a special knot, and $\m{t}_0 \in \spc{Y_0}$ an $\mu$-torsion \sst\ structure.Choose some $\m{t}_{\infty} \in \m{S}_{\infty}(\m{t}_0)$, and write $\xi_i$ for $[\m{t}_{\infty} - i\mbox{PD}[K], -\frac{\langle c_1(\m{t}_0), [\widehat{dS}] \rangle}{d} ] \in \rspc{Y, K}$. 

Define the map 
$$\underline{f}^+_{K, \m{t}_0}: \bigoplus_{i \in \mathbb{Z}} \underline{C}^+_{\xi_i}(Y,K) \rightarrow \bigoplus_{i \in \mathbb{Z}} \underline{CF}^+\big(Y, G_K(\xi_i) \big) $$
so that $x \in \underline{C}^+_{\xi_i}(Y,K)$ goes to
$$ v_{\xi_i}(x)  \oplus h_{\xi_i}(x) \in \underline{CF}^+\big(Y, G_K(\xi_i) \big) \oplus \underline{CF}^+\big(Y, G_K(\xi_{i+1}) \big).$$
Even though $G_K(\xi_i) = G_K(\xi_{i+d})$, we treat the corresponding summands as distinct.

\begin{theorem} \label{thm:6.1} There is a quasi-isomorphism from $M(\underline{f}^+_{K, \m{t}_0})$ to $\underline{CF}^+(Y_0, \m{t}_0)$. Furthermore, $M(\underline{f}^+_{K, \m{t}_0})$ admits a relative $\mathbb{Z}$-grading and a $U$-action which the quasi-isomorphism respects.
\end{theorem}

We prove this in a number of steps.  For the first, we set some notation.  

Fix a family $\mathcal{F}$ of standard diagrams for $(Y,K)$.  For the rest of this section, every chain complex we speak of hereafter will be isomorphic to one calculated from an element of $\mathcal{F}$.  

For $\m{t} \in \spc{Y_N}$, we write 
$$ \m{s}^k_+(\m{t}) \equiv \m{s}_{K+}(\m{t}) - k\cdot\mbox{PD}[F']|_{W'_N},\mbox{ }
\m{s}^k_-(\m{t}) \equiv \m{s}_{K-}(\m{t}) + k\cdot\mbox{PD}[F']|_{W'_N}$$
for $k \geq 0$, where $\m{s}_{K\pm}(\m{t}) \in \spc{W'_N}$ are as described in Section 2.5; when $k=0$, we drop the superscript.
We can write 
$$\underline{f}^+ = \sum_{\m{t} \in \m{S}_N} \sum_{k \geq 0} \big(\underline{f}^+_{\m{s}^k_+(\m{t})} + \underline{f}^+_{\m{s}^k_-(\m{t})}\big),$$ 
with notation as in the statement of Theorem \ref{thm:4.5}.  Then define
$$\underline{f}^+_* = \sum_{\m{t} \in \m{S}_N} \big(\underline{f}^+_{\m{s}_{+}(\m{t})} + \underline{f}^+_{\m{s}_{-}(\m{t})}\big),$$ 
i.e. the summands of $\underline{f}^+$ for $k=0$.

Proposition \ref{thm:2.4} says that under the assumption that $\m{t}_0$ is $\mu$-torsion, $\m{S}_{\infty}(\m{t}_0)$ and $\m{S}_N(\m{t}_0)$ consist of torsion structures, so that
$\underline{CF}^+\big(Y_N, \m{S}_N(\m{t}_0)\big)$ and $\underline{CF}^+\big(Y, \m{S}_{\infty}(\m{t}_0)\big)$ come equipped with absolute $\mathbb{Q}$-gradings, which extend to $\underline{CF}^+\big(Y_N, \m{S}_N(\m{t}_0)\big)\otimes \mathbb{Z}[\mathbb{Z}]$ and $\underline{CF}^+\big(Y, \m{S}_{\infty}(\m{t}_0)\big)\otimes \mathbb{Z}[\mathbb{Z}]$.  Furthermore, the usual untwisted grading shift formula still holds, so that if $x$ is a homogenous element of $\underline{CF}^+\big(Y_N, \m{t} \big) \otimes \mathbb{Z}[\mathbb{Z}]$, then
\begin{equation}
\label{eq:12}
\widetilde{\mbox{gr}}\big(\underline{f}^+_{\m{s}}(x)\big) - \widetilde{\mbox{gr}}(x) = \frac{c_1^2(\m{s}) - 2\chi(W'_N) - 3\sigma(W'_N)}{4} = \frac{c_1^2(\m{s}) + 1}{4}
\end{equation}
if $N$ is large enough so that $W'_N$ is negative definite.

\begin{prop} \label{thm:6.2} Fix an integer $\de$.  Then for all sufficiently large $N$, if $U^{\de}x = 0$ for $x$ in $\underline{CF}^+\big(Y_N, \m{S}_N(\m{t}_0)\big) \otimes \mathbb{Z}[\mathbb{Z}]$, then $\underline{f}^+_*(x) = \underline{f}^+(x)$. 
\end{prop}

\begin{proof}  Choose a particular value of $N$.  Let us find sufficient conditions so that $\underline{f}^+_*(x) = \underline{f}^+(x)$ when $U^{\de}x = 0$, in terms of $N$.

Take $x$ to be a homogenous generator of $\underline{CF}^+(Y_N, \m{t}) \otimes \mathbb{Z}[\mathbb{Z}]$ which satisfies $U^{\de}x = 0$, where $\m{t} \in \m{S}_N(\m{t}_0)$ and the chain complex is constructed from a $\m{t}$-proper diagram.  If $x$ itself is nonzero, then Theorem \ref{thm:5.2} guarantees in particular that  $\underline{f}^+_{\m{s}_+(\m{t})}(x) \ne 0$.  Of course, we also have $U^{\de}\underline{f}^+_{\m{s}_+(\m{t})}(x) = 0$.  Since there are a finite number of intersection points in $\mathbb{T}_{\al} \cap \mathbb{T}_{\de}$, it follows that there are two constants $L^-_{\infty}$ and $L^+_{\infty}$, depending only on ${\de}$ and on the family $\mathcal{F}$, such that
$$L^-_{\infty} \leq \widetilde{\mbox{gr}}\big(\underline{f}^+_{\m{s}_+(\m{t})}(x)\big) \leq L^+_{\infty}.$$
The same can be said of $\underline{f}^+_{\m{s}_-(\m{t})}(x)$.

Using the grading shift formula (\ref{eq:12}) and the definition of $Q_K$, we then have that
$$ \widetilde{\mbox{gr}}\big(\underline{f}^+_{\m{s}^k_+(\m{t})}(x)\big) = \widetilde{\mbox{gr}}\big(\underline{f}^+_{\m{s}_{+}(\m{t})}(x)\big) + Q_K\big(-k; \m{s}_{+}(\m{t})\big)$$
and
$$ \widetilde{\mbox{gr}}\big(\underline{f}^+_{\m{s}^k_-(\m{t})}(x)\big) = \widetilde{\mbox{gr}}\big(\underline{f}^+_{\m{s}_{-}(\m{t})}(x)\big) + Q_K\big(k; \m{s}_{-}(\m{t})\big)$$
for $k \geq 0$.  Then, a sufficient condition for $\underline{f}^+_*(x) = \underline{f}^+(x)$ is that $Q_K\big(-k; \m{s}_{+}(\m{t})\big)$ and $Q_K\big(k; \m{s}_{-}(\m{t})\big)$ are both less than $L^-_{\infty} - L^+_{\infty}$ for $k \ne 0$. 

Let us examine $Q_K\big(-k; \m{s}_{+}(\m{t})\big)$, which we can think of as a quadratic function of $k \in \mathbb{Q}$.  It is not hard to see that
$$ Q_K\big(-k; \m{s}_{+}(\m{t})\big) = k^2\frac{[\widetilde{dF'}]^2}{d^2} + k\Big(\frac{[\widetilde{dF'}]^2}{d^2} - q_K\big(\m{s}_{+}(\m{t})\big)\Big). $$
Note first that if $N>0$, then we have
$$ \frac{-\Big(\frac{[\widetilde{dF'}]^2}{d^2} - q_K\big(\m{s}_{+}(\m{t})\big)\Big)}{2\Big(\frac{[\widetilde{dF'}]^2}{d^2}\Big)} \leq \frac{-2N-C_q}{-2N} = 1 + \frac{-C_q}{-2N},$$
recalling Lemma \ref{thm:3.3}, Proposition \ref{thm:2.7}, and the fact that $\m{s}_{+}(\m{t})$ is represented by a small triangle.
Hence, for any value of $\epsilon > 0$, the value of $k \in \mathbb{Q}$ that maximizes $Q\big(-k; \m{s}_{+}(\m{t})\big)$ will be bounded above by $1 + \epsilon$ if our value of $N$ is large enough.  If our value of $N$ is large enough so that this holds with some $\epsilon < \frac{1}{2}$, we will then have 
$$ Q_K\big(-k; \m{s}_{+}(\m{t})\big) \leq Q_K\big(-1; \m{s}_{+}(\m{t})\big) $$
for all integers $k$ greater than 1.

We also have (for all $N$) that
$$ Q_K\big(-1; \m{s}_{+}(\m{t})\big) = 2\frac{[\widetilde{dF'}]^2}{d^2} - q_K\big(\m{s}_{+}(\m{t})\big) \leq C_q - N,$$
where the inequality again comes from Lemma \ref{thm:3.3} and Proposition \ref{thm:2.7}.  Hence, putting the two previous inequalities together yields
$$ Q_K\big(-k; \m{s}_{+}(\m{t})\big) \leq C_q - N$$
for $k$ an integer greater than 0 if $N$ is large enough.

Let us turn to the other function $Q_K\big(k; \m{s}_{-}(\m{t})\big)$.
In a completely analogous manner, we find that if our $N$ is large enough, then we will have 
$$ Q_K\big(k; \m{s}_{-}(\m{t})\big) \leq Q_K\big(1; \m{s}_{-}(\m{t})\big) $$
for all integers $k$ greater than 1.
To find an upper bound for $Q_K\big(1; \m{s}_{-}(\m{t})\big)$, recall that this is equal to
$$ \frac{c^2_1(\m{s}^1_-(\m{t})\big) - c^2_1(\m{s}_{-}(\m{t})\big)}{4} = \frac{c^2_1(\m{s}^1_-(\m{t})\big) - c^2_1(\m{s}_{+}(\m{t})\big)}{4} - \frac{c^2_1(\m{s}_{-}(\m{t})\big) - c^2_1(\m{s}_{+}(\m{t})\big)}{4} $$
$$\mbox{ } = Q_K\big(2; \m{s}_{+}(\m{t})\big) - Q_K\big(1; \m{s}_{+}(\m{t})\big).$$
The latter is equal to $q_K\big(\m{s}_w(\psi_{\m{t}})\big) + 2\frac{[\widetilde{dF'}]^2}{d^2}$, and we observed in the proof of Proposition \ref{thm:3.4} that
$$ q_K\big(\m{s}_w(\psi_{\m{t}})\big) + 2\frac{[\widetilde{dF'}]^2}{d^2} \leq C_q - N. $$
Hence, 
$$ Q_K\big(k; \m{s}_{-}(\m{t})\big) \leq Q_K\big(1; \m{s}_{-}(\m{t})\big) \leq C_q - N $$
for $k>0$ if $N$ is large enough.

Therefore, if our $N$ is sufficiently large, we will have
$$Q_K\big(-k; \m{s}_{+}(\m{t})\big) < L^-_{\infty} - L^+_{\infty}$$
and 
$$Q_K\big(k; \m{s}_{-}(\m{t})\big) < L^-_{\infty} - L^+_{\infty}$$ 
for $k \ne 0$.  There are a finite number of $\m{t}$ in $\m{S}_N(\m{t}_0)$; hence, for large enough $N$, it follows that $\underline{f}^+_*(x) = \underline{f}^+(x)$ for all $x \in \underline{CF}^+\big(Y_N, \m{S}_N(\m{t}_0)\big) \otimes \mathbb{Z}[\mathbb{Z}]$ that satisfy $U^{\de}x = 0$. \end{proof}

\begin{corollary} \label{thm:6.3} Write $\xi(\m{t})$ for $E_{K,N}\big(\m{s}_{K+}(\m{t})\big)$. 
Choose $\m{t}_N \in \m{S}_N(\m{t}_0)$, and let For all $\de \geq 0$, there is a quasi-isomorphism from
$M(\underline{f}^{(\de)}_{K, N, \m{t}_0})$ to $\underline{CF}^{\de}\big(Y_0, \m{S}^N_0(\m{t}_0)\big)$ are quasi-isomorphic for large enough $N$, 
where 
$$\underline{f}^{(\de)}_{K, N, \m{t}_0}: \bigoplus_{i \in \mathbb{Z}} \underline{C}^{\de}_{\xi(\m{t} + i\mathrm{\scriptsize{PD}}[F']|_{Y_N})}(Y,K) \rightarrow \bigoplus_{i \in \mathbb{Z}} \underline{CF}^{\de}\big(Y, G_K\left(\xi\left(\m{t} + i\mathrm{PD}[F']|_{Y_N})\right)\right) $$ 
is given by taking $x$ in summand $i$ to 
$$v_{\xi(\m{t} + i\mathrm{\scriptsize{PD}}[F']|_{Y_N})}(x) + h_{\xi(\m{t} + i\mathrm{\scriptsize{PD}}[F']|_{Y_N})}(x)$$
in summands $i$ and $i+1$.
\end{corollary}

\begin{proof} This follows from Proposition \ref{thm:6.2} and Corollary \ref{thm:4.6}, via Theorem \ref{thm:5.2}.
More precisely, given $\de$, choose $N$ large enough so that both Proposition \ref{thm:6.2} and Theorem \ref{thm:5.2} hold. In the commutative squares of Theorem \ref{thm:5.2}, we can clearly replace $+$ with $\de$ in all the groups, and we take $M$ to be $\mathbb{Z}[H^1(Y)] \otimes \mathbb{Z}[\mathbb{Z}]$.  Summing over \sst\ structures then gives a commutative square
$$\begin{CD}
\underline{CF}^{\de}\big(Y_N, \m{S}_N(\m{t}_0)\big) \otimes \mathbb{Z}[\mathbb{Z}] @>{\sum \big(\underline{f}^+_{W'_N, \m{s}_{K+}(\m{t})} + \underline{f}^+_{W'_N, \m{s}_{K-}(\m{t})}\big) \otimes \1}>> \underline{CF}^{\de}\big(Y, \m{S}_{\infty}(\m{t}_0)\big) \otimes \mathbb{Z}[\mathbb{Z}]\\
@V{\sum \Psi^+_{\m{t}, N} \otimes \1}VV      @VVV \\
\bigoplus \underline{C}^{\de}_{\xi(\m{t})}(Y, K) \otimes \mathbb{Z}[\mathbb{Z}] @>{\sum \big(v_{\xi(\m{t})} + h_{\xi(\m{t})}\big) \otimes \1}>> \underline{CF}^{\de}\big(Y, \m{S}_{\infty}(\m{t}_0)\big) \otimes \mathbb{Z}[\mathbb{Z}]
\end{CD}$$
where the vertical maps are isomorphisms of chain complexes, and each sum is taken over $\m{t} \in \m{S}_N(\m{t}_0)$.  Proposition \ref{thm:6.2} then says that replacing the upper horizontal map by $\underline{f}^+$ doesn't change the commutativity.  Hence, the mapping cone of the bottom is isomorphic to the mapping cone of $\underline{f}^+|_{\underline{CF}^{\de}\big(Y_N, \m{S}_N(\m{t}_0)\big) \otimes \mathbb{Z}[\mathbb{Z}]}$, which in turn is quasi-isomorphic to $\underline{CF}^{\de}\big(Y_0, \m{S}^N_0(\m{t}_0)\big)$ by Corollary \ref{thm:4.6}.  Finally, comparing the commutative square above with the last claim of Theorem \ref{thm:4.5} establishes the Corollary. \end{proof}

\noindent \textit{Proof of Theorem \ref{thm:6.1}.} When $N$ is large enough $N$, Proposition \ref{thm:2.9} show that $M(\underline{f}^{(\de)}_{K, N, \m{t}_0})$ and $M(\underline{f}^{(\de)}_{K, \m{t}_0})$ are quasi-isomorphic.  So combining this with Corollary \ref{thm:6.3}, we have for each value of $\de$ a quasi-isomorphism 
$$\psi^{(\de)}: M(\underline{f}^{(\de)}_{K, \m{t}_0}) \rightarrow \underline{CF}^{\de}\big(Y_0, \m{S}^N_0(\m{t}_0)\big)$$
for all large $N$.  The latter is just $\underline{CF}^{\de}(Y_0, \m{t}_0)$.  Specifically, since there are only finitely many \sst\ structures for which the homology is non-trivial, we can choose large $N$ so that $\m{S}^N_0(\m{t}_0) \setminus \{\m{t}_0\}$ doesn't contain any of them.
  
We need only replace each $\de$ with a $+$.  But, in fact, this can be done directly: it is straightforward to work on the chain level to show that the corresponding map 
$$\psi^+: M(\underline{f}^+_{K, \m{t}_0}) \rightarrow \underline{CF}^+\big(Y_0, \m{S}^N_0(\m{t}_0)\big)$$
is in fact injective and surjective on homology.  

Each summand of $\bigoplus_{i \in \mathbb{Z}} \underline{C}^+_{\xi_i}(Y,K)$ and $\bigoplus_{i \in \mathbb{Z}} \underline{CF}^+\big(Y, G_K(\xi_i) \big)$ admits a relative $\mathbb{Z}$-grading, as is always the case for twisted coefficient Floer homology: if two generators are connected by a disk, that disk is unique, and the grading difference is the Maslov index of this disk.  We can extend this to a relative $\mathbb{Z}$-grading on the entire mapping cone, by demanding that $\underline{f}^+_{K, \m{t}_0}$ lowers grading by one.  It is now easy to see that $\psi^+$ respects relative $\mathbb{Z}$-gradings, by simply inspecting Maslov indices of polygons in the diagram.   
There is also the $U$-action on the mapping cone induced by that of the summands; $\psi^+$ clearly respects this also, and the action lowers the relative $\mathbb{Z}$-grading by 2. \halbox

\section{Computations}
For the computations we are interested in, we will need to compute the twisted filtered chain complexes for two different families of knots:  the Borromean knots $B_g$ and the $O$-knots $O_{p,q}$.
The Borromean knot $B_1$ is gotten by performing $0$-surgery on two of the components of the Borromean link in $S^3$, thinking of the last component as a knot in $S^1 \times S^2 \# S^1 \times S^2$; the Borromean knot $B_g$ is the $g$-fold connect sum of $B_1$.
The knot $O_{p,q}$ is obtained by performing $-p/q$-surgery along one component of the Hopf link in $S^3$, and thinking of the other component as a knot in $L(p,q)$.   

We make computations for these knots, and then we proceed to our final results.

\subsection{The Borromean knots $B_1$}  
We first want to calculate $\underline{C}^+_{\xi}(\#^2 S^1 \times S^2, B_1)$ for $\xi \in \underline{\mbox{Spin}}^{\mbox{\scriptsize{c}}}(\#^2 S^1 \times S^2, B_1)$. As it will turn out, we need only consider those relative \sst\ structures in $G_{B_1}^{-1}(\m{t}_{B_1})$, where $\m{t}_{B_1}$ is the unique torsion structure on $\#^2 S^1 \times S^2$.  So, let $\xi^0_{B_1}$ be the unique structure in $G_{B_1}^{-1}(\m{t}_{B_1})$ that extends to a torsion \sst\ structure over $T^3$ (which is 0-surgery on $B_1$), and let $\xi^k_{B_1}$ denote $\xi^0_{B_1} + k\mbox{PD}[\mu]$.  

Our calculations are essentially carried out in \cite{JM}, but we want to detail the result.
We adapt the following Lemma from that paper.  Let $R$ denote the group ring $\mathbb{Z}[H^1(\#^2S^1 \times S^2)]$.

\begin{lemma} \label{thm:7.1} For $\xi \in \rspc{\#^2 S^1 \times S^2, B_1}$, we have
$$ \underline{\widehat{HFK}}(\#^2 S^1 \times S^2, B_1, \xi) \cong \left\{ 
\begin{array}{ll}
\Lambda^{1-k}_R\left(H_1(\#^2S^1\times S^2) \otimes_{\mathbb{Z}} R\right), & \xi = \xi^k_{B_1} \\
0, &\mbox{otherwise,} \\
\end{array} \right.$$
where $\Lambda^*\left(H_1(\#^2S^1\times S^2) \otimes_{\mathbb{Z}} R\right)$ is absolutely $\mathbb{Q}$-graded with $\Lambda^{i}\left(H_1(\#^2S^1\times S^2)\right)$ supported in grading level $1-i$.
Also,
$$ \underline{\widehat{HF}}(\#^2 S^1 \times S^2, \m{t}_{B_1}) \cong \mathbb{Z}_{(-2)}.$$
\end{lemma}

\begin{proof} Specifically, these are results of Section 9 of \cite{JM}.  Essentially, the assertions all follow from the observation that a long exact sequence analogous to the one for untwisted $\widehat{HFK}$ holds for $\underline{\widehat{HFK}}$, along with prior results regarding the untwisted analogues. \end{proof}

In Figure \myfig{4}, we exhibit an explicit weakly admissable Heegaard diagram for $(\#^2 S^1 \times S^2, B_1)$.  It is not hard to verify that there are four intersection points in our diagram that represent $\m{t}_{B_1}$, which we will denote $\mathbf{y}_1 =\{p_1, p_2,r\}, \mathbf{y}_2=\{p_1, q_2,r\}, \mathbf{y}_3=\{q_1, p_2,r\}$, and $\mathbf{y}_4=\{q_1, q_2,r\}$.  
There are four other generators representing different \sst\ structures, but it is not difficult to adjust this diagram to get a different weakly admissable diagram with these structures not represented. Hence, $\underline{C}^+_{\xi}(\#^2 S^1 \times S^2, B_1)$ vanishes for $\xi$ not equal to $\xi^k_{B_1}$ for some $k$.  
Of course, since $\m{t}_{B_1}$ is torsion, this diagram is also $\m{t}_{B_1}$-strongly admissable, and so going forth, we might as well just look at $\underline{C}_{\xi^k_{B_1}}(\#^2 S^1 \times S^2, B_1) = \underline{CFK}^{\infty}(\#^2 S^1 \times S^2, B_1, \xi^k_{B_1})$, instead of the quotient version.

\begin{figure}[t!]
\label{fig:4}
\centering \includegraphics[scale=.45]{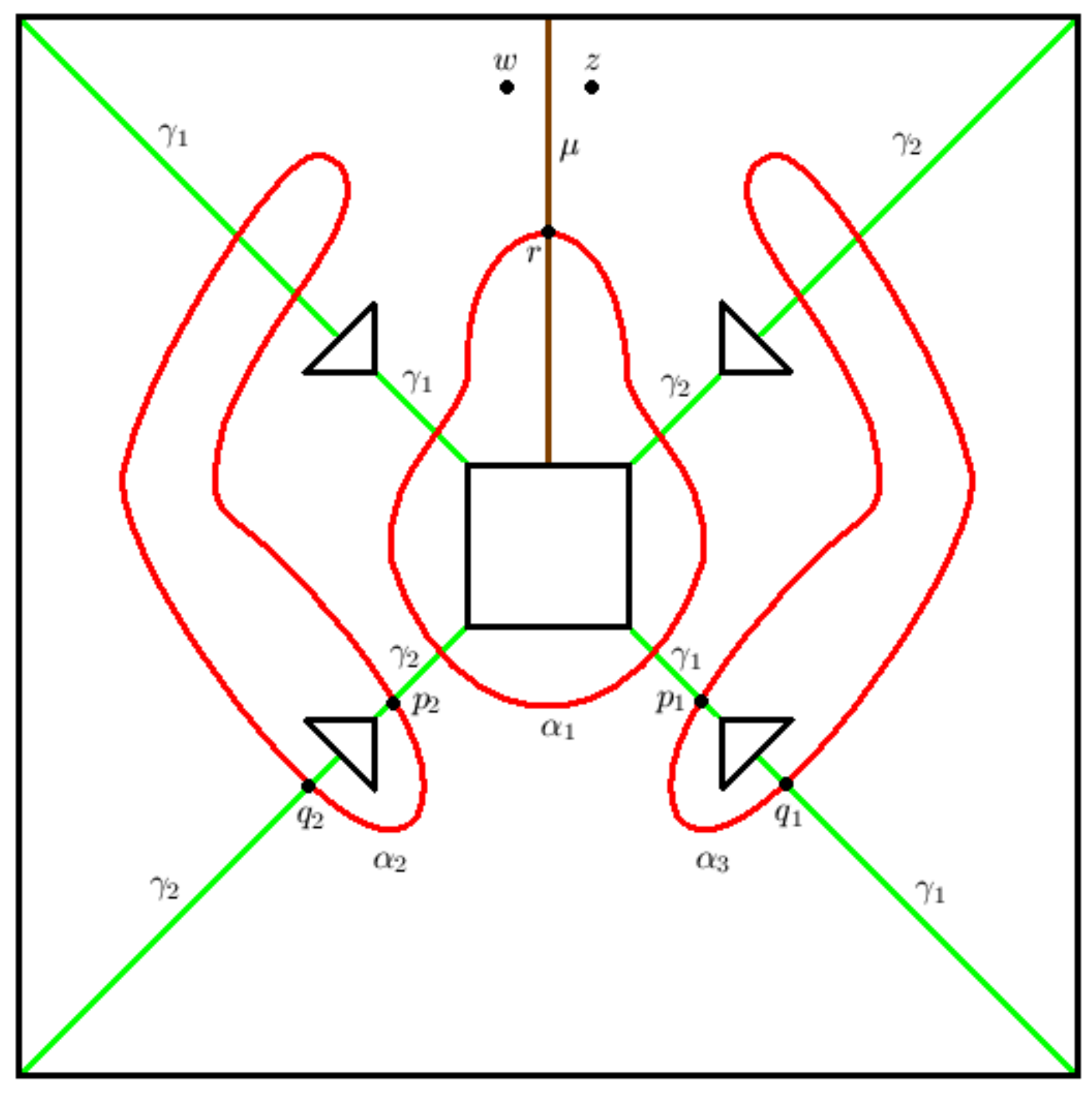}
\caption{An admissable doubly-pointed Heegaard diagram for $(\#^2 S^1 \times S^2, B_1)$.  The outer square and inner square are feet of a 1-handle, the upper left and lower right triangles are feet of another and the upper right and lower left triangles are feet of a third.  The points $p_1, q_1, p_2, q_2$ and $r$ are components of the generators $\mathbf{y}_1, \mathbf{y}_2, \mathbf{y}_3$, and $\mathbf{y}_4$.  }
\end{figure}

From the Lemma along with a quick examination of the diagram, it is clear that $\mathbf{y}_1, \mathbf{y}_2, \mathbf{y}_3$, and $\mathbf{y}_4$ have respective absolute $\mathbb{Q}$-gradings of $1, 0, 0$ and $-1$; and they therefore represent relative \sst\ structures $\xi^1_{B_1}, \xi^0_{B_1}, \xi^0_{B_1}$ and $\xi^{-1}_{B_1}$, respectively.  (In fact, the untwisted chain complex $\widehat{CF}(\#^2S^1\times S^2, \m{t}_{B_1})$ is generated over $\Lambda^*\left(H_1(\#^2S^1 \times S^2)\right)$ by $\mathbf{y}_1$.)  Given this, we can write
$$\underline{C}_{\xi^k_{B_1}}(\#^2 S^1 \times S^2, B_1) \cong \bigoplus_{(i,j) \in \mathbb{Z} \oplus \mathbb{Z}} \underline{C}^{i,j}_{\xi^k_{B_1}}(\#^2 S^1 \times S^2, B_1)$$
as a $\mathbb{Z} \oplus \mathbb{Z}$-graded group, with summands given by
$$ \underline{C}^{i,j}_{\xi^k_{B_1}}(\#^2 S^1 \times S^2, B_1) \cong \Lambda^{1+i-j-k}_R\left(H_1(\#^2S^1\times S^2) \otimes_{\mathbb{Z}} R\right).$$
Precisely, the summand $\underline{C}^{i,i-k+1}_{\xi^k_{B_1}}(\#^2 S^1 \times S^2, B_1)$
is a free module over $R$ generated by $[\mathbf{y}_1, i, i-k+1]$, lying in absolute $\mathbb{Q}$-grading $2i+1$; and the other summands are described analogously.

Now, note that we have a $\mathbb{Z}$-filtration of $\underline{C}_{\xi^k_{B_1}}(\#^2 S^1 \times S^2, B_1)$, with filtration level $\ell$ the direct sum of those $\underline{C}^{i,j}_{\xi^k_{B_1}}(\#^2 S^1 \times S^2, B_1)$ with $i+j \leq \ell$.  This gives rise to a Leray spectral sequence, which has $E^1$ term comprised of the group $\underline{\widehat{HFK}}(\#^2 S^1 \times S^2, B_1, \xi^k_{B_1} + (i-j)\mbox{PD}[\mu])$ in level $(i,j)$, and $E^{\infty}$ term the associated graded group of $\underline{HF}^{\infty}\big(\#^2 S^1 \times S^2, G_{B_1}(\xi^k_{B_1})\big)$.  

We can therefore easily see that $d_0$, the differential on the $E^0$ term of the sequence (or equivalently the portion of the differential that doesn't lower filtration level), must vanish, for otherwise the rank of $\underline{\widehat{HFK}}(\#^2 S^1 \times S^2, B_1)$ would be too small.  Then, observe that the $\mathbb{Q}$-grading on all of $\underline{C}^{i,j}_{\m{t}_k}(\#^2 S^1 \times S^2, B_1)$ is equal to $i+j+k$; thus, simply for grading reasons, the differentials $d_{\ell}$ for $\ell \geq 2$ all must vanish.  Thus, the only non-trivial differential in the spectral sequence is $d_1$.  
In particular, the differential on the filtered chain complex $\underline{C}_{\xi^k_{B_1}}(\#^2 S^1 \times S^2, B_1)$ can be written as $V + H$, where $V$ takes $\underline{C}^{i,j}_{\xi^k_{B_1}}(\#^2 S^1 \times S^2, B_1)$ to $\underline{C}^{i,j-1}_{\xi^k_{B_1}}(\#^2 S^1 \times S^2, B_1)$ and $H$ takes $\underline{C}^{i,j}_{\xi^k_{B_1}}(\#^2 S^1 \times S^2, B_1)$ to $\underline{C}^{i-1,j}_{\xi^k_{B_1}}(\#^2 S^1 \times S^2, B_1)$.

Let us look at $V$.  To do this, observe that there is also a spectral sequence from the induced filtration on $\underline{\widehat{CFK}}(\#^2 S^1 \times S^2, B_1, \xi^k_{B_1})$ which converges to the associated graded group of $\underline{\widehat{HF}}(\#^2 S^1 \times S^2, \m{t}_{B_1})$.  (The group $\underline{\widehat{CFK}}(\#^2 S^1 \times S^2, B_1, \xi^k_{B_1})$ can be viewed as the $i=0$ column of $\underline{C}_{\xi^k_{B_1}}(\#^2 S^1 \times S^2, B_1)$, or with any of the columns after minor adjustments.)  Of course, this spectral sequence also collapses after $E^1$.  The last claim of Lemma \ref{thm:7.1}, then, dictates that $V$ should be injective on 
$\underline{C}^{i,i-k+1}_{\xi^k_{B_1}}(\#^2 S^1 \times S^2, B_1)$, and that the image of $V|_{\underline{C}^{i,i-k+1}_{\xi^k_{B_1}}(\#^2 S^1 \times S^2, B_1)}$ should be exactly the kernel of $V|_{\underline{C}^{i,i-k}_{\xi^k_{B_1}}(\#^2 S^1 \times S^2, B_1)}$, for otherwise there would be elements of the wrong absolute grading living past $E^1$ in this spectral sequence. 
The fact that the $E^{\infty}$ term is rank one over $\mathbb{Z}$, supported entirely in grading $-2$, also implies that  $V|_{\underline{C}^{i,i-k}_{\xi^k_{B_1}}(\#^2 S^1 \times S^2, B_1)}$ has cokernel $\mathbb{Z}$ in $\underline{C}^{i,i-k-1}_{\xi^k_{B_1}}(\#^2 S^1 \times S^2, B_1)$.

It is clear that the remarks in the last paragraph apply equally to $H$, aside from changes in the notation.  To be precise, we wrap everything up in the following proposition.

\begin{prop} \label{thm:7.2} Let $\xi^0_{B_1} \in  \rspc{\#^2 S^1 \times S^2, B_1}$  be the relative \sst\ structure that induces torsion ones on $\#^2 S^1 \times S^2$ and $T^3$, and let $\xi^k_{B_1} = \xi^0_{B_1} + k\mathrm{PD}[\mu]$, where $\mu$ is the oriented meridian of $B_1$.  Then, $q_{B_1}(\xi^k_{B_1}) = 2k$.

Write $\Lambda^*$ for $\Lambda^*_R\left(H_1(\#^2S^1\times S^2) \otimes_{\mathbb{Z}} R \right)$.  We also have an isomorphism of $\mathbb{Z} \oplus \mathbb{Z}$-filtered modules
$$\underline{C}_{\xi}(\#^2 S^1 \times S^2, B_1) \cong 
\left\{ \begin{array}{ll}
\bigoplus_{i \in \mathbb{Z}} \Lambda^*\{i, i-k\}, & \xi = \xi^k_{B_1} \\
0, & \mbox{otherwise,} \\ 
\end{array} \right. $$
where $\Lambda^i$ is supported in filtration level $(0,1-i)$, and the curly braces denote the filtration level shift.

Under this identification, the differential on $\underline{C}_{\xi^k_{B_1}}(\#^2 S^1 \times S^2, B_1)$ is given by a map $V+H$, where $V$ and $H$ lower filtration level by $(0,1)$ and $(1,0)$ respectively, and act in the following manner.  There are exact sequences of $R$-modules
$$ 0 \rightarrow \Lambda^{0} \overset{V}{\rightarrow} \Lambda^{1} \overset{V}{\rightarrow} \Lambda^{2} \rightarrow \mathbb{Z} \rightarrow 0 $$
and
$$ 0 \rightarrow \Lambda^{2} \overset{H}{\rightarrow} \Lambda^{1} \overset{H}{\rightarrow} \Lambda^{0} \rightarrow \mathbb{Z} \rightarrow 0 $$
such that $V|_{\Lambda^{2}}$ and $H|_{\Lambda^{0}}$ are both zero maps.
\end{prop}

\begin{proof} That $q_{B_1}(\xi^0_{B_1}) = 0$ is clear, and thus $q_{B_1}(\xi^k_{B_1}) = 2k$ holds for all $k$.  The sequences are not explicitly mentioned above, but they are clear from the discussion.  
\end{proof}

Denote by $C_k$ the particular chain complex $\underline{C}_{\xi}(\#^2 S^1 \times S^2, B_1)$ given above, with $\xi = \xi^k_{B_1}$.

Consider the chain complex $C'_0$, for the structure $\xi^0_{B_1}$, obtained when we switch the basepoints $w$ and $z$.  Obviously, $C'_0$ is chain homotopy equivalent to $C_0$; the algebraic structure of $C'_0$ is clearly equivalent to that of $C_k$, except with filtration levels switched.  So, for example, in filtration level $(0,1)$, $C_0$ has the group generated by $[\mathbf{y}_1,0,1]$, while $C'_0$ has the group generated by $[\mathbf{y}_4,0,1]$; in filtration level $(1,0)$, $C_0$ has the group generated by $[\mathbf{y}_4,1,0]$, while $C'_0$ has the group generated by $[\mathbf{y}_1,1,0]$.  By a minor abuse of notation, it makes sense to call the portion of $C'_0$ in level $(i,j)$ by the same name $\Lambda^{1+j-i}$ as the portion of $C_0$ in level $(j,i)$.  We also retain the names of the components of the differential, although now $V$ lowers filtration level by $(1,0)$ instead of by $(0,1)$, and vice-versa for $H$. 

Let $\varphi:C_0 \rightarrow C'_0$ be a filtered chain homotopy equivalence.  The restriction of $\varphi$ to an individual filtration level is easily seen to depend not on the precise filtration level, but only on the groups $\Lambda^i$ and $\Lambda^{2-i}$ lying in the filtration level in each of the complexes.  So let $*: \Lambda^i \rightarrow \Lambda^{2-i}$ be the appropriate restriction of $\varphi$.  We have in particular the commutative diagram
$$\begin{array}{lclclclclcl}
0 &
\longrightarrow &
\Lambda^0 &
\overset{V}{\longrightarrow} &
\Lambda^1 &
\overset{V}{\longrightarrow} &
\Lambda^2 &
\longrightarrow &
\mathbb{Z} &
\longrightarrow &
0 \\
 &
 &
* \downarrow &
 &
* \downarrow &
 & 
* \downarrow &
 &
\downarrow &
 &  \\
0 &
\longrightarrow &
\Lambda^2 &
\overset{H}{\longrightarrow} &
\Lambda^1 &
\overset{H}{\longrightarrow} &
\Lambda^0 &
\longrightarrow &
\mathbb{Z} &
\longrightarrow &
0\\
\end{array}$$
where the rightmost horizontal arrow is an isomorphism, and the rows are the exact sequences of Proposition \ref{thm:7.2}.

For general $k$, we have similar remarks, except that they relate $C_k$ with a chain complex $C'_{-k}$ for the structure $\xi^{-k}_{B_1}$.  At the point in the discussion where they are relevant, we can use precisely the same maps $*$ as before.

\subsection{$O$-knots}  For the $O$-knot, of course, non-trivial twisting doesn't occur, as the ambient manifolds are rational homology spheres.  Nonetheless, we take a close look at them, since we want to carefully write down what relative \sst\ structures the generators lie in.   

Let $K=O_{p,q}$; we restrict to the case $p, q > 0$.
In Figure \myfig{5}, we depict a standard doubly-pointed Heegaard triple for $K$, equipped with the $0$-framing as the longitude $\la$ (i.e., surgery along this longitude is the same as surgering the Hopf link in $S^3$ with coefficents $-\frac{p}{q}$ and $0$).  Hence, $Y_{\al\de}$ is $L(p,q)$; and, for ease of computation, we choose $\ga$ so that $Y_{\al\ga}$ is surgery on $K$ with framing $\mu + \la$.  We write $W'_K$ for the cobordism $X_{\al\ga\de}$ filled in by $B^4$ along $Y_{\ga\de}$.  This cobordism can also be described as the orientation reversal of the cobordism obtained by attaching a 1-framed 2-handle to $K$.  
We fix orientations for the circles, and label the points of $\mathbb{T}_{\al} \cap \mathbb{T}_{\de}$ by $\mathbf{x}_K(n), n = 1, \ldots, p$, as shown.

\begin{figure}[t!]
\label{fig:5}
\centering \includegraphics[scale=.70]{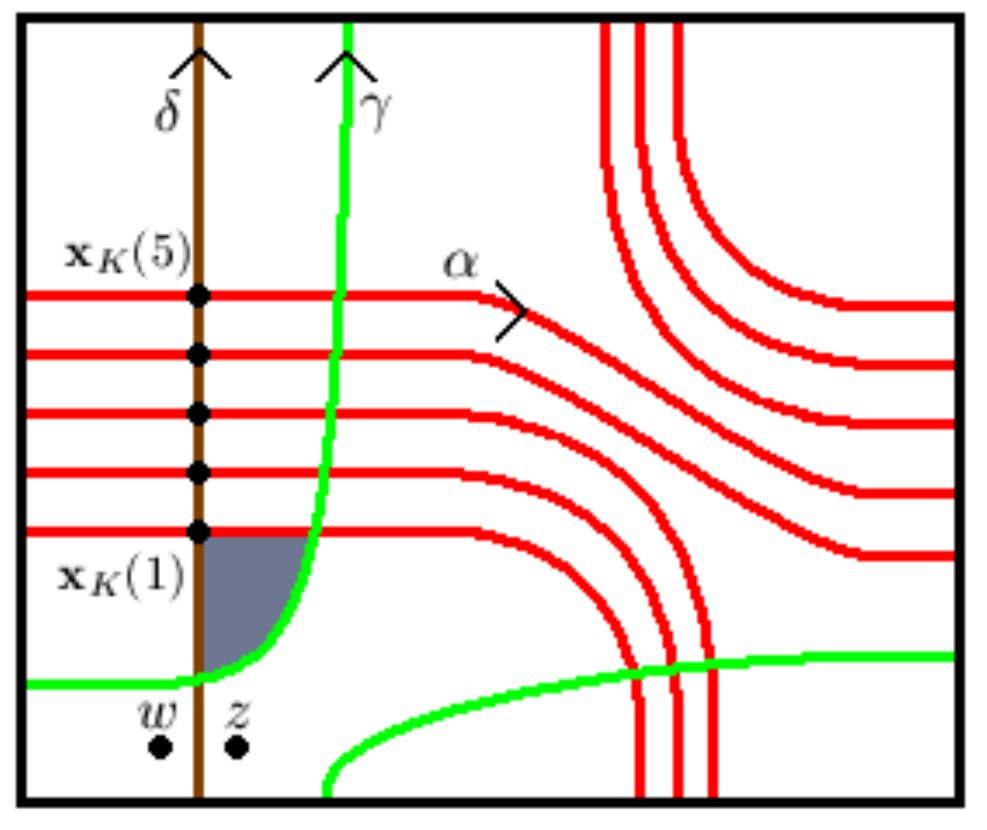}
\caption{A standard Heegaard diagram for $W'_K$, where $K=O_{5,3}$.  The marked points are $\mathbf{x}_K(n)$ for $n=1, \ldots, 5$.  The domain of the triangle $\psi_n$ is shaded in.}
\end{figure}

Let $\mathcal{P}$ be the periodic domain such that $\partial \mathcal{P} = -\al + p\ga - (p+q)\de$ and $n_w(\mathcal{P}) = 0$.  The placement of basepoints specifies an orientation of $K$; with respect to this, the class $[\widetilde{dF'}]$ which generates $H_2(W'_K)$ corresponds to this domain.

For each $n$, we have a small triangle $\psi_n \in \pi_2(\mathbf{y}, \Theta_{\ga\de}, \mathbf{x}_K(n))$ through one of the points $\mathbf{y} \in \mathbb{T}_{\al} \cap \mathbb{T}_{\ga}$, which has multiplicity 0 at both basepoints.  We calculate using Proposition \ref{thm:2.3} that 
$$\langle c_1(\underline{\m{s}}_{w,z}(\mathbf{x}_K(n)), \phi_*^{-1}([\widetilde{dF'}]) \rangle = 2p - 2n + 1,$$
recalling that we have an isomorphism $\phi_*: H_2(Y,K) \rightarrow H_2(W'_K)$.
If $[\mu] \in H_1(L(p,q) \setminus K)$ is the homology class of an oriented meridian of $K$, then $\langle \mbox{PD}[\mu], \phi_*^{-1}([\widetilde{dF'}]) \rangle = p$, thinking of $\mbox{PD}[\mu]$ as an element of $H^2(L(p,q), K)$.   Thus, 
$$\langle c_1\big(\underline{\m{s}}_{w,z}(\mathbf{x}_K(n)) + m\mbox{PD}[\mu]\big), \phi_*^{-1}([\widetilde{dF'}]) \rangle = 2(m+1)p - 2n + 1.$$

We can also identify $\m{s}_w(\psi_n)$ using the Chern class evaluation formula (\ref{eq:12}); we get that
$$\langle c_1(\m{s}_w(\psi_n)), [\widetilde{dF'}] \rangle = 2p + q -2n + 1.$$
Furthermore, the initial defining Equation \ref{eq:5} of $E_{K,N}$ shows that $E_{K,1}\big(\m{s}_w(\psi_n)\big) = \um{s}_{w,z}\big(\mathbf{x}_K(n)\big)$; then Proposition \ref{thm:2.6} shows that 
$$E_{K,1}\big(\m{s}_w(\psi_n) + m{\phi^*}^{-1}(\mbox{PD}[\mu])\big) = \um{s}_{w,z}\big(\mathbf{x}_K(n)\big) + m\mbox{PD}[\mu].$$

With respect to the orientation on $O_{p,q}$, it is not difficult to see that $\kappa = \frac{q}{p}$, so that $\frac{[\widetilde{dF'}]^2}{d^2} = -\frac{q}{p} - 1$.  Hence, 
$$q_K\big(\um{s}_{w,z}\big(\mathbf{x}_K(n)\big) + m\mbox{PD}[\mu] \big) = -\frac{q}{p} - 1 + \frac{2p + q -2n + 1 + 2pm}{p} = \frac{p - 2n + 1}{p} + 2m.$$ 

In addition, observe that $\um{s}_{w,z}\big(\mathbf{x}_K(1)\big) - \um{s}_{w,z}\big(\mathbf{x}_K(q+1)\big) = \mbox{PD}[\la]$.  From this, it is easy to extract that 
$$q_K(\xi + \mbox{PD}[\la]) = q_K(\xi) + \frac{2q}{p}.$$

Each of the intersection points in $\mathbb{T}_{\al} \cap \mathbb{T}_{\ga}$ lies in a different absolute \sst\ structure; hence we infer that all differentials in the complex vanish.  So, it is straightforward to manipulate the above to get the following summation.

\begin{prop} \label{thm:7.3} Take $p, q > 0$, and let $K = O_{p,q}$, equipped with the $0$-framed longitude $\la$ (as described above), with oriented meridian $\mu$.
Then for each $r \in \mathbb{Z}$, there is precisely one relative \sst\ structure $\xi^r_{p,q}$ with $q_K(\xi^r_{p,q}) = \frac{2r - p - 1}{p}$, and \rspc{L(p,q), K} is composed of precisely these structures.  (So, denoting $G_K(\xi^r_{p,q})$ by $\m{t}^r_{p,q}$, we have that $\m{t}^{r+p}_{p,q} = \m{t}^{r}_{p,q}$, and of course 
$\xi^r_{p,q} =  \big[\m{t}^r_{p,q}, \frac{2r - p - 1}{p}\big].$)
Furthermore, we have
$$\underline{C}_{\xi^r_{p,q}}(L(p,q), K) \cong \mathbb{Z}[U, U^{-1}],$$ 
generated over $\mathbb{Z}$ by a single generator in grading $\left(i, i - \lfloor \frac{p-r}{p} \rfloor \right)$ for each integer $i$; and 
$$ \xi^r_{p,q} + \mathrm{PD}[\la] = \xi^{r+q}_{p,q}.$$

\subsection{The full filtered complex}  
Now, we describe the chain complex we are ultimately interested in, that of $B_g \#_{\ell =1}^n O_{p_{\ell}, q_{\ell}}$.  Denote this knot by $K$, and let $Y$ be the ambient manifold for $K$.  

Denote $\xi^j_{B_1} \#^g_{i=2} \xi^0_{B_1}$ by $\xi^j_{B_g}$; and let $\xi(j; r_1, \ldots, r_n) = \xi^j_{B_g} \#_{\ell =1}^n \xi^{r_{\ell}}_{p_{\ell}, q_{\ell}}$.  Write $\m{t}(r_1, \ldots, r_n)$ for $G_K\big(\xi(j; r_1, \ldots, r_n)\big).$ Also, let 
$$\eta\big(\xi(j; r_1, \ldots, r_n)\big) = j - \sum_{\ell = 1}^n \lfloor \frac{p_{\ell} - r_{\ell}}{p_{\ell}} \rfloor.$$  
\end{prop}

\begin{prop} \label{thm:7.4} Take $Y$ and $K$ as above.  Let $R = \mathbb{Z}[H^1(\#^{2g}S^1 \times S^2)]$, and write $\Lambda^*$ for $\Lambda^*_R\left(H_1(\#^{2g}S^1\times S^2) \otimes_{\mathbb{Z}} R \right)$.
If $\xi = \xi(j; r_1, \ldots, r_n)$, then there is an isomorphism of $\mathbb{Z} \oplus \mathbb{Z}$-filtered modules
$$\underline{C}_{\xi}(Y, K) \cong \bigoplus_{i \in \mathbb{Z}} \Lambda^*\{i, i-\eta(\xi)\} $$
where $\Lambda^i$ is supported in filtration level $(0, g-i)$, and the curly braces denote the filtration level shift. If $\xi$ is not of the above form, then $\underline{C}_{\xi}(Y, K)$ is trivial. 
Under this identification, the differential on $\underline{C}_{\xi}(Y, K)$ is given by a map $V+H$, where $V$ and $H$ lower filtration level by $(1,0)$ and $(0,1)$ respectively, and act in the following manner.  There is a map $*: \Lambda^i \rightarrow \Lambda^{2g-i}$ (depending only on $g$ and $i$, not on $Y$ and $K$) and a commutative  diagram of $R$-modules with exact rows,
$$\begin{array}{lclclclclclclcl}
0 &
\longrightarrow &
\Lambda^{0} &
\overset{V}{\longrightarrow} &
\Lambda^{1} &
\overset{V}{\longrightarrow} &
\ldots &
\overset{V}{\longrightarrow} &
\Lambda^{2g-1} &
\overset{V}{\longrightarrow} &
\Lambda^{2g} &
\longrightarrow &
\mathbb{Z} &
\longrightarrow &
0 \\
 &
 &
* \downarrow &
 &
* \downarrow &
 &
 &
 & 
* \downarrow &
 &
* \downarrow &
 &
\downarrow &
 & \\
0 &
\longrightarrow &
\Lambda^{2g} &
\overset{H}{\longrightarrow} &
\Lambda^{2g-1} &
\overset{H}{\longrightarrow} &
\ldots &
\overset{H}{\longrightarrow} &
\Lambda^{1} &
\overset{H}{\longrightarrow} &
\Lambda^{0} &
\longrightarrow &
\mathbb{Z} &
\longrightarrow &
0 \\
\end{array}$$
such that the rightmost horizontal arrow is an isomorphism, and $V|_{\Lambda^{2g}}$ and $H|_{\Lambda^{0}}$ are both zero maps.
\end{prop}

\begin{proof}  First, consider the case where $n=0$, that is, when we are looking at one of the Borromean knots $B_g$.  The K\"unneth formula then shows that the groups of our chain complex are as described, and that the differential still is the sum of a horizontal and vertical differential.  
The claimed exactness of the rows of the diagram can be verified by an induction argument on $g$.  This involves some diagram chasing, and utilizes the fact that the exterior algebras are free $R$-modules.  We omit the details.  

The maps $*$ are components of a filtered chain homotopy equivalence between a chain complex of the form described and a chain complex obtained by switching basepoints, similar to those described after the statement of Proposition \ref{thm:7.2}.  The discussion there carries over to show the claims stated here.

When $n \ne 0$, the computation for $O$-knots shows that the only affect that these have on the complex is to shift gradings.   \end{proof}

Now, take $\m{t}_0 \in \spc{Y_0}$, and let us assume that $\m{t}_0$ is $\mu$-torsion.  Suppose that $\m{t}_0$ extends over $W_0$ to a \sst\ structure $\m{s} \in \spc{W_0}$ for which 
$\m{s}|_Y = \m{t}(r_1, \ldots, r_n)$.  

We have $q_K(\m{s}) = -\frac{\langle c_1(\m{t}_0), [\widehat{dS}] \rangle}{d}$.  Therefore, the relative \sst\ structure $\xi_i$ used in Theorem \ref{thm:6.1} will be given by
$\left[\m{s}|_Y - i\mbox{PD}[K], -\frac{\langle c_1(\m{t}_0), [\widehat{dS}] \rangle}{d}\right]$.  It is easy to see that we can also write $\xi_i$ as $\xi(j_i; r_1 - iq_1, \ldots, r_n - iq_n)$ for some $j_i$; this value must satisfy the equation
$$ 2j_i + \sum_{\ell = 1}^n \frac{2(r_{\ell} - iq_{\ell}) - p_{\ell} - 1}{p_{\ell}} = -\frac{\langle c_1(\m{t}_0), [\widehat{dS}] \rangle}{d}.$$
Hence,
$$ j_i = -\frac{\langle c_1(\m{t}_0), [\widehat{dS}] \rangle}{2d} + \sum_{\ell = 1}^n \frac{1 - p_{\ell}}{2p_{\ell}} + \frac{p_{\ell} - (r_{\ell} - iq_{\ell})}{p_{\ell}}.$$

Let $\eta(i) = \eta(\xi_i)$. We compute
$$\begin{array}{lcl}
 \eta(i) & = & -\frac{\langle c_1(\m{t}_0), [\widehat{dS}] \rangle}{2d} + \sum_{\ell = 1}^n \frac{1 - p_{\ell}}{2p_{\ell}} + \frac{p_{\ell} - (r_{\ell} - iq_{\ell})}{p_{\ell}} - \left\lfloor \frac{p_{\ell} - (r_{\ell} - iq_{\ell})}{p_{\ell}} \right\rfloor \\
 & = & -\frac{\langle c_1(\m{t}_0), [\widehat{dS}] \rangle}{2d} + \sum_{\ell = 1}^n \frac{1 - p_{\ell}}{2p_{\ell}} + \left\{\frac{iq_{\ell} - r_{\ell}}{p_{\ell}}\right\},
\end{array} $$
where the curly braces denote the fractional part, $\{x\} = x - \lfloor x \rfloor$.

\section{Proof of Main Theorem}
We now only have to compute the homology of the complex of Proposition \ref{thm:7.4}.

For simplicity, let us work abstractly.  So, for a nonnegative integer $g$, let us suppose that we have

\begin{itemize}

\item $2g + 1$ $R$-modules $M_0, M_1, \ldots, M_{2g}$ over a ring $R$;

\item maps $V: M_i \rightarrow M_{i-1}$ and $H: M_i \rightarrow M_{i+1}$ for $0 \leq i \leq 2g$, where $M_{-1}$ and $M_{2g+1}$ are taken to be trivial; and

\item a map $*: M_i \rightarrow M_{2g-i}$ for $0 \leq i \leq 2g$.

\end{itemize}

\noindent (So, $R$ plays the role of $\mathbb{Z}[H^1(\#^{2g} S^1 \times S^2)]$ and $M_i$ plays the role of $\Lambda^{2g-i}$ as in the previous section; we find some convenience in the switch of indices in the latter.)  We want these to satisfy 

\begin{itemize}

\item $VH + HV = 0$;

\item there is a commutative diagram
\begin{equation}
\label{eq:13}
\begin{array}{lclclclclclclcl}
0 &
\longrightarrow &
M_{2g} &
\overset{V}{\longrightarrow} &
M_{2g-1} &
\overset{V}{\longrightarrow} &
\ldots &
\overset{V}{\longrightarrow} &
M_0 &
\longrightarrow &
\mathbb{Z} &
\longrightarrow &
0 \\
 & &
* \downarrow &
 &
* \downarrow &
 & 
 &
 &
* \downarrow &
 &
\downarrow & & \\
0 &
\longrightarrow &
M_0 &
\overset{H}{\longrightarrow} &
M_1 &
\overset{H}{\longrightarrow} &
\ldots &
\overset{H}{\longrightarrow} &
M_{2g} &
\longrightarrow &
\mathbb{Z} &
\longrightarrow &
0 \\
\end{array} 
\end{equation}
with exact rows, where the rightmost horizontal arrow is an isomorphism of $R$-modules.

\end{itemize}

Next, suppose that we are given a periodic string of integers $\ldots, \eta(0), \eta(1), \ldots, \eta(d-1), \eta(d) = \eta(0), \ldots$.
Corresponding to each integer $p$, we have two $\mathbb{Z} \oplus \mathbb{Z}$-graded groups, 
$$L^p = \bigoplus_{0 \leq m \leq 2g, i \in \mathbb{Z}} L^p_m(i), \mbox{ }\mbox{ }R^p = \bigoplus_{0 \leq m \leq 2g, i \in \mathbb{Z}} R^p_m(i),$$
with summands given as follows.  Given $i, p$ and $m$, let $j = i - \eta(p) - g + m$.
If $\mbox{max}\{i, j\} \geq 0$, then $L^p_m(i) \cong M_m$ is concentrated in grading level $(i,j)$; otherwise, $L^p_m(i)$ is trivial.  If $i \geq 0$, then $R^p_m(i) \cong M_m$ is concentrated in grading level $(i,j)$; otherwise, $R^p_m(i)$ is trivial.  There is an obvious quotient map $q_p: L^p \rightarrow R^p$.  There is also a $U$-action on each $L^p$ and $R^p$, taking a summand in grading level $(i,j)$ to the one in level $(i-1, j-1)$, by either an isomorphism or the zero map.

We refer to $L^p \oplus R^p$ as \emph{page $p$}.  Extending the metaphor, $L^p$ and $R^p$ will be referred to as the \emph{left side} and \emph{right side} of page $p$, and the sum $\mathcal{B} = \bigoplus_{p \in \mathbb{Z}} L^p \oplus R^p$ will be referred to as the \emph{book}. 

If $x \in L^p_m(i)$, define maps $V^{L}_p$ and $H^L_p$ from $L^p$ to itself by 
$$V^{L}_p(x) = Vx \in L^p_{m-1}(i) \mbox{, }\mbox{ } H^{L}_p(x) = Hx \in L^p_{m+1}(i-1),$$
taking $Vx$ or $Hx$ to be zero if the target summand is trivial.  Define $V^{R}_p$ and $H^{R}_p$ similarly, except replacing all $L$'s with $R$'s.  
Then, define $$\la_p = V^{L}_p + H^{L}_p \mbox{, }\mbox{ }\rho_p = V^{R}_p + H^{R}_p.$$
Also, define $S_p: L^p \rightarrow R^{p+1}$ by 
$$S_p(x) = *x \in R^{p+1}_{2g - m}(i - \eta(p) -g + m).$$
An illustration of the geography of all this algebra is shown in Figure \myfig{6}; in terms of this, $S_p$ takes column $i$ on page $p$ to row $i$ on page $p+1$.

\begin{figure}[p]
\label{fig:6}
\centering  \includegraphics[scale=.60]{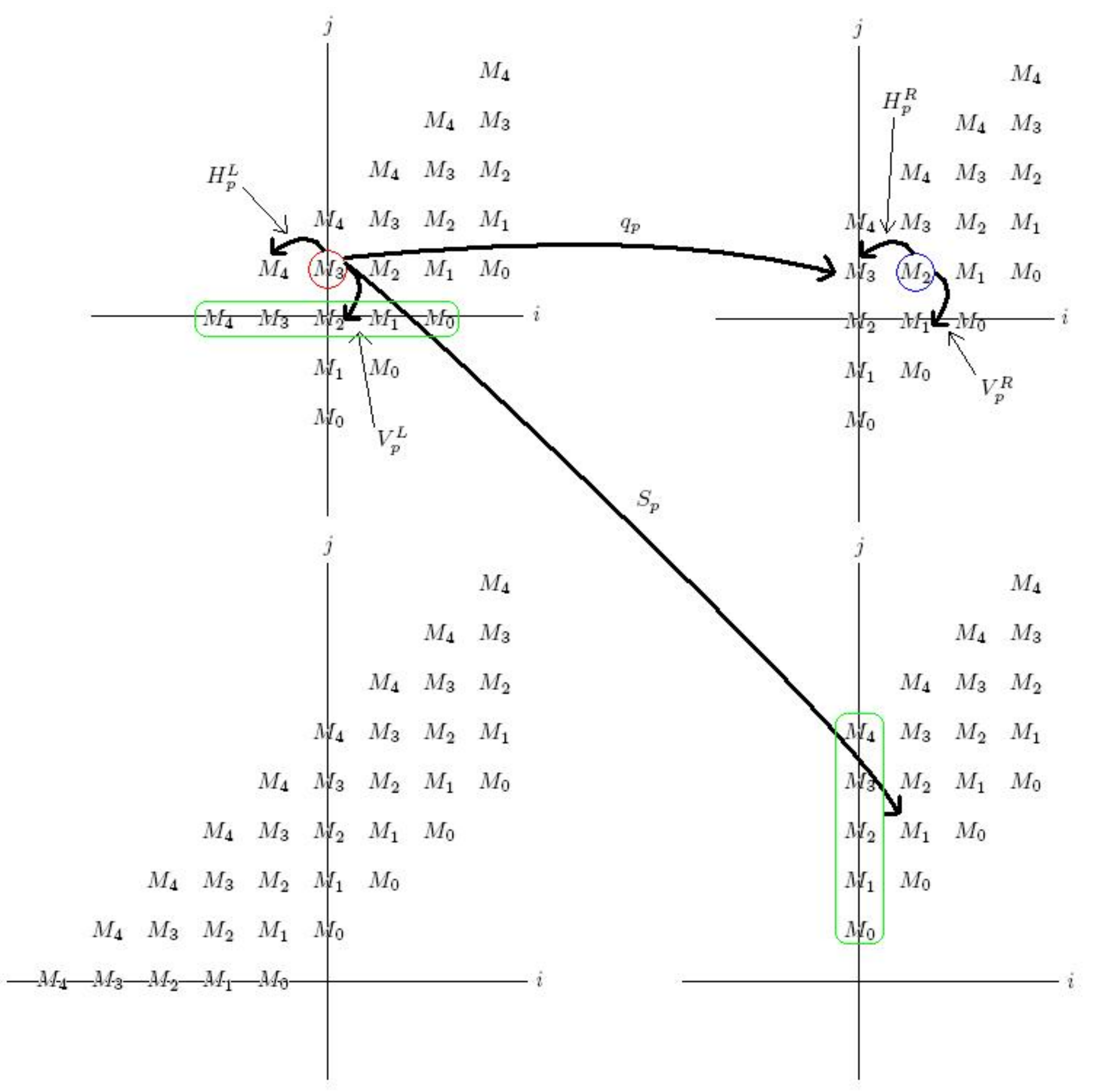}
\caption{The top two grids represent groups $L^p$ (left) and $R^p$ (right); the bottom two grids represent $L^{p+1}$ and $R^{p+1}$.  Here, $g=2$, $\eta(p) = 0$ and $\eta(p+1) = -3$.  (So $\eta(p)$ measures where the middle group $M_g$ intersects the $i$-axis.)  The red-circled group lies in page-grading $(0,1)$, and it is denoted by $L^p_3(0)$; the blue-circled group lies in page-grading $(1,1)$, and is denoted by $R^p_2(1)$.  The arrows show where each map takes $L^p_3(0)$ or $R^p_2(1)$; note that these two groups lie in the same level, since the targets of $q_p$ and $H^R_p$ coincide.  Of course, $\la_p = V^L_p + H^L_p$ and $\rho_p = V^R_p + H^R_p$; the maps $\la_p$ and $\rho_p$ always take a diagonal with slope -1 to the next such diagonal to the left.  The map $S_p$ takes the green box in the upper left grid to the green box in the lower right grid, with the page-grading level containing $M_i$ in the top going to the page-grading level containing $M_{4-i}$ in the bottom.}
\end{figure}

It is easy to see that under our assumptions about $V, H$ and $*$, we have 
\begin{equation}
\label{eq:14}
q_p \circ \la_p = \rho_p \circ q_p
\end{equation}
and
\begin{equation}
\label{eq:15}
S_p \circ \la_p = \rho_{p+1} \circ S_p.
\end{equation}

Using the above, we turn a book $\mathcal{B}$ into a chain complex, with differential $D: \mathcal{B} \rightarrow \mathcal{B}$
given by 
$$D = \sum_{p\in\mathbb{Z}} \la_p + \rho_p + q_p + S_p.$$

\vs

\bline{Convention:} For simplicity, when working with elements in a book, we will ignore signs (i.e., we ``work mod 2''); adding these in to make the arguments work precisely is straightforward.

\vs

\begin{lemma} \label{thm:8.1} The map $D$ is indeed a differential; the maps $\la_p$ and $\rho_p$ are as well, and hence the maps $q_p: L^p \rightarrow R^p$ and $S_p: L^p \rightarrow R^{p+1}$ are chain maps with respect to these differentials. 
\end{lemma}

\begin{proof} That $\la_p^2 = 0$ and $\rho_p^2 = 0$ are straightforward consequences of the assumption that $VH + HV = 0$.  The map $D$ will be a differential if $\la_p^2 = 0$ and $\rho_p \circ(\rho_p + q_p + S_{p-1}) + q_p \circ \la_p + S_{p-1} \circ \la_{p-1} = 0$.  Equations \ref{eq:14} and \ref{eq:15} show that the latter does hold; they also show that $q_p$ and $S_p$ are chain maps. \end{proof}

We endow $\mathcal{B}$ with a $\mathbb{Z}$-grading as follows.  Let $f(p)$ be defined by 
$$f(0) = g- \eta(0) + 1, \mbox{ }\mbox{ }\mbox{ }f(p+1) = f(p) + \eta(p) + \eta(p+1).$$
Then we set the summand of $L^p$ lying in $\mathbb{Z} \oplus \mathbb{Z}$-grading $(i,j)$ to lie in grading level $i+j+f(p)$, and the summand of $R^p$ lying in grading $(i,j)$ to lie in grading level $i+j+f(p)+1$.  We henceforth refer to the $\mathbb{Z} \oplus \mathbb{Z}$-grading as the \emph{page grading}.  It is easy to see that the differential $D$ lowers the $\mathbb{Z}$-grading by 1.  Notice that if $R^p_m(i)$ lies in $\mathbb{Z}$-grading $s$, then $m$ has the same parity as $s$; and if $L^p_m(i)$ lies in grading $s$, then $m$ and $s$ have different parities.

\subsection{The homologies of $L^p$ and $R^p$} 
The differentials $\la_p$ on $L^p$ and $\rho_p$ on $R^p$ also respect the grading we have introduced.  So, we now carry out the computation of $H_s(L^p, \la_p)$ and $H_s(R^p, \rho_p)$, which we hereafter write as $H_s(L^p)$ and $H_s(R^p)$.  A level $s$ in $L^p$ is considered trivial if the intersection of grading level $s$ with $L^p$ contains only trivial summands, and similarly for $R^p$.

We start with a preliminary Lemma.

\begin{lemma} \label{thm:8.2} For $x \in M_i$ with $i \geq 1$, $Vx \in \mathrm{Im }VH \cap M_{i-1}$ if and only if $x \in \mathrm{Im }V + \mathrm{Im }H$; and $\mathrm{Im }VH \cap M_{0}$ is trivial.  For $x \in M_i$ with $i \leq 2g-1$, $Hx \in \mathrm{Im }VH \cap M_{i+1}$ if and only if $x \in \mathrm{Im }V + \mathrm{Im }H$; and $\mathrm{Im }VH \cap M_{2g}$ is trivial.
\end{lemma}

\begin{proof} We prove the first pair of statements; the second pair is similar.  Suppose $x \in M_i$ with $i \geq 1$, and $Vx = VHy$ for some $y \in M_{i-1}$.  Then $x + Hy \in \mbox{Ker }V = \mbox{Im }V$, so $x + Hy = Vz$ for some $z \in M_i$.  Hence, $x = Hy + Vz \in \mbox{Im }V + \mbox{Im }H$.  Conversely, if $x \in M_i$ with $i \geq 1$, and $x = Hy + Vz$, then $Vx = VHy + VVz = VHy \in \mbox{Im }VH$. 

If $x \in M_0$, then $VHx = HVx = H0 = 0$.  Since $VH$ takes $M_i$ to $M_i$, this shows that the image of $VH$ in $M_0$ is trivial.\end{proof}

The following computation is essentially that of $\underline{HF}^+(\#^{2g} S^1 \times S^2)$ in the torsion \sst\ structure, but we go through it carefully.

\begin{prop} \label{thm:8.3}  There is an isomorphism 
$$H_s(R^p) \cong 
\left\{ 
\begin{array}{ll}
\mathbb{Z} & \mbox{if }s\mbox{ is even and a non-trival level in }R^p\\
0 & \mbox{otherwise.} 
\end{array} \right.$$
\end{prop}

\begin{proof} We continue to ignore signs; we also suppress also sub- and superscripts of $V$ and $H$.

In any non-trivial even level of $R^p$, write an arbitrary element as $x=x_0 + x_2 + \ldots + x_{2r}$, where for some $i_0 \geq 0$, we have $x_{2m} \in R^p_{2m}(i_0 - m)$ for $0 \leq m \leq r$, with $r = \mbox{min}\{i_0, g\}$.    Likewise, in any non-trivial odd level of $R^p$, write an arbitrary element as $x=x_1 + x_3 + \ldots + x_{2r+1}$, where for some $i_0 \geq 0$, $x_{2m + 1} \in R^p_{2m+1}(i_0 - m)$ for $0 \leq m \leq r$, with $r = \mbox{min}\{i_0, g-1\}$.

Let us first compute $H_s(R^p)$ for $s$ a non-trivial even level.  To start, let us 
suppose that $x=x_0 + \ldots + x_{2r}$ is a boundary in an even level of $R^p$.  This is equivalent to the existence of $y_1, y_3, \ldots, y_{2r-1}$ and $y_{2r+1}$ (the latter of which we take to be $0$ if $r=g$) which satisfy:  
$$Vy_1 = x_0,$$
$$Vy_{2i+1} = x_{2i} + Hy_{2i-1} \mbox{ for } 0 < i \leq r,$$
where we have arranged the equations suggestively.

Let us find under what conditions the above system has a solution.  Clearly, $x_0$ must be in the image of $V$ to solve the first equation.  Next, suppose that we are given $y_1$ and $x_0 = Vy_1$.  Then the equation $Vy_3 = x_2 + Hy_1$ will have a solution $y_3$ if and only if $Vx_2 + VHy_1 = 0$.  But $VHy_1 = HVy_1 = Hx_0$; hence, the first two equations will be solvable when $x_0$ is in the image of $V$ and $Vx_2 = Hx_0$, and this statement of course does not depend on what $y_1$ is.

Going forward, the same argument shows that given $x_0, \ldots, x_{2r}$, existence of a simultaneous solution to all the above equations is equivalent to satisfying the conditions
$$x_0 \in \mbox{Im }V$$
and 
$$Vx_{2i + 2} = Hx_{2i} \mbox{ for }0 \leq i < r;$$
that is, these are necessary and sufficient conditions for $x$ to be a boundary.  

On the other hand, $x$ is a cycle if and only if it satisfies at least the second condition above.  We claim that for any given $x_0$, we can find $x_2, x_4, \ldots, x_{2r}$ such that $x$ will then be a cycle.  Given this, we can show that $H_s(R^p) \cong \mathbb{Z}$ as follows.  If $x$ and $x'$ are two cycles with $x_0 = x'_0$, then it is clear that, writing $y$ for $x + x'$, that $y_0 = 0$ and $Vy_{2i + 2} = Hy_{2i}$ for $0 \leq i < r$, so that $y$ is a boundary and hence $x$ and $x'$ are homologous.  So the homology class of any cycle $x$ is determined by $x_0$.  By our claim, $x_0$ can take on any value in $M_0$, so we have a homomorphism from $M_0$ to $H_s(R^p)$, and the kernel must clearly be the image of $V$.  Thus, $H_s(R^p)$ will be isomorphic to $M_0/ \mbox{Im}(V:M_1 \rightarrow M_0) \cong \mathbb{Z}$.

To see the claim, note first that since $Vx_0 = 0$, it follows that $HVx_0 = VHx_0 = 0$.  Hence $Hx_0$ is in the image of $V$ by the exactness of the top sequence of (\ref{eq:13}), guaranteeing the existence of some $x_2$ satisfying $Hx_0 = Vx_2$.  Next, $VHx_2 = HVx_2 = HHx_0 = 0$; so $Hx_2$ is in the image of $V$, and write $Hx_2 = Vx_4$.  We continue this process until we have constructed $x_{2r}$.  So, we have shown the claim.

For the case where $s$ is a non-trivial odd level, a similar argument to the one used above shows that $x=x_1 + x_3 + \ldots + x_{2r+1}$ is a boundary if and only if
$$x_1 \in \mbox{Im }V + \mbox{Im }H$$
and
$$Vx_{2i+1} = Hx_{2i-1} \mbox{ for }0 < i \leq r.$$
However, note that the first condition is equivalent to $Vx_1 = 0$, as a consequence of Lemma \ref{thm:8.2}.   
Note also that if $r = g-1$, the condition that $Vx_{2g-1} = Hx_{2g-3}$ implies that $HVx_{2g-1} = 0 = VHx_{2g-1}$, which implies that $Hx_{2g-1} = 0$ since $V$ is injective on $M_{2g}$; therefore, these conditions are precisely those under which $x$ is a cycle.  Hence $H_s(R^p)$ is trivial. \end{proof}

For the computation of $H_s(L^p)$, we make some more definitions.

A level $s$ in $L^p$ is called a \emph{stradler} for $L^p$ if the summands of level $s$ lie in page-gradings $(i,j)$ with $i+j \leq -2$, and there is both a non-trivial summand of $s$ in some page-grading $(i,j)$ with $i<0$ and 
a non-trivial summand of $s$ in some page-grading $(i,j)$ with $j<0$.  If $0 < k < 2g$, a level $s$ for $L^p$ is called a \emph{$k$-corner} for $L^p$ if there are two non-trivial summands of $s$, $L^p_{k-1}(0)$ and $L^p_{k+1}(-1)$, which respectively lie in page-gradings $(0, -1)$ and $(-1,0)$. 

Also, define $f_q(p) = f(p) - \eta(p) - g$.  Then it is easy to check that level $s$ on page $L^p$ will have non-trivial summands with $i \geq 0$ if $s \geq f_q(p)$; that this level will have non-trivial summands with $j \geq 0$ if $s \geq f_q(p+1)$; and if both of these hold, that the level will be a stradler or a $k$-corner if $s < f(p)-1$ or $s=f(p)-1$, respectively.

\begin{prop} \label{thm:8.4} If $s$ is \emph{not} a $k$-corner for some $k$, then there is an isomorphism  
$$H_s(L^p) \cong
\left\{ 
\begin{array}{ll}
\mathbb{Z} \oplus \mathbb{Z} & \mbox{if }s\mbox{ is an odd stradler} \\
\mathbb{Z} & \mbox{if }s\mbox{ is any other odd non-trivial level} \\
0 & \mbox{otherwise;} 
\end{array} \right.$$
The first case occurs for odd $s$ satisfying $s \geq f_q(p)$, $s \geq f_q(p+1)$, and $s < f(p)-1$. In this case, both $H_{s-1}(R^p)$ and $H_{s-1}(R^{p+1})$ will be isomorphic to $\mathbb{Z}$; and the maps ${q_p}_*$ and ${S_p}_*$ will be surjective maps whose kernels have trivial intersection.   Among levels $s$ which are not in the first case, the levels in the second case occur precisely for odd $s$ which  satisfy at least one of $s \geq f_q(p)$ and $s \geq f_q(p+1)$.  In the second and third cases, $H_{s-1}(R^p) \cong \mathbb{Z}$ and ${q_p}_*$ is an isomorphism precisely if $s$ is odd and $s \geq f_q(p)$, and otherwise both the group and the map are trivial; and $H_{s-1}(R^{p+1}) \cong \mathbb{Z}$ and ${S_p}_*$ is an isomorphism precisely if $s$ is odd and $s \geq f_q(p+1)$, and otherwise both the group and the map are trivial.  
\end{prop}

\begin{proof} We again work ignoring signs, and write $V$ and $H$ without sub- and superscripts.  

Suppose the level $s$ is not a stradler.  Then, we end up with the same calculation as in the previous Proposition except with opposite parity (since we are on the left page now): $H_s(L^p)$ is $\mathbb{Z}$ is $s$ is non-trivial and odd, and $0$ otherwise.  To see this, note that if $s$ is not a stradler, then the page-gradings $(i,j)$ of the non-trivial summands of $s$ either all satisfy $i \geq 0$ or all satisfy $j \geq 0$.  In the former case, the calculation is identical to that of Proposition \ref{thm:8.3}.  The latter case immediately follows by a symmetrical argument.  More precisely, this follows from the arguments of Proposition \ref{thm:8.3} when notation changes are made throughout -- mostly switching $V$'s and $H$'s and replacing terms identified with $M_i$ with terms identified with $M_{2g - i}$.

Now, suppose the level $s$ is a stradler that is not a $k$-corner.  We assume that $s$ is odd.  Then the non-trivial summands of in level $s$ will be $L^p_{2m}(i_0 - m)$ for $m \in [0, i_0] \cup[s,g]$, where $i_0$ and $s$ are some values in $\{0, 1, \ldots, g\}$ that satisfy $i_0 \leq s - 2 $ (this follows from the fact that $i+j \leq -2$).  It is not hard to see that, in fact, level $s$ of the chain complex splits as the direct sum of two separate complexes, one corresponding to the $0 \leq m \leq i_0$ summands, and one corresponding to the $s \leq m \leq g$ summands.  Furthermore, the computations of the homologies of these complexes proceeds just as before, yielding $\mathbb{Z}$.  Putting them together, $H_s(L^p)$ is $\mathbb{Z} \oplus \mathbb{Z}$ if $s$ is an odd (non-corner) stradler.  
Similar remarks apply to the case when $s$ is an even level, and we get that $H_s(L^p)$ is trivial in this case.

The statements about which of the three cases a level $s$ falls into follow from the discussion above.  The statements about when $H_{s-1}(R^p) \cong \mathbb{Z}$ or $H_{s-1}(R^{p+1}) \cong \mathbb{Z}$ are straightforward. To see the statements about ${q_p}_*$ and ${S_p}_*$, note that if $[x] \in H_s(L^p)$, 
then ${q_p}_*([x])$ will be determined by the homology class of $x_0 \in M_0/\mbox{Im }V$ if there is such a component; likewise, ${S_p}_*([x])$ will be determined by the homology class of $x_{2g} \in M_{2g}/\mbox{Im }H$.  The statement for ${q_p}_*$ follows straightforwardly from this after examining the constructions; the same can be said about the statement for ${S_p}_*$, after recalling the fact that $*$ induces an isomorphism from $M_0/\mbox{Im }V$ to $M_{2g}/\mbox{Im }H$ (the righthand isomorphism of (\ref{eq:13})).  \end{proof}

Now, we finally define the groups $\Omega^g(k)$ of the introduction.

\begin{defi}
\label{def:8.5}
Let for $0 < k < 2g$, let $\Omega^g(k)$ be the $R$-module 
$\mbox{Im }VH \cap M_k$, and let $\Omega^g(k)$ be trivial for other values of $k$.
Leaving the abstract setting of this section, the above is standing in for the $\mathbb{Z}[H^1(\#^{2g} S^1 \times S^2)]$-module $\Omega^g(k) = \mbox{Im }VH \cap \Lambda^{2g-k}$,
where $V, H$ and $\Lambda^*$ are as in the statement of Proposition \ref{thm:7.4}.
\end{defi}

Note that as subgroups of torsion-free groups, the $\Omega^g(k)$ are all torsion-free as groups.

\begin{prop} \label{thm:8.6} If $s$ is a $k$-corner for some $k$, then there is a short exact sequence
$$ 0 \rightarrow \Omega^g(k) \rightarrow H_s(L^p) \overset{{q_p}_* \oplus {S_p}_*}{\longrightarrow} H_{s-1}(R^p) \oplus H_{s-1}(R^{p+1}) \rightarrow 0.$$
If $s$ is an odd level, then $H_{s-1}(R^p) \cong \mathbb{Z}$ and $H_{s-1}(R^{p+1}) \cong \mathbb{Z}$, and the kernels of the maps ${q_p}_*$ and ${S_p}_*$ have intersection $\Omega^g(k)$.  If $s$ is an even level, then $H_{s-1}(R^p)$ and $H_{s-1}(R^{p+1})$ are both trivial.
\end{prop}

\begin{proof}  This case is similar to the case when $s$ is a stradler, but note that here, the level does not quitesplit as the direct sum of two separate complexes, as the group in page-grading $(0,0)$ maps into both ``sides of the straddle.''   

We work out the case where $s$ is odd in detail; for even $s$, the argument is similar.  For such a level, we can write an arbitrary element of $s$ as $x = x_0 + x_2 + \ldots + x_{2g}$, where 
for some value of $k$ satisfying $0 \leq \frac{k-1}{2} \leq g -1$, we have $x_{2m} \in L^p_{2m}(\frac{k-1}{2} - m)$ for $m$ between $0$ and $g$.  

The criterion for $x$ to be a boundary will clearly be the same as if $s$ were a non-trivial non-stradler; as in the proof of Proposition \ref{thm:8.3}, the criterion is that $x_0, \ldots, x_{2g}$ satisfy 
$$x_0 \in \mbox{Im }V$$
and 
$$Vx_{2i + 2} = Hx_{2i} \mbox{ for }0 \leq i < g.$$
In fact, it is not hard to see that this also implies that $x_{2g} \in \mbox{Im }H.$
On the other hand, inspection shows that the criterion for $x$ to be a cycle will now be just that
$$Vx_{2i + 2} = Hx_{2i} \mbox{ for }0 \leq i < \frac{k-1}{2} \mbox{ and } \frac{k-1}{2} < i < g.$$

We now construct some subgroups of the group of cycles.  First, choose some fixed generator $x^*_0$ of $M_0$.  It is easy to see that $x^*_0$ represents a generator of $M_0 / \mbox{Im }V$.  Then, use the process described in Proposition \ref{thm:8.3} to construct fixed components $x^*_0, x^*_2, \ldots, x^*_{k-1}$ of a cycle, stopping there, where $x^*_{i} \in M_i$.  We call this cycle $X^*_0$; let $S_0$ be the free $R$-module generated by this fixed cycle.  Likewise, construct a fixed element $X^*_{2g}$ by a symmetric process, starting with some generator of $x^*_{2g} \in M_{2g}$, and let $S_{2g}$ be the free $R$-module generated by $X^*_{2g}$.  
Also, let $S_{k-1}$ denote the $R$-module $\mbox{Ker }V \cap M_{k-1}$.

So, $S_0 \oplus S_{k-1} \oplus S_{2g}$ is a submodule of the set of cycles.  We claim that any cycle is homologous to one in $S_0 \oplus S_{k-1} \oplus S_{2g}$.  
To see the claim, let $y = y_0 + \ldots + y_{2g}$ be a cycle, with $y_0 = r_0x^*_0$ and $y_{2g} = r_{2g}x^*_{2g}$ for $r_0, r_{2g} \in R$.  Let $y' = y - r_0X^*_0 - r_{2g}X^*_{2g}$; we show the existence of $s_{k-1} \in S_{k-1}$ such that $y' -s_{k-1}$ is a boundary, which of course suffices for the claim.  

To simplify the following discussion, assume that $k \neq 1, 2g-1$ (modifications when $k=1, 2g-1$ are straightforward).  Start by noting the fact that $y'_0 = 0$.  So $Hy'_0=0$; since $y'$ is a cycle, $Vy'_2 = Hy'_0 = 0 \in \mbox{Im }VH \cap M_1$.  Now, by Lemma \ref{thm:8.2}, this means that in particular, $y'_2 \in \mbox{Im }V + \mbox{Im }H$; if $k>3$, continue by applying Lemma \ref{thm:8.2} again, to conclude that $Vy'_4 = Hy'_2 \in \mbox{Im }VH \cap M_3$.  Continuing this, we see that $Vy'_{2m} \in \mbox{Im }VH \cap M_{2m-1}$ for $1 \leq m \leq \frac{k-1}{2}$.  A symmetric argument shows that $Hy'_{2m} \in \mbox{Im }VH \cap M_{2m+1}$ for $\frac{k+1}{2} \leq m \leq g-1$. 

In particular, $Vy'_{k-1} \in \mbox{Im }VH \cap M_{k-2}$, and so applying Lemma \ref{thm:8.2} twice, $Hy'_{k-1} \in \mbox{Im }VH \cap M_{k}$; likewise, $Vy'_{k+1} \in \mbox{Im }VH \cap M_{k}$.  Therefore, $Hy'_{k-1} + Vy'_{k+1} \in \mbox{Im }HV$, or in otherwords, $Hy'_{k-1} + Vy'_{k+1} = Hs_{k-1}$ for some $s_{k-1} \in S_{k-1}$.  Now, it is easy to check the boundary conditions to show that $y' - s_{k-1}$ is a boundary (since $Vs_{k-1}$ of course equals 0).    

It follows that $H_s(L^p)$ is isomorphic to the quotient of $S_0 \oplus S_{k-1} \oplus S_{2g}$ by the submodule of boundaries it contains.  Now, it is easy to see that a cycle of the form $0 \oplus s_{k-1} \oplus 0$ is a boundary if and only if $Hs_{k-1} = 0$; so there is a subgroup of $H_s(L^p)$ that is isomorphic to $\frac{\mbox{\scriptsize{Ker }}V \cap M_{k-1}}{\mbox{\scriptsize{Ker }}V \cap \mbox{\scriptsize{Ker }}H \cap M_{k-1}}$.  This module is isomorphic to the image of $H$ restricted to $\mbox{Ker }V \cap M_{k-1} = \mbox{Im }V \cap M_{k-1}$, which in turn is isomorphic to $\mbox{Im }VH \cap M_{k} = \Omega^g(k)$.  

On the other hand, the maps ${q_p}_*$ and ${S_p}_*$ clearly take the cycle $r_0X^*_0 + r_{2g}X^*_{2g} + s_{k-1}$ to, respectively, $r_0[X^*_0] \in H_{s-1}(R^p)$ and $r_{2g}[X^*_{2g}] \in H_{s-1}(R^{p+1})$ Both of the groups are easily checked to be $\mathbb{Z}$, generated respectively by $[X^*_0]$ and  $[X^*_{2g}]$.  Therefore, ${q_p}_* \oplus {S_p}_*$ is surjective.  The kernel of this map is composed of cycles of the form $r_0X^*_0 + r_{2g}X^*_{2g} + s_{k-1}$ where $r_0x^*_0 \in \mbox{Im }V$ and $r_{2g}x^*_{2g} \in \mbox{Im }H$.  But the argument used to show that $y'$ is homologous to an element of $S_{k-1}$ can be altered slightly, to work when $y'$ is instead taken to be any element with $y'_0 \in \mbox{Im }V$ and $y'_{2g} \in \mbox{Im }H$.  So, the kernel of ${q_p}_* \oplus {S_p}_*$ is contained in the submodule $\Omega^g(k)$, and the opposite inclusion is obvious.  \end{proof}

\subsection{Calculation of $H_*(\mathcal{B})$} 
We now give a relatively simple description of $H_*(\mathcal{B})$, and then we proceed to prove the formula that computes it.

\begin{prop} \label{thm:8.7} Let $H^{(r)}_*(\mathcal{B})$ be the quotient of $\bigoplus_{p \in \mathbb{Z}} H_*(R^p)$ by the subgroup consisting of those $\bigoplus_{p \in \mathbb{Z}} [r_p]$ such that there exist $[\ell_p] \in H_s(L^p)$ satisfying ${q_p}_*([\ell_p]) + {S_{p-1}}_*([\ell_{p-1}]) = [r_p]$.
Also, let $H^{(l)}_*(\mathcal{B})$ be the subgroup of $\bigoplus_{p \in \mathbb{Z}} H_*(L^p)$ consisting of those $\bigoplus_{p \in \mathbb{Z}} [\ell_p]$ for which ${S_p}_*([\ell_{p-1}]) + {q_p}_*([\ell_p]) = 0 \in H_*(R^p)$.
Then there is a short exact sequence 
$$ 0\rightarrow H^{(r)}_*(\mathcal{B}) \rightarrow H_*(\mathcal{B}) \rightarrow H^{(l)}_*(\mathcal{B}) \rightarrow 0.$$
\end{prop}

\begin{proof}  First, consider any element of the form $\bigoplus_{p \in \mathbb{Z}} 0 \oplus r_p$.  This will be a cycle if and only if $\rho_p(r_p) = 0$ for all $p$.  It will be a boundary if there exist $\ell'_p$ and $r'_p$ such that $\la_p(\ell'_p)=0$ and $q_p(\ell'_p) + S_{p-1}(\ell'_{p-1}) + \rho_p(r'_p) = r_p$
for all $p$.  It is straightforward to see that this is equivalent to the existence of $[\ell'_p] \in H_s(L^p)$ satisfying ${q_p}_*([\ell'_p]) + {S_{p-1}}_*([\ell'_{p-1}]) = [r_p]$.  So the subgroup of $H_*(\mathcal{B})$ generated by cycles of this form will be isomorphic to $H^{(r)}_*(\mathcal{B})$.

Next, suppose that we are given arbitrary $\ell = \bigoplus_{p\in \mathbb{Z}} \ell_p \in \bigoplus_{p\in \mathbb{Z}} L^p$, with each $\ell_p$ supported in level $s$.  If $\la_p(\ell_p) = 0$ and ${q_p}_*([\ell_p]) + {S_{p-1}}_*([\ell_{p-1}]) = 0 \in H_s(R^p)$, then there is a cycle of the form $\bigoplus_{p\in \mathbb{Z}} \ell_p \oplus r_p$; for in this case, we have ${q_p}(\ell_p) + {S_{p-1}}(\ell_{p-1}) = \rho_p(r_p)$ for some $r_p \in R^p$ supported in level $s$.  Conversely, it is clear that $\ell$ will need to satisfy these conditions for such values of $r_p$ to exist.  Furthermore, there will exist $r_p$ which make $\bigoplus_{p\in \mathbb{Z}} \ell_p \oplus r_p$ a boundary if and only if $[\ell_p] = 0 \in H_s(L^p)$ for all $p$.  Hence, the quotient of the group of $\ell$ which extend to a cycle by the subgroup of $\ell$ which extend to a boundary is $H^{(l)}_*(\mathcal{B})$.

The claim follows easily. \end{proof}

Now, define a function $F_{\mathcal{B}}: \frac{1}{2}\mathbb{Z} \rightarrow \mathbb{Z}$
by
$$F_{\mathcal{B}}(m) = 
\left\{ \begin{array}{ll}
f_q\left(m + 1\right), & m \in \mathbb{Z} \\
f\left(m+\frac{1}{2}\right), & m \in \frac{1}{2} + \mathbb{Z}.
\end{array} \right.$$
This function can be characterized more directly by 
$$ F_{\mathcal{B}}(m) = G_{\mathcal{B}}(m) + \left\{ \begin{array}{ll}
0, & m \in \mathbb{Z} \\
g, & m \in \frac{1}{2} + \mathbb{Z},
\end{array} \right.$$
where $G_{\mathcal{B}}: \frac{1}{2}\mathbb{Z} \rightarrow \mathbb{Z}$ is characterized by
$$G_{\mathcal{B}}(0) = 1,$$
$$G_{\mathcal{B}}(m) = G_{\mathcal{B}}(m-1) + 2\eta(m) \mbox{ for } m \in \mathbb{Z}, $$
$$G_{\mathcal{B}}\left(m + \frac{1}{2}\right) = \frac{1}{2}\left(G_{\mathcal{B}}(m) + G_{\mathcal{B}}(m+1) \right) \mbox{ for } x \in \mathbb{Z}.$$

Let 
$$ \Omega_s(\mathcal{B}) \cong \bigoplus_{ \{p\in \mathbb{Z} | F_{\mathcal{B}}(p+\frac{1}{2}) = s + 1\} } \Omega^g\left(g + \eta(p)\right),$$
recalling that we take $\Omega^g(k)$ to be $0$ for $k < 1$ and $k > 2g - 1$.  This is easily checked to be the direct sum of all the $\Omega^g(k)$ subgroups of $\bigoplus_{p \in \mathbb{Z}} H_*(L^p)$. Also, recall from the introduction the definition of the well group $\mathbb{H}_*(F)$ of a function $F:\frac{1}{2}\mathbb{Z}\rightarrow\mathbb{Z}$.

\begin{prop} \label{thm:8.8} If $s$ is even, we have $H^{(l)}_s(\mathcal{B}) \cong \Omega_s(\mathcal{B})$.  If $s$ is odd, we have a short exact sequence
$$ 0 \rightarrow \Omega_s(\mathcal{B}) \rightarrow H^{(l)}_s(\mathcal{B}) \rightarrow \mathbb{H}_s(F_{\mathcal{B}}) \rightarrow 0.$$
\end{prop}

\begin{proof} We work through the case where $s$ is odd.  The case where $s$ is even is straightforward, since ${q_p}_*$ and ${S_p}_*$ will vanish on $H_s(L^p)$ for all $p$ in this case.

We first tackle the case where $\Omega_s(\mathcal{B})$ is trivial.  We return to the general case at the end.  We start by drawing the following diagram:

\begin{center} \includegraphics[scale=.70]{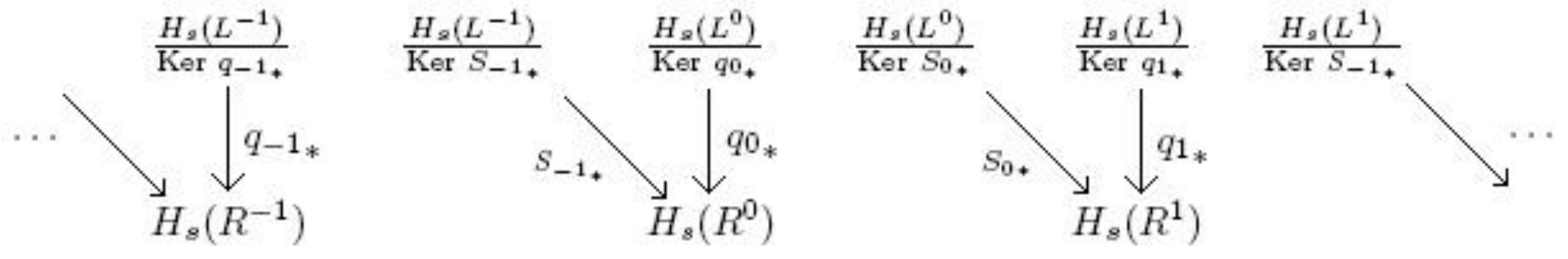}\end{center} 

\noindent Notice that the maps ${q_p}_*$ and ${S_p}_*$ are always surjective, so that every map is an isomorphism.
We think of this as a graph, with maps as edges and groups as vertices.  In these terms, our diagram in its current state has one connected component for each integer $p$, each component with 3 vertices.  
Now, we alter this diagram.
\begin{itemize}

\item If $s < f_q(p)$, then all three groups in the component containing $H_s(R^p)$ are trivial, so remove the component. 

\item After doing this, if anytime we have two consecutive components remaining, it means that $H_s(L^p)$ is either isomorphic to $\mathbb{Z}$ or to $\mathbb{Z} \oplus \mathbb{Z}$.  In the former case, this means that $H_s(L^p)/\mbox{Ker }{q_p}_*$ and $H_s(L^p)/\mbox{Ker }{S_p}_*$ are both canonically isomorphic to $H_s(L^p)$, and hence to each other.  This happens when $s \geq f(p)$ (in addition to $s \geq f_q(p)$ and $s \geq f_q(p+1)$); for such values of $p$, draw an arrow between $\mbox{Im }{q_p}_*$ and $\mbox{Im }{S_p}_*$.  

\item Finally, remove any infinite connected components.  

\end{itemize}
For any element of $H^{(l)}_s(\mathcal{B})$, its summands have to fit into this diagram.  That is, each non-trivial summand of an element of $H^{(l)}_s(\mathcal{B})$ lives in some $H_s(L^p)$, and hence via quotient maps, has two homes on the top half of the diagram; then, if two arrows lead to the same summand on the bottom, the elements in front of these two arrows must map to the same element on bottom.  Indeed, it is easy to see that this is a necessary and sufficient for an element of $\bigoplus_{p \in \mathbb{Z}} H_*(L^p)$ to be in the subgroup $H^{(l)}_s(\mathcal{B})$.  

But it is clear, then, that the value in any summand left over is an element of $\mathbb{Z}$, if we choose a value in one such summand, this determines the values in any of the other summands in the same connected component.  We therefore have one $\mathbb{Z}$ summand for each such component left over.  (The reason that we remove any infinite connected components is that the purported $\mathbb{Z}$ summand for such a component would have only elements with an infinite number of non-zero direct summands.)

Now, let us extract the process for determining the number of components (and hence the number of $\mathbb{Z}$ summands).  Start with an infinite graph with labelled vertices, like so, 
 
\begin{center} \includegraphics[scale=.50]{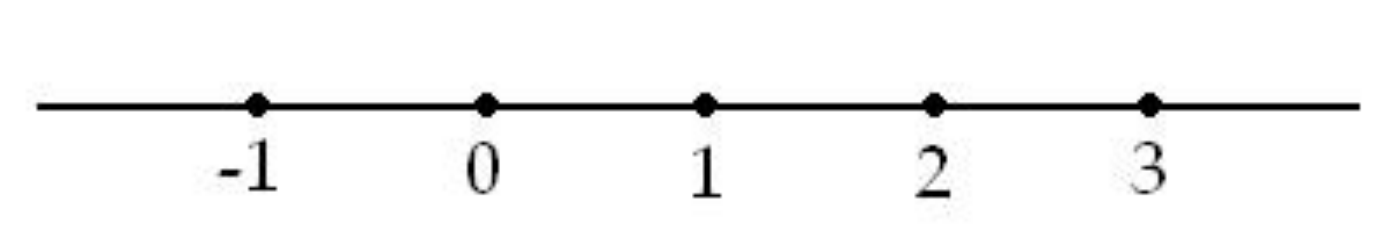}\end{center}

Weight vertex $p$ by $f_q(p)$, and weight the edge between vertices $p$ and $p+1$ by $f(p)$.  Then, for a value of $s$,  we simply remove edges and vertices with weight greater than $s$, and then throw out any infinite components; this is essentially the process described above, and it yields the same number of connected components.

To count these components, note that each finite run of consecutive non-removed \textbf{vertices} corresponds to a disjoint union of finite components of the graph; and the total number of finite components will equal the number of such runs, plus the number of removed edges between two non-removed vertices.  The number of runs is the number of wells of $f_q$ at level $s$, and the number of such edges is the number of values of $p$ such that $f(p) > s \geq \mbox{max}\{f_q(p), f_q(p+1)\}$. 

Finally, we claim that the latter number is the difference between the number of wells of $f_q$ and the number of wells of $F_{\mathcal{B}}$ at level $s$.  To see this, we compare the topography of $f_q$ and $F_{\mathcal{B}}$.  The graph of the latter is formed by shifting the graph of the former to the left by $1$, and then adding points $(p+\frac{1}{2},f(p+1))$ for each integer $p$. We care about those new points that disturb a run of points above height $s$ or a run of points below height $s$.  Note that, examining definitions, either $f(p)$ will be between $f_q(p)$ and $f_q(p+1)$ inclusive, or it will be strictly greater than both of them.  
Any point $(p, f(p))$ with $p$ satisfying the first alternative can be ignored; each point with $p$ satisfying the latter will be pertinent precisely when $f(p) > s$ and $\mbox{max}\{f_q(p), f_q(p+1)\} \leq s$.  It is easy to see that each such point contributes a single well of $F_{\mathcal{B}}$ at level $s$ that is not present for $f_q$; precisely, such a point either splits a well in two or adds a well to the left or right of all previous ones.  

Now, let us look at the case where $\Omega_s(\mathcal{B})$ is non-trivial.  Obviously, this subgroup will inject into $\bigoplus_{p \in \mathbb{Z}} H_*(L^p)$; and for each page $p$ that has a $k$-corner in level $s$, the quotient of $\bigoplus_{p \in \mathbb{Z}} H_*(L^p)$ by $\Omega_s(\mathcal{B})$ will have a $\mathbb{Z} \oplus \mathbb{Z}$ summand on that page.  

The group $\Omega_s(\mathcal{B})$ also injects into $H^{(l)}_s(\mathcal{B})$.
To calculate the quotient of $H^{(l)}_s(\mathcal{B})$ by $\Omega_s(\mathcal{B})$, we pretend, for any page $p$ which has a $k$-corner in level $s$, that $H_s(L^p)$ (rather than $H_s(L^p)/\Omega^g(k)$) actually was equal to $\mathbb{Z} \oplus \mathbb{Z}$, and then to proceed as before.  This is easily seen to work, and gives the same description as above.  \end{proof}

\begin{prop} \label{thm:8.9} We have 
$$H^{(r)}_s(\mathcal{B}) \cong 
\left\{ \begin{array}{ll}
\mathbb{Z} & s \mbox{ is even and }s \geq F_{\mathcal{B}}(p)\mbox{ for all }p \\
0 & \mbox{otherwise.} 
\end{array} \right.$$
\end{prop}

\begin{proof}  For odd $s$ the result is clear, so suppose $s$ is even.  Draw the initial diagram that we drew in the proof of Proposition \ref{thm:8.8} (with groups rather than just vertices), and perform the prescribed alterations, except leaving infinite components.  (If there is a $k$-corners on page $p$, just pretend that as before that $H_s(L^p)$ were $\mathbb{Z} \oplus \mathbb{Z}$.)  Each remaining group is isomorphic to $\mathbb{Z}$, and we will think of elements of $\bigoplus_{p \in \mathbb{Z}} H_*(R^p, \rho_p)$ and $\bigoplus_{p \in \mathbb{Z}} H_*(L^p)$ as strings of integers, one for each group remaining in the bottom or top, respectively.  Of course, in either case we want our strings to contain only a finite number of non-zero elements.

Proposition \ref{thm:8.6} gives $H^{(r)}_s(\mathcal{B})$ as a quotient of $\bigoplus_{p \in \mathbb{Z}} H_*(R^p, \rho_p)$.  We want to know which elements die in this quotient.  Combining the statement of Proposition \ref{thm:8.6} and the above, we see that the dying elements correspond to strings on the bottom, such that we can place a string on the top that satisfies the following: if two arrows point to the same group $H_s(R^p, \rho_p)$, the sum of the numbers at the tops of the arrows equals the number at the bottom. 

Now, it is not hard to see that if our diagram has more than one component (or one finite component), then we can do this for any placement of numbers at the bottom of the graph.  In other words, in this case, $H^{(r)}_s(\mathcal{B})$ is trivial.  This will happen if and only if $s < F_{\mathcal{B}}(p)$ for any $p$.

On the other hand, if our diagram is one infinite component, then for any placement of numbers on the bottom, we can make a placement of numbers on the top that satisfies the condition, if and only if the sum of the numbers on the bottom (with alternating signs) is zero.  In other words, in this case, the homology class in $H^{(r)}_s(\mathcal{B})$ only depends on the alternating sum, and so $H^{(r)}_s(\mathcal{B}) \cong \mathbb{Z}$. \end{proof}

\vs

\noindent \textit{Proof of Lemma \ref{thm:1.3}.}  Let $\m{t}_0$ be an $\mu$-torsion \sst\ structure on $Y_0$, which extends over $W_0$ to some $\m{s}$ for which $\m{s}|_Y = \m{t}(r_1, \ldots, r_n)$, in the notation of the end of Section 7.  We can assume that $0 \leq r_{\ell} < p_{\ell}$ for each $\ell$.
In this case, let $A = (Q; r_1, \ldots, r_n) \in \widetilde{\mathcal{MT}_K}$, where $Q$ is chosen to make Equations \ref{eq:1} and \ref{eq:2} hold.  
We want to define $\theta_K:\mathcal{MT}_K \rightarrow \spc{Y_0}$ by setting $\theta_K^{-1}(\m{t}_0)$ to be $[A]$.  In fact, it is not hard to check that every class in $\widetilde{\mathcal{MT}_K}$ can be realized in this manner, starting with some $\mu$-torsion $\m{t}_0$; and that, if we start with $\m{t}'_0$ and end up with $A'$, then $A \sim A'$ if and only if $\m{t}_0 = \m{t}'_0$.  So we can indeed define $\theta_K$ by $\theta_K([A]) = \m{t}_0$.   \halbox

\vs

\noindent \textit{Proof of Theorem \ref{thm:1.4}.} Let $\m{t}_0 = \theta_K([A]) \in \spc{Y_0}$ be $\mu$-torsion.
Choose a representative $A = (Q; r_1, \ldots, r_n) \in \widetilde{\mathcal{MT}_K}$ of $[A]$, and construct a book $\mathcal{B}$ from the values of $\eta(i)$ given at the end of Section 7.  Then $F_A$ as given in Section 1.1 is equal to $F_{\mathcal{B}}$ as constructed from this book.

Note that $H^{(l)}_s(\mathcal{B})$ and $H^{(r)}_s(\mathcal{B})$ will never both be non-trivial for any given value of $s$.  Thus, the short exact sequence of Proposition \ref{thm:8.9} splits.  Observe that $H^{(r)}_s(\mathcal{B})$ is isomorphic to $\mathcal{T}^+_{b_A}$ as a graded abelian group, with $b_A$ given as in the statement.  Hence, we have that the graded group $\underline{HF}^+(Y_0, \m{t}_0)$ sits in a short exact sequence as stated when $\m{t}_0$ is $\mu$-torsion.  It is straightforward to check that this isomorphism respects the $U$ action as well, by examining all of the above.

For the claim about the boundedness of $F_A$, note that this function is the sum of a periodic function with period $d$, and a linear function with slope $S\ell(A) = - \frac{\langle c_1(\m{t}_0), [\widehat{dS}] \rangle}{d}$. Hence, $F_A$ will be bounded precisely when $\m{t}_0$ is torsion. 
 
For \sst\ structures $\m{t}_0$ that are not $\mu$-torsion, we note that $\m{S}_N(\m{t}_0)$ and $\m{S}_{\infty}(\m{t}_0)$ will both consist entirely of non-torsion \sst\ structures by Proposition \ref{thm:2.4}.  The adjunction inequality of \cite{OSPA} works the same in the twisted coefficient setting as it does in the untwisted setting.  Since the Thurston semi-norms on $Y_N$ and $Y$ will both be trivial, the inequality implies that $\underline{HF}^+(Y_N, \m{S}_N(\m{t}_0)) \otimes \mathbb{Z}[T, T^{-1}]$ and $\underline{HF}^+(Y, \m{S}_{\infty}(\m{t}_0)) \otimes \mathbb{Z}[T, T^{-1}]$ will both be trivial. Hence, the long exact sequence shows that $\underline{HF}^+(Y_0, \m{S}^N_0(\m{t}_0))$ will be trivial as well. 

Finally, the statement about the action of $T$ follows from Theorem \ref{thm:4.5}, via the results of this section. \halbox

\vs

\noindent \textit{Proof of Corollary \ref{thm:1.5}.} As we noted, $\Omega^g(k)$ is a torsion-free group.  The well groups and $\mathcal{T}^+$ are clearly both torsion-free as well, and so the same can be said of $\underline{HF}^+(Y_0, \m{t}_0)$ for any $\m{t}_0$.  The well groups are in fact free abelian groups, so we have a splitting of the short exact sequences of Theorem \ref{thm:1.4} as $\mathbb{Z}$-modules, and it is easy to see that such a splitting will respect the $U$-action. \halbox

\section{Proofs of Theorem 1.1 and Corollary 1.2}
Choose a function $\rho: \mathcal{MT}_K \rightarrow \widetilde{\mathcal{MT}_K}$ such that $\left[\rho([A])\right] = [A]$ for $[A] \in \mathcal{MT}_K$.  We start by calculating the number of pairs $([A],x) \in \mathcal{MT}_K \times \mathbb{Z}/d\mathbb{Z}$ that satisfy
\begin{equation}
\label{eq:16}
\epsilon([A]) = E
\end{equation}
and
\begin{equation}
\label{eq:17}
\eta_{\rho([A])}(x) = D
\end{equation}
for a given pair of integers $D,E$, with notation as in Section 1 (and $x$ taken to be an integer between $0$ and $d-1$).  Let $\mathcal{N}_0(D,E)$ denote this number.

\begin{lemma} \label{thm:9.1} The number $\mathcal{N}_0(D,E)$ is independent of the choice of $\rho$.
\end{lemma}

\begin{proof}  Let $\rho([A]) = A_1 = (Q; r_1, \ldots, r_n)$ and $\rho'([A]) = A_2 = (Q'; r'_1, \ldots, r'_n)$.  
Then there is a unique $i$ with $0 \leq i < d$ such that $r'_{\ell} = r_{\ell} + iq_{\ell}$ for all $\ell$.  Note that for this $i$,
$$\sum_{\ell=1}^n \left\{ \frac{q_{\ell}x - r_{\ell}}{p_{\ell}} \right\} = 
\sum_{\ell=1}^n \left\{ \frac{q_{\ell}(x+i) - r'_{\ell}}{p_{\ell}} \right\}.$$
So $\eta_{A_1}(x) = D$ if and only if $\eta_{A_2}(x+i) = D$.  The functions $\eta_A$ are all periodic with period $d$; hence, the solutions to $\eta_{A_1}(x) = D$ and $\eta_{A_2}(x) = D$ are in 1-1 correspondence.  The total number of solutions therefore does not depend on $\rho$. \end{proof}

The argument of Lemma \ref{thm:9.1} also shows that $\eta_{\rho([A])}(x) = D$ if and only $\eta_{A_x}(0) = D$
for $A_x = (Q_x; r_1 - q_1x, \ldots, r_n - q_nx)$, where $r_{\ell} - q_{\ell}x$ is replaced with the equivalent value mod $p_{\ell}$ between 0 and $p_{\ell}-1$, and $Q_x$ is chosen so that $[A_x] = [A]$.  Therefore, the pairs $([A],x)$ that satisfy Equations \ref{eq:16} and \ref{eq:17} are in 1-1 correspondence with elements $A \in \widetilde{\mathcal{MT}_K}$ that satisfy
$$\epsilon(A) = E$$
and
$$\eta_{A}(0) = D.$$
For $A = (Q; r_1, \ldots, r_n)$ to satisfy these equations means that
\begin{equation}
\label{eq:18}
E = g_{\Sigma} - 1 - \frac{d}{2}\cdot S\ell(A) = d\left(g - 1 + \frac{1}{2}\sum_{\ell=1}^n\left(1 - \frac{1}{p_{\ell}}\right) - Q - \sum_{\ell=1}^n \frac{r_{\ell}}{p_{\ell}}\right)
\end{equation}
and
\begin{equation}
\label{eq:19}
D = \sum_{\ell=1}^n \left\{-\frac{r_{\ell}}{p_{\ell}}\right\} -\frac{E}{d} + g - 1.
\end{equation}
Given values of $r_{\ell}$, the value of $Q$ is then determined by Equation \ref{eq:18}; it is not hard to show that for any values of $r_{\ell}$ that make Equation \ref{eq:19}, the value of $Q$ thus determined will be an integer.  So, the number $\mathcal{N}_0(D,E)$ of solutions $(Q; r_1, \ldots, r_n)$ to Equations \ref{eq:18} and \ref{eq:19} is the same as the number of solutions $(r_1, \ldots, r_n)$ to Equation \ref{eq:19}.

For each $\ell$, the function $r_{\ell} \mapsto \left\{-\frac{r_{\ell}}{p_{\ell}}\right\}$ gives a bijection of the possible values of $r_{\ell}$ with $\left\{\frac{0}{p_{\ell}}, \ldots,\frac{p_{\ell}-1}{p_{\ell}}\right\}$.  So, in summation, we have the following.

\begin{lemma} \label{thm:9.2} The number $\mathcal{N}_0(D,E)$ of solutions to Equations \ref{eq:16} and \ref{eq:17} is the same as the number $\mathcal{N}(D,E)$ of solutions $(i_1, \ldots, i_n)$ to
$$ \sum_{\ell=1}^n \frac{i_{\ell}}{p_{\ell}} = D + \frac{E}{d} + 1 - g$$
with $0 \leq i_{\ell} < p_{\ell}$.
\end{lemma}

\noindent \textit{Proof of Theorem \ref{thm:1.1}.} We first look at the case where $g = 0$.  Fix a value of $i$ between $0$ and $g_{\Sigma} - 2$.  If $[A] \in \mathcal{MT}_K$ satisfies $\epsilon([A]) = i$, then $S\ell([A]) = \frac{2(g_{\Sigma} - 1 - i)}{d}$, which will hence be positive.  
Then, it is clear that every decline that $F_A$ takes between consecutive integers will ``open'' a well, and that all wells opened will eventually be closed.  Precisely, $F_{\rho([A])} = G_{\rho([A])}$ when $g=0$, and it is not hard to see $G_{\rho([A])}$ takes values in the odd integers.  Then for each integer $x$ for which $F_{\rho([A])}(x) < F_{\rho([A])}(x-1)$, we have $\frac{1}{2}\left(F_{\rho([A])}(x-1) - F_{\rho([A])}(x)\right)$ wells, the ones whose left sides are at $x-1$.  

So, the number of wells encountered per period of $F_{\rho([A])}$ will be the sum of \linebreak
$\frac{1}{2}\left(F_{\rho([A])}(x-1) - F_{\rho([A])}(x)\right)$ over those $x$ between $0$ and $d-1$ (or between any two values $d-1$ apart) for which it is positive.  This is equal to 
$$\sum_{x=0}^{d-1} \mbox{max}\left\{0, -\eta_{\rho([A])}(x)\right\} = \sum_{D < 0} -D\cdot \#\{x | \eta_{\rho([A])}(x) = D, 0 \leq x \leq d-1\}.$$
If we do this for all $[A] \in \mathcal{MT}_K$ for which $\epsilon([A]) = i$, the sum of the values we get will be
$$\sum_{D < 0} -D\cdot \mathcal{N}(D,i),$$
which is equivalent to the stated expression for the $g=0$ case.  

For the case of higher genus, we note the effect of adding ``spikes'' that are $g$ tall at half-integers on the number of $\mathbb{Z}$ summands.  This is best done by inspection, but we note the general idea: if $g$ is less than half the difference between $F_{\rho([A])}(x-1)$ and $F_{\rho([A])}(x)$, the spike doesn't change the number of wells initiated between $x$ and $x-1$; if $g$ is greater than half this difference, then the number of wells initiated depends on the difference between $g + \frac{1}{2}\left(F_{\rho([A])}(x-1) + F_{\rho([A])}(x)\right)$ and $F_{\rho([A])}(x)$ (and the parity of $g$).  All told, if $\frac{1}{2}\left(F_{\rho([A])}(x-1) - F_{\rho([A])}(x)\right) = D$, then the spike will yield $\left\lfloor \frac{g-D+1}{2} \right\rfloor$ initiated wells if $|D| \leq g$, and $-D$ initiated wells of $D \leq -g$.  This count agrees with the formula given for $a_i$.  

The remaining terms, for the $\Omega^g(k)$ subgroups, are easy to count; the terms contributed between $x-1$ and $x$ depend precisely on the difference between $g + \frac{1}{2}\left(F_{\rho([A])}(x-1) + F_{\rho([A])}(x)\right)$ and $F_{\rho([A])}(x-1)$, which equals $g + \eta_{\rho([A])}(x)$.  An $\Omega^g(k)$ subgroup will be contributed when $k = g+ \eta_{\rho([A])}(x)$, or equivalently when $\eta_{\rho([A])}(x) = k -g$, and the number of times this occurs is counted by $\mathcal{N}(k-g,i)$. \halbox

\vs

\noindent \textit{Proof of Corollary \ref{thm:1.2}.}  Let $\phi:\Sigma \rightarrow \Sigma$ be a periodic diffeomorphism, whose mapping class is of order $d$.  Then there is a representative $\phi'$ of the mapping class of $\phi$ such that $\phi'^n$ either has only isolated fixed points or is the identity.
Since the mapping torus $M(\phi)$ of $\phi$ depends only on its mapping class, we can assume that $\phi$ is such a diffeomorphism, and we assume that $\phi$ is not the identity.  

In this case, the quotient of $\Sigma$ by the action of $\phi$ is an orbifold $B$, and $M(\phi)$ is a degree 0 Seifert fibered space with this orbifold has its base.  In other words, $M(\phi)$ is realized
by special surgery on some knot $B_g \#_{\ell =1}^n O_{p_{\ell}, q_{\ell}}$; and the base orbifold $B$ of the Seifert fibration will have underlying surface of genus $g$, with one cone point of angle $\frac{2\pi}{p_{\ell}}$ for each $\ell$.  These cone points will correspond to the $n$ points fixed by $\phi^i$ for $0 < i < d$: more precisely, point $\ell$ will be fixed by $\phi^{i\cdot\frac{d}{p_{\ell}}}$ for integers $i$.  
In particular, the number of fixed points of $\phi$ is the number 
of $\ell$ for which $p_{\ell}=d$.  This number is just $\mathcal{N}(-1,1)$; by the Lefschetz fixed point theorem, it also equals $\Lambda(\phi)$.

We start with the case $g=0$.  The claim for $i=0$ is clear from Theorem \ref{thm:1.1}.  For the claims when $i=1$,  
it is easy to see that the wells furnished by the proof of Theorem \ref{thm:1.1} all belong to different \sst\ structures.  Indeed, there will be one for each structure of the form $\theta_K(Q; 0, \ldots, 0, d-1,0, \ldots,0)$ for each $\ell$ with $p_{\ell} = d$, and a little algebra rules out the possibility of any of these being equal to one another.  From this, the other claims follow for $i=1$ (the last coming from detailed observation of Theorem \ref{thm:1.4}).

For the case when $g > 0$, the count of $\mathbb{Z}$ summands is the same as before; but now, for each $\mathbb{Z}$ summand, there is also an $\Omega^g(2g-1)$ subgroup.  Closely examining the proof of Theorem \ref{thm:1.4}, however, it is not hard to see that the $\mathbb{Z}$ summands and $\Omega^g(2g-1)$ subgroups come in pairs each lying in the same relative grading; and in fact, they all come about via natural short exact sequences 
$$ 0 \rightarrow \Omega^g(2g-1) \rightarrow \mathbb{Z}[H^1(\#^{2g}S^1 \times S^2)] \rightarrow \mathbb{Z} \rightarrow 0.$$  
Hence, the identification of the groups still holds up, and the other claims follow as before. \halbox

\vs

We now extend Corollary \ref{thm:1.2} to the third-to-highest level.

\begin{theorem} \label{thm:9.3} Keep notation from Corollary \ref{thm:1.2}, and assume that the mapping class of $\phi$ is not of order 1 or 2.  We have
$$\underline{HF}^+\left(Y_0, [-2]\right) \cong \textstyle R^{\Lambda(\phi^2) + \frac{\Lambda(\phi)^2 - \Lambda(\phi)}{2}}$$
where $\Lambda$ denotes Lefschetz number.  If the mapping class is not of order 4, we have in addition that
the $U$-action is trivial, each summand lies in a different \sst structure, each copy of $R$ is supported in a single relative $\mathbb{Z}$-grading, and $T$ lowers this relative grading by $2d(g_{\Sigma} - 1 - i)$.  
\end{theorem}

\begin{proof} This follows along similar lines as Corollary \ref{thm:1.2}.  We again leave out the case where the mapping class of $\phi$ is trivial.  
We also assume that $g=0$; the modifications for higher genus are exactly as above.

The number of fixed points of $\phi^n$ will be the number of $\ell$ for which $\frac{d}{p_{\ell}}$ divides $n$.  Hence, when $n=2$, the Lefschetz number $\Lambda(\phi^2)$ will be $\mathcal{N}(-1,1)$ plus the number of $\ell$ for which $p_{\ell} = \frac{d}{2}$.  

The rank of $\underline{HF}^+(Y_0, [-2])$ will be equal to $\sum_{D < 0} -D\cdot\mathcal{N}(D,2)$, where $\mathcal{N}(D,2)$ is the number of solutions to 
$$\sum_{\ell=1}^n \frac{i_{\ell}}{p_{\ell}} = \frac{2}{d} + D+1$$
for which $0\leq i_{\ell} < p_{\ell}$ for all $\ell$.  If $D<0$ (and $d \ne 1$), the right side will only be positive only if $D=-1$ or if $d=2$ and $D=-2$.  Leaving the latter case aside, we see that for $d \ne 2$, the rank is simply $\mathcal{N}(-1,2)$, and we calculate this to be 
$$\left(\Lambda(\phi^2) - \mathcal{N}(-1,1)\right) + \left(\begin{array}{c}\mathcal{N}(-1,1)\\2 \end{array} \right) = \Lambda(\phi^2) + \frac{\Lambda(\phi)^2 - \Lambda(\phi)}{2}.$$ 

In this case, a little algebra shows that the only time that different wells furnished by the proof of Theorem \ref{thm:1.1} land in the same \sst\ structure is when $\phi$ is isotopic to a diffeomorphism of order 4 with two fixed points, for which $\phi^2$ has an odd number of fixed points.  This case and the case of $\phi$ of order 1 or 2 aside, the other claims follow again. \end{proof}

\section{Sample Calculations}
We give an example.  Let $K = O_{2,1} \# O_{7,3} \# O_{14,1} \subset Y = L(2,1) \# L(7,3) \# L(14,1)$, as depicted in Figure \myfig{7}. Surgery on $K$ with coefficient $-1$ gives the same manifold as $0$-surgery on the (2,7) torus knot.  

\begin{figure}[h]
\label{fig:7}
\centering  \includegraphics[scale=.55]{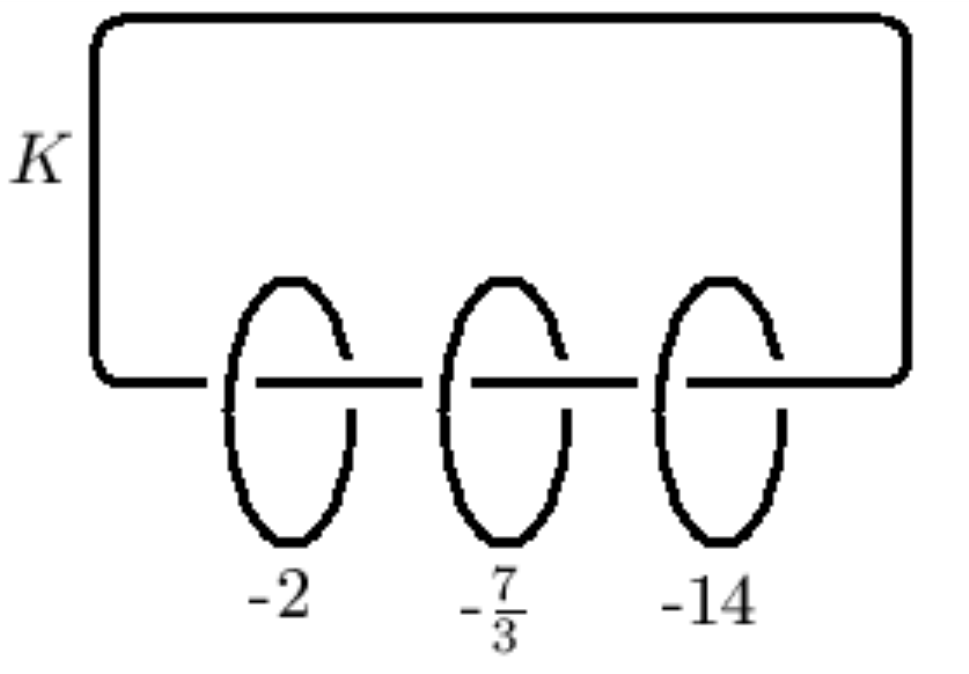}
\caption{The knot $K = O_{2,1} \# O_{7,3} \# O_{14,1}$.}
\end{figure}

So, $(p_1,q_1) = (2,1)$, $(p_2,q_2) = (7,3)$, $(p_3,q_3) = (14,1)$.  
It is easy to see that the set $\mu\mathcal{T}_K$ of $\mu$-torsion structures in $\spc{Y_0}$ consists of $\theta_K([A])$ for where $A = (Q; 0,0,i) \in \widetilde{{AT}_K}$ for $Q \in \mathbb{Z}$ and $i=0,1, \ldots, 13$. 

Let us look at the case of $A_Q = (Q; 0,0,0)$ in depth.  Using the process described in the introduction, we graph the function $F_{A_Q}$ for $Q =-2, -1, 0$.  The knot $K$ is of order 14 in $H_1(Y)$, so these functions will all be the sum of linear functions plus functions that are periodic of period 14.  We show the graphs of each function over two periods in Figure \myfig{8}.  
It is clear that of these three, only the graph of $F_{A_{-1}}$ possesses any wells.  In fact, since $F_{A_Q}$ will be gotten from $F_{A_{-1}}$ by adding a linear function whose slope depends on $Q$, it is easy to see that $F_{A_Q}$ will have no wells for any other value of $Q \in \mathbb{Z}$.  

We can identify the wells in the function for $Q={-1}$: there will be one at each height $1+4n$ for $n \in \mathbb{Z}$, each having trivial $U$-action.  So, there is an absolute lift of the relative $\mathbb{Z}$-grading on $\underline{HF}^+\left(Y_0, \theta_K([A_{-1}]) \right)$, so that
$$\underline{HF}^+\left(Y_0, \theta_K([A_{-1}]) \right) \cong \mathbb{Z}_{(1)} \otimes \mathbb{Z}[T_{(4)}, T_{(4)}^{-1}],$$
where $T_{(4)}$ takes a well to the corresponding well one period to the right and raises $\mathbb{Z}$-grading by $4$. 
We see that $-14\cdot S\ell(A_{-1}) = -4 = \langle c_1(\m{t}_0), [\widehat{dS}] \rangle$.  

The elements $A = (Q; 0,0,i)$ for the other values of $i$ admit a similar analysis. All told, we end up with five elements $\m{t} \in \spc{Y_0}$ for which $\underline{HF}^+(Y_0, \m{t})$ is nontrivial.  We label these as $\m{t}_4, \m{t}_2, \m{t}_0, \m{t}_{-2},$ and $\m{t}_{-4}$, where $\langle c_1(\m{t}_i), [\widehat{dS}] \rangle = i$.  
Let $\mathcal{T}^n_{(s)} \cong \mathbb{Z}[U^{-1}]/U^{-}n\cdot \mathbb{Z}[U^{-1}]$ as $\mathbb{Z}[U]$ modules, graded so that $U^{-i}$ lies in level $s + 2i$ for $0 \leq i < n$ (so that the bottom degree non-trivial elements live in level $s$).
Then, there is a lift of the relative-grading on $\underline{HF}^+(Y_0, \m{t}_i)$ for which
$$\underline{HF}^+(Y_0, \m{t}_i) \cong 
\left\{
\begin{array}{ll}
\mathcal{T}^1_{(-1)} \otimes \mathbb{Z}[T_{(-4)}, T_{(-4)}^{-1}] & i = 4 \\
\mathcal{T}^1_{(-1)} \otimes \mathbb{Z}[T_{(-2)}, T_{(-2)}^{-1}] & i = 2 \\
\left(\mathcal{T}^2_{(-1)} \otimes \mathbb{Z}[T_{(0)}, T_{(0)}^{-1}]\right) \oplus \mathcal{T}^+_{(4)} & i = 0 \\
\mathcal{T}^1_{(-1)} \otimes \mathbb{Z}[T_{(2)}, T_{(2)}^{-1}] & i = -2 \\
\mathcal{T}^1_{(-1)} \otimes \mathbb{Z}[T_{(4)}, T_{(4)}^{-1}] & i = -4, \\
\end{array}
\right. $$
where $U$ lowers the grading by 2, and all the groups have trivial $U$ action except for the one for $\m{t}_0$.  Forgetting about relative $\mathbb{Z}$-gradings, this can be summarized as saying that
$$\underline{HF}^+(Y_0) \cong HF^+(Y_0) \otimes \mathbb{Z}[T, T^{-1}]$$
in light of Proposition \ref{thm:8.1} of \cite{OSAG}.

\subsection{Higher genus}  
Now, let $K_g$ be the connect sum of $K$ with $B_g$ for some $g \geq 0$, as depicted in Figure \myfig{9}. The functions that we need to compute for each value of $A \in \mathcal{MT}_{K_g}$ are the same as the ones for the corresponding value $A \in \mathcal{MT}_{K}$ (as $\mathcal{MT}_{K_g} = \mathcal{MT}_K$ as sets), except with $g$ added at each half-integer value of $x$.

We just look at the case where $A = (-1;0,0,0)$ for varying values of $g$.  Figures \myfig{10} and \myfig{11} shows graphs of $F_A$ for $g=1$ and $g=2$, respectively.
For $Y_0 = Y_0(K_1)$, we have a short exact sequence
$$ 0 \rightarrow \left(\left(\Omega^1(1)\right)^8_{(1)} \oplus \left(\Omega^1(1)\right)_{(3)} \oplus \left(\Omega^1(1)\right)_{(5)}\right) \otimes \mathbb{Z}[T_{(4)}, T^{-1}_{(4)}] $$
$$ \rightarrow \underline{HF}^+\left(Y_0(K_1), \theta_{K_1}([A]) \right) \rightarrow \left((\mathcal{T}^1)^9_{(5)} \oplus \mathcal{T}^1_{(7)} \oplus \mathcal{T}^1_{(9)}\right) \otimes \mathbb{Z}[T_{(4)}, T^{-1}_{(4)}] \rightarrow 0.$$
The $\mathcal{T}^i$ summands are found just as they are before.  As for the $\Omega^g(k)$ subgroups, we have $g=1$, and $\Omega^1(k)$ is non-trivial only when $k=1$.  So, we look for half-integral $x$ between $0$ and $14$ for which $F_A(x) - F_A(x-\frac{1}{2}) = 1$,  and for each such $x$, $F_A(x) - 1$ gives the grading of the summand.  The values of $x$ that we find are $x = \frac{1}{2}, \frac{3}{2},\frac{5}{2},\frac{7}{2},\frac{9}{2},\frac{11}{2},\frac{13}{2},\frac{15}{2},\frac{19}{2},\frac{23}{2}$; we have $F_A(\frac{19}{2}) - 1 = 3$, $F_A(\frac{23}{2})-1= 5$, and $F_A(x)-1= 1$ for the other eight values of $x$.

In the case where $g=2$, we similarly find the sequence
$$ \hspace*{-50pt} 0 \rightarrow \left(\left(\Omega^2(2)\right)^8_{(2)} \oplus \left(\Omega^2(3)\right)_{(3)} \oplus \left(\Omega^2(2)\right)_{(4)}
\oplus \left(\Omega^2(3)\right)_{(5)} \oplus \left(\Omega^2(2)\right)_{(6)} \oplus  \left(\Omega^2(3)\right)_{(7)} \oplus \left(\Omega^2(1)\right)_{(7)}\right) \otimes \mathbb{Z}[T_{(4)}, T^{-1}_{(4)}] $$
$$\rightarrow \underline{HF}^+\left(Y_0(K_2), \theta_{K_2}([A]) \right) \rightarrow 
\left((\mathcal{T}^1)^8_{(5)} \oplus (\mathcal{T}^1)^3_{(7)} \oplus \mathcal{T}^2_{(5)} \oplus (\mathcal{T}^1)^2_{(9)}\right) \otimes \mathbb{Z}[T_{(4)}, T^{-1}_{(4)}] \rightarrow 0.$$ \newpage

\begin{figure}[H]
\label{fig:8}
\centering  \hspace*{-50pt} \includegraphics[scale=.60]{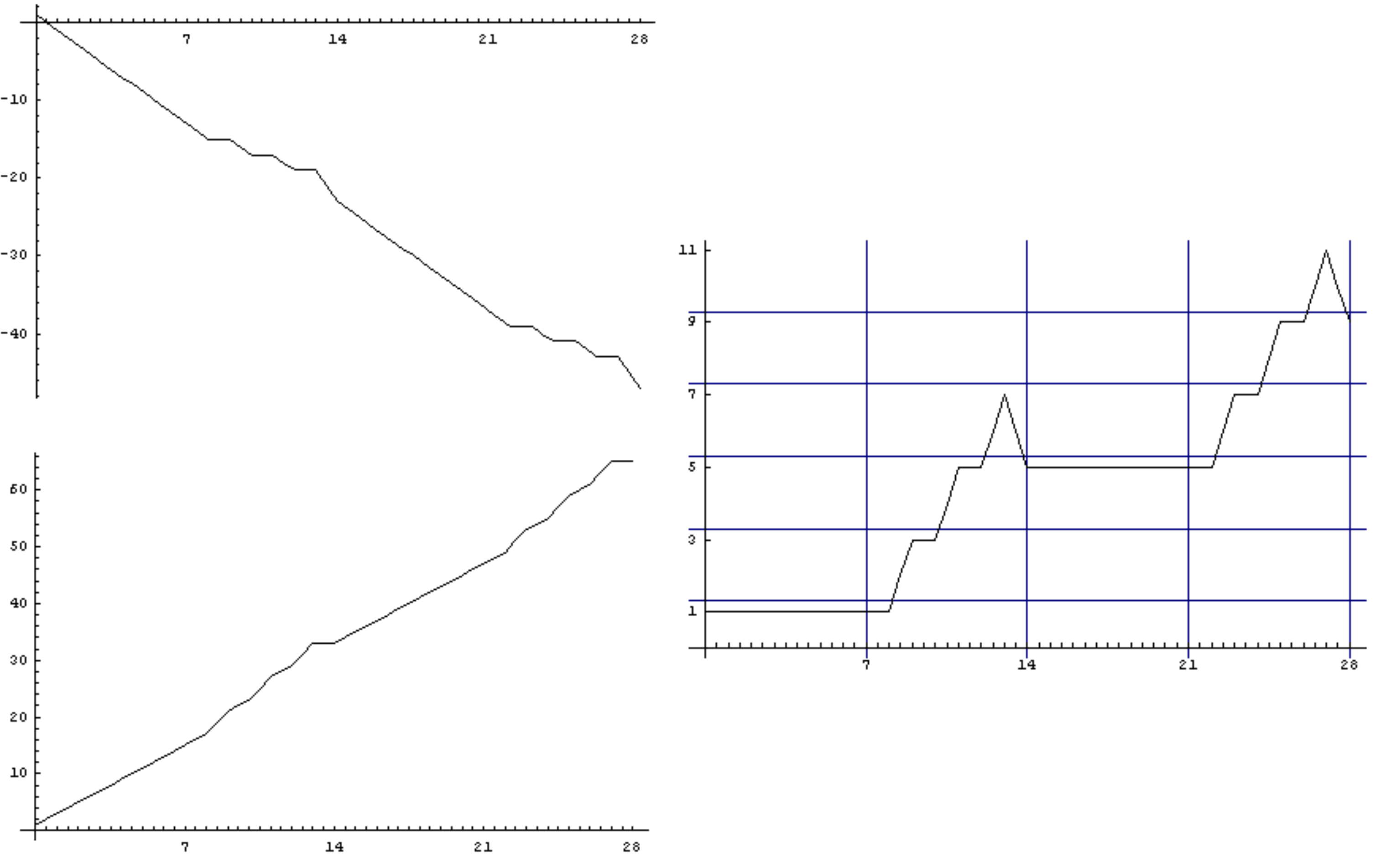}
\caption{The graphs of the functions $F_A$ where $A = (-2;0,0,0)$ (top), $A=(-1;0,0,0)$ (right) and $A=(0;0,0,0)$ (bottom).  Only when $A=(-1;0,0,0)$ does $F_A$ have any wells, which correspond to finite intervals where the graph runs below any of the horizontal lines (which are placed just above odd integers). }
\end{figure}

\begin{figure}[h!]
\label{fig:9}
\centering   \includegraphics[scale=.50]{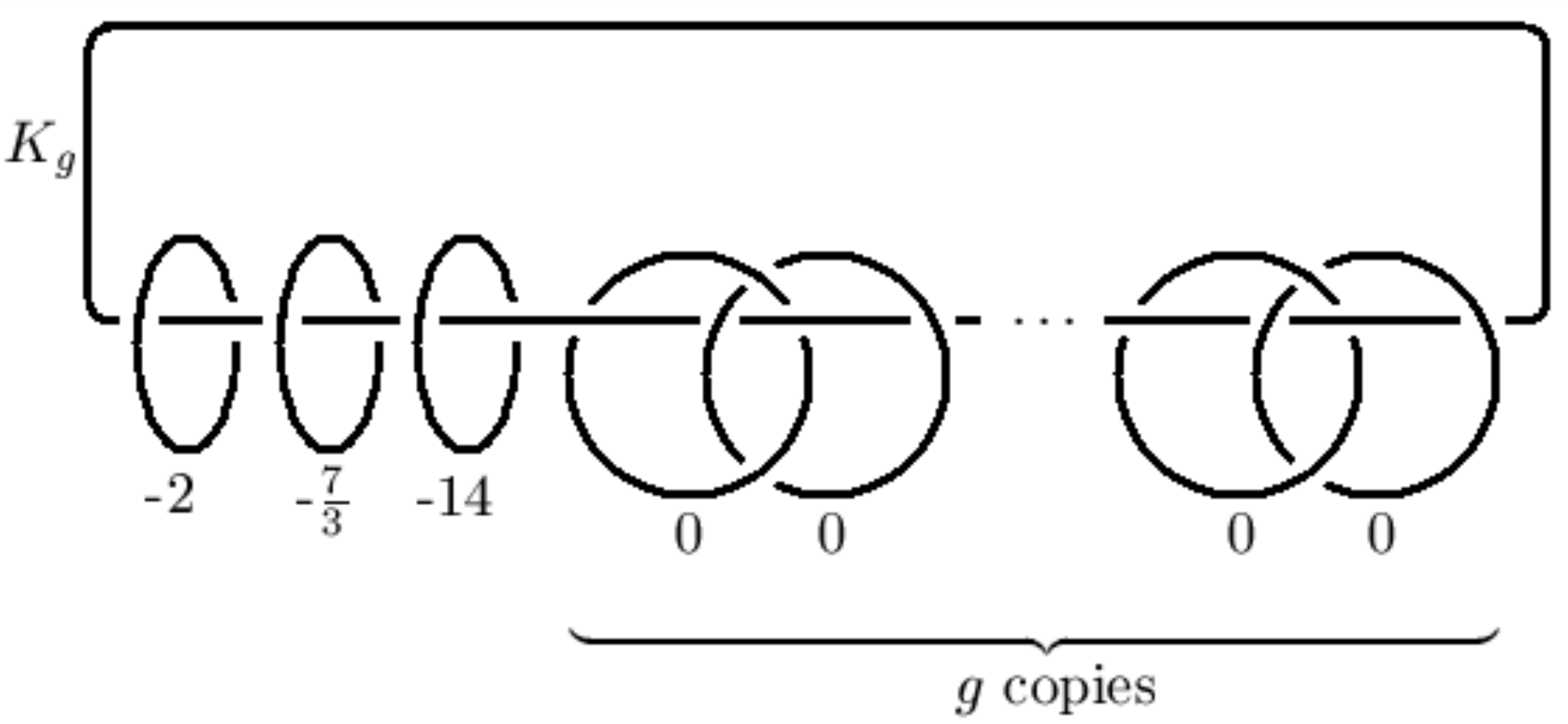}
\caption{The knot $K_g = O_{2,1} \# O_{7,3} \# O_{14,1} \# B_g$.}
\end{figure}

\newpage

\begin{figure}[H]
\label{fig:10}
\centering   \includegraphics[scale=.55]{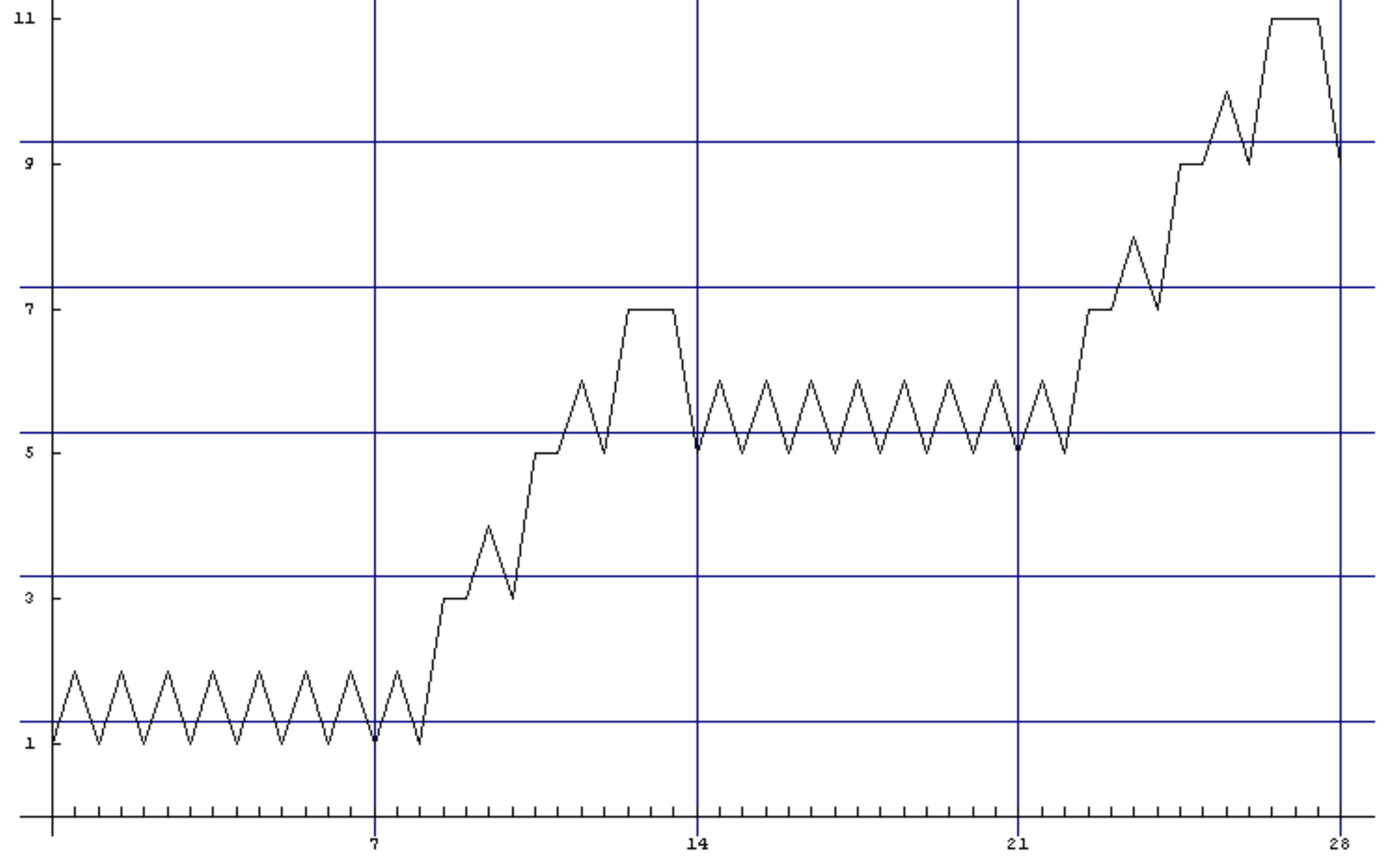}
\caption{The graph of the function $F_A$ when $A=(0;0,0,0) \in \mathcal{MT}_{K_1}$.}
\end{figure}

\begin{figure}[H]
\label{fig:11}
\centering   \includegraphics[scale=.55]{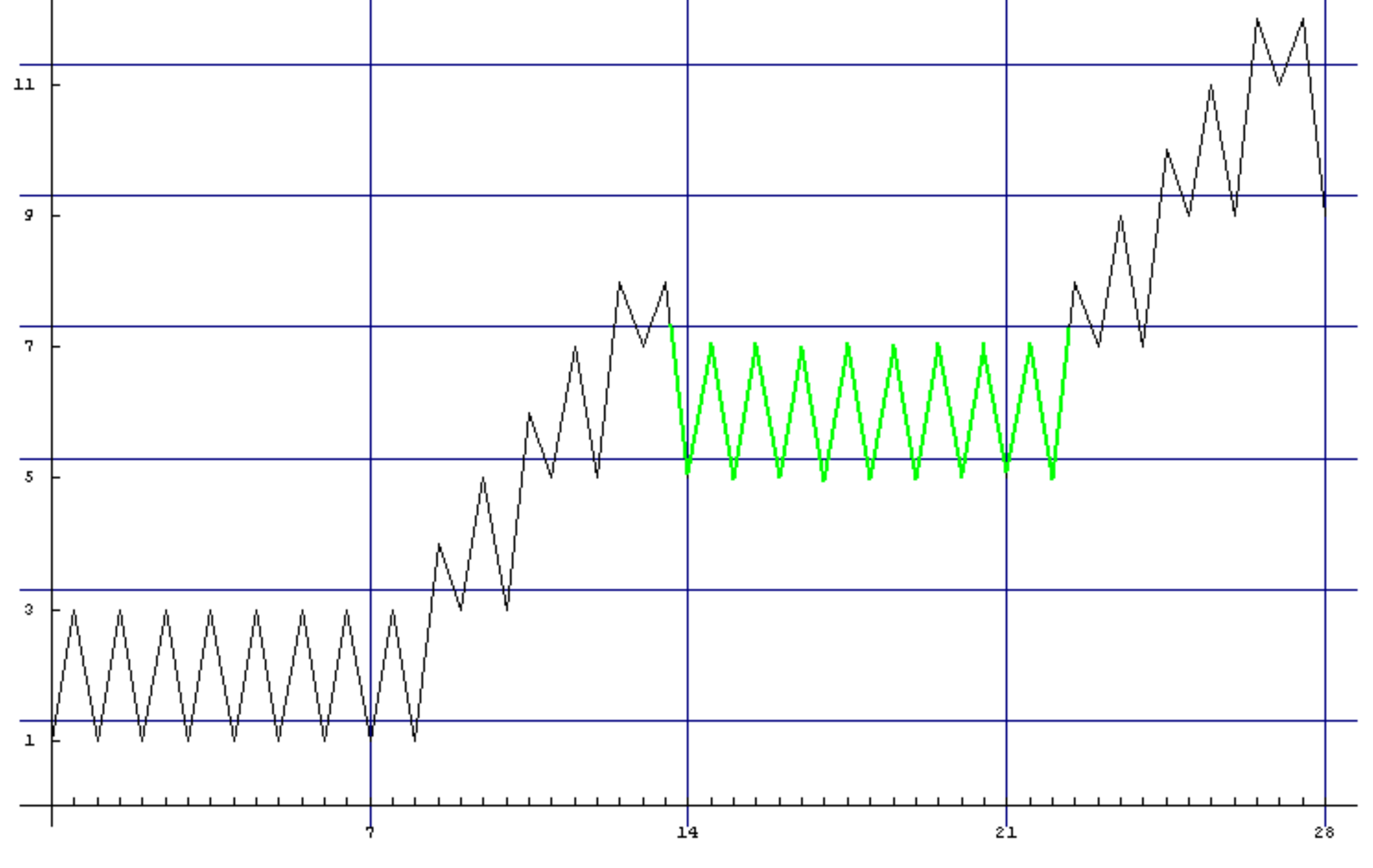}
\caption{The graph of the function $F_A$ when $A=(0;0,0,0) \in \mathcal{MT}_{K_2}$.  We highlight the portion of the graph corresponding to a well $w$ at height $7$, on which $U$ acts non-trivially: $Uw$ is equal to the sum of nine wells at height $5$.}
\end{figure}

\section{Appendix}
We now return to some results from Section 2 whose proofs are fairly straightforward, but somewhat technical.

Take an oriented, rationally nullhomologous knot $K \subset Y$.  We want to shift to a concrete setting to talk about algebraic topology issues more clearly.  So, present $Y$ as surgery on a framed link $L \subset S^3$; this also serves as a Kirby diagram for a 2-handlebody $U_{\infty}$ with boundary $Y$.  Order the components of the link, and assign each an orientation; call the linking matrix $G_{\infty}$.  This specifies a basis for $H_2(U_{\infty})$, with respect to which the intersection form is just $G_{\infty}$.  
To be precise, the basis is given by taking oreinted Seifert surfaces for the link components pushed into the interior of $U_{\infty}$, and then capped off.

We can represent $K$ as an extra component on this link, which comes equipped with orientation; the implicit framing $\lambda$ will be some integer, which we call $I_0$.  Let $K_{+N}$ be $K$ with framing $I_0+N$.  The framed link $L_N = L \cup K_{+N}$ is a Kirby diagram for a handlebody $U_N$ with boundary $Y_N$ (allowing $N=0$, of course).  Again, $H_2(U_N)$ has a canonical ordered basis, with respect to which the intersection form on $U_N$ is equal to the linking matrix $G_N$ of $L_N$.  Written in block form, 
we have 
$$G_N = \left( \begin{array}{cc} G_{\infty} & \vec{k} \\
\vec{k}^{\mbox{ }T} & I_0 + N \\
\end{array} \right),$$ 
where the last row and column correspond to the component $K_{+N}$.  Note that $U_N = U_{\infty} \cup_Y W_N$.

We form another 2-handlebody as follows.  To $L_0$, add a $0$-framed meridian to $K$, and an $N$-framed meridian of the $0$-framed meridian; call this link $L_{0N}$, the represented 2-handlebody $U_{0N}$, and the linking matrix $G_{0N}$.  Ordering and orienting the components of $L_{0N}$, we have
$$G_{0N} = \left( \begin{array}{cccc} G_{\infty} & \vec{k} & \vec{0} & \vec{0} \\
\vec{k}^{\mbox{ }T} & I_0 & 1 & 0 \\
\vec{0} & 1 & 0 & 1\\
\vec{0} & 0 & 1 & N\\
\end{array} \right).$$ 
It is not hard to see that, in fact, the cobordism $X_{\al\be\ga}$ of Section 2.5, when glued to a tubular neighborhood of a sphere with self intersection $N$ along $L(N,1)$, gives a cobordism $W_{0N}$ such that $U_{0N} = U_0 \cup_{Y_0} W_{0N}$.  We define $\m{S}_{0N}(\m{t}_0) \subset \spc{W_{0N}}$ to be those structures whose first Chern class evaluates to $N$ on this sphere.  Then, of course, $\m{S}_N(\m{t}_0)$ can be described as the restrictions of $\m{S}_{0N}(\m{t}_0) \subset \spc{W_{0N}}$ to $Y_N$.

We denote elements of second homology and cohomology of $U_*$ by vectors and covectors respectively (where $U_*$ denotes any of $U_{\infty}, U_N$, and $U_{0N}$), so that evaluation of a cohomology class on a homology class is given by the normal dot product.  We also denote elements of $H^2(U_*, \partial U_*)$ by covectors, so that the Poincar{\'e} dual of $(\vec{a})_{U_*} \in H_2(U_*)$ is just $(\vec{a}^{\mbox{ }T})_{U_*, \partial U_*}$.  We will use the same notation for the corresponding cohomology groups with rational coefficients. In the long exact cohomology sequence for  $(U_*, \partial U_*)$, we have 
\begin{equation}
\label{eq:20}
H^2(U_*, \partial U_*) \rightarrow H^2(U_*) \rightarrow H^2(\partial U_*) \rightarrow 0,
\end{equation} 
with the first map given in terms of our bases by right multiplication by $G_*$.  We often write elements of $H^2(Y_N)$ as $h = (\vec{b}^{\mbox{ }T}, c)_{U_N}|_{Y_N}$, so as to cooperate with the block form expression of $G_N$.  

Since all the handlebodies are simply-connected, the \sst\ structures on each can be identified via the first Chern class with characteristic covectors of the corresponding matrices (i.e., covectors whose $i^{\mbox{\scriptsize{th}}}$ component is congruent mod 2 to the $i^{\mbox{\scriptsize{th}}}$ diagonal entry of the matrix).  We denote the \sst\ structures by $\langle \vec{a}^{\mbox{ }T}\rangle_{U_*}$ so that $c_1(\langle \vec{a}^{\mbox{ }T}\rangle_{U_*}) = (\vec{a}^{\mbox{ }T})_{U_*}$.  (In general, angle brackets will be used to signify that we are talking about a \sst\ structure rather than a cohomology class.)  

The restriction maps $\spc{U_*} \rightarrow \spc{\partial U_*}$, $\spc{U_N} \rightarrow \spc{W_N}$, and $\spc{U_{0N}} \rightarrow \spc{W_{0N}}$ are all surjections, and the restriction maps are all equivariant with respect to the action of $H^2(U_*)$ (acting on the targets via the restriction maps).  It follows that every \sst\ structure on $W_*$ or $Y_*$ can be specified as the restriction of one on some $U_*$, and, identifying $\spc{U_*}$ with $\mbox{Char}(G_{*})$, that we can view $\spc{\partial U_*}$ as $\mbox{Char}(G_*)/2\cdot\mbox{Im} G_*$. (The same can be said with \sst\ structures replaced by elements of second cohomology.)

It is not hard to see that $U_{0N}$ is diffeomorphic to $U_N\# S$, with $S$ denoting either $\mathbb{CP}^2\#\overline{\mathbb{CP}}^2$ or $S^2 \times S^2$; the handleslides necessary to realize this are shown in Figure \myfig{12}.   We can also perform all the same moves if we bracket the framings for all but the two meridians shown in Figure \myfig{12}a, which would depict $W_{0N}$.
We can see a spanning disk for the $N$-framed component in $-Y_0$ in Figure \myfig{12}d, since this component is unlinked from the link of bracketed components; gluing this to the core of the corresponding handle gives a distinguished sphere $V$ embedded in $W_{0N}$.  

Now, we can describe the sets mentioned at the beginning of Section 2.4. 

\begin{prop} \label{thm:A.1} Suppose $\m{t}_0 = \langle \vec{b}^{\mbox{ }T}, c \rangle_{U_0}|_{Y_0} \in \spc{Y_0}$.  Then we have 
$$\m{S}^N_0(\m{t}_0) = \big\{\langle \vec{b}^{\mbox{ }T}, c + 2iN \rangle_{U_0}|_{Y_0} \big| i \in \mathbb{Z} \big\},$$
$$\m{S}_{0N}(\m{t}_0) = \big\{\langle \vec{b}^{\mbox{ }T}, c + 2iN, 2j, N \rangle_{U_{0N}}|_{W_{0N}} \big| i,j \in \mathbb{Z} \big\},$$
$$\m{S}_N(\m{t}_0) = \big\{\langle \vec{b}^{\mbox{ }T}, c + (2i-1)N \rangle_{U_N}|_{Y_N} \big| i \in \mathbb{Z} \big\},$$
$$\m{S}_{N\infty}(\m{t}_0) = \big\{\langle \vec{b}^{\mbox{ }T} + 2i\vec{k}^{\mbox{ }T}, c + 2iI_0 + (2j-1)N \rangle_{U_N}|_{W_N} \big| i,j \in \mathbb{Z} \big\},$$
and
$$\m{S}_{\infty}(\m{t}_0) = \big\{\langle \vec{b}^{\mbox{ }T} + 2i\vec{k}^{\mbox{ }T} \rangle_{U_{\infty}}|_Y \big| i \in \mathbb{Z} \big\}.$$
In particular, $\m{S}_{\infty}(\m{t}_0)$ is independent of $N$.
\end{prop}

\begin{proof}  The first claim is clear.  The second claim follows by looking at the preimage of $\m{S}^N_0(\m{t}_0)$ under the restriction map induced by the inclusion $Y_0 \rightarrow U_0 \rightarrow U_{0N}$.  We must have the last component equal to $N$ since structures in $\m{S}_{0N}(\m{t}_0)$ have specified evaluation on $V$, and it is not hard to see that all rows of the matrix $G_{0N}$ besides the last two vanish in $H^2(W_{0N})$.  

The fourth claim follows from the third similarly, and the fifth follows straightforwardly from the fourth.  To identify the restriction of $\m{S}_{0N}(\m{t}_0)$ to $\spc{Y_N}$, we perform the change of basis corresponding to the moves shown in Figure \myfig{12}.  In terms of the new basis, where the last three components correspond respectively to the $I_0+N$-, \linebreak $I_0$-, and $0$-framed components, $c_1\big(\langle \vec{b}^{\mbox{ }T}, c + 2iN, 2j, N \rangle_{U_{0N}}|_{W_{0N}}\big)$ will be written as $(\vec{b}^{\mbox{ }T} - 2j\vec{k}^{\mbox{ }T}, c -2jI_0 + (2i-1)N, c + 2iN, 2j)$.  The first two components give the restriction of this to $H^2(U_N)$, and the third claim follows straightforwardly from this after noting that $(\vec{k}^{\mbox{ }T}, I_0)_{U_N}|_{Y_N} = (\vec{0}^{\mbox{ }T}, -N)_{U_N}|_{Y_N}$.  \end{proof}

\begin{figure}[h!]
\label{fig:12}
\centering   \includegraphics[scale=.60]{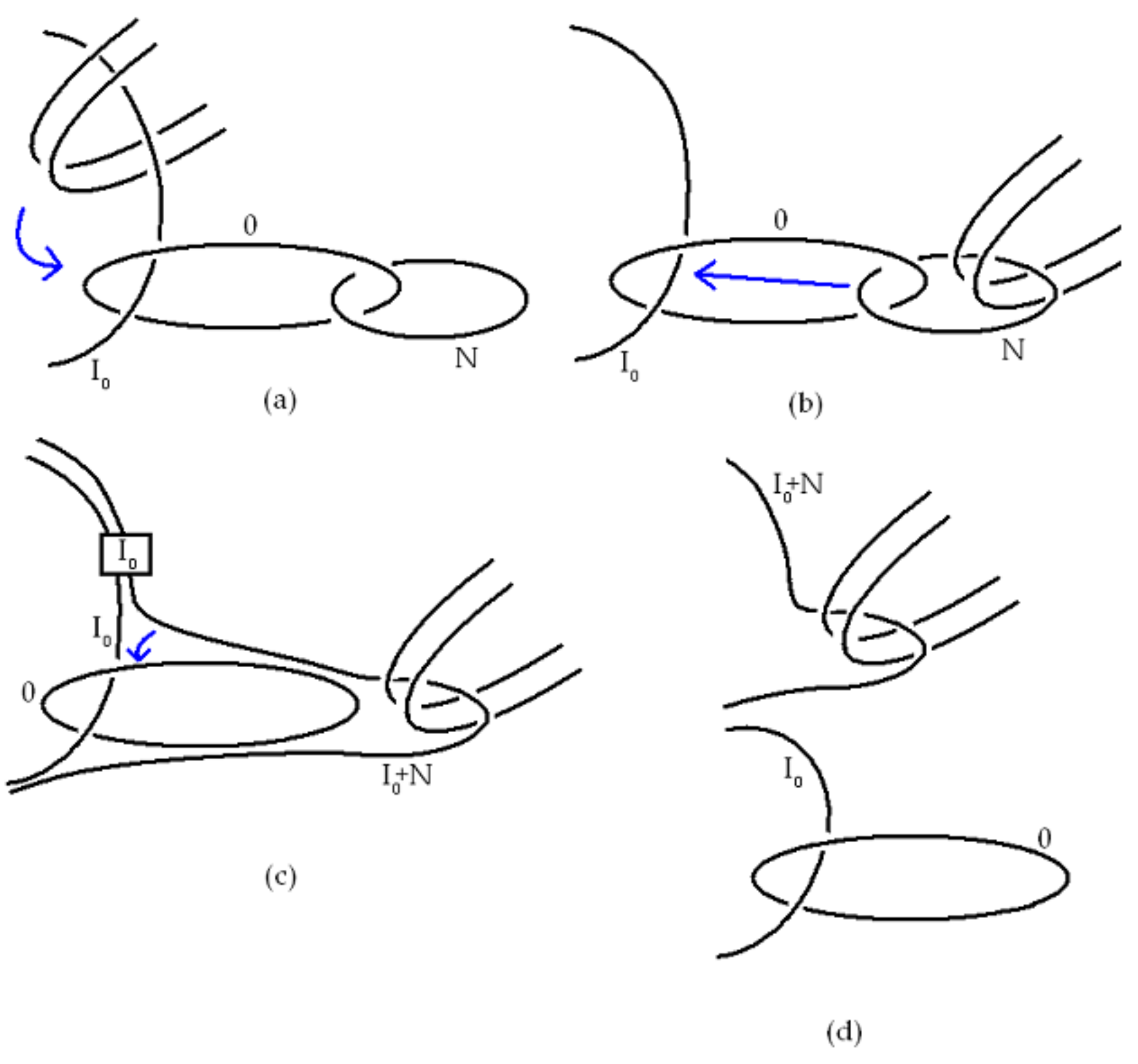}
\caption{We start with the diagram for $X_N$, and then we handleslide all of the components linking the $I_0$-framed knot over the $0$-framed unknot, so that afterward this unknot is the only component that the $I_0$-framed knot links with.  We then handleslide the $N$-framed unknot over the $I_0$-framed knot, turning the unknot into an $I_0 + N$ framed knot of the same type.  Finally, we handle slide the $I_0 + N$-framed knot over the $0$-framed component enough times to unlink it from the $I_0$-framed knot.  We are left with the Kirby diagram for $X_N$ together with a knot with a 0-framed meridian, which presents $X_N \# \mathbb{CP}^2\#\overline{\mathbb{CP}}^2$ or $X_N \# S^2 \times S^2$.   }
\end{figure}

Let us explicitly identify some (co)homology classes.
We may choose a surface $P$ in $U_{\infty}$ with $\partial P = K$ (i.e., a ``pushed-in'' Seifert surface), so that $(\vec{k})_{U_{\infty}, \partial U_{\infty}}$ represents $P$ in $H_2(U_{\infty}, \partial U_{\infty})$.
Let us also choose a Seifert surface $dS$ for $d\cdot K$ in $Y$, where $d$ is the order of $K$ in $H_1(Y)$.  Gluing $-dS$ to $d\cdot P$ yields a class $S_1$ of $H_2(U_{\infty}; \mathbb{Z})$.  Choose some $(\vec{p}_0)_{U_{\infty}} \in H_2(U_{\infty}; \mathbb{Q})$ so that $S_1 = (d\vec{p}_0)_{U_{\infty}}$.  This choice must satisfy $G_{\infty}\vec{p}_0 = \vec{k}$. 

We can also glue $P$ to $F$, recalling that the latter is the core of the 2-handle of $W_N$. The resulting class will be the oriented generator of $H_2(U_N)$ corresponding to the link component $K$, and so it will be written as      
$$\left( \begin{array}{c} \vec{0} \\
1 \\
\end{array} \right)_{U_N} \in H_2(U_N).$$
Now choose $\widetilde{dF}$ to be $d\cdot F$ glued to $dS$; it follows that if $i_*: H_2(W_N) \rightarrow H_2(U_N)$ is induced by inclusion,
$$i_*\left([\widetilde{dF}]\right) = \left( \begin{array}{c} \mbox{-}d\vec{p}_0 \\
d \\
\end{array} \right)_{U_N}.$$

We will also need to identify $\mbox{PD}[F]|_{W_N}$ (we think of $F$ as a generator of $H_2(W_N, Y)$, so its dual technically lives in $H^2(W_N, Y_N)$).  To find this, we take advantage of the commutative diagram 
$$ \begin{array}{ccccc}
H_2(W_N) & \cong & H^2(W_N, \partial W_N) & \overset{j_{W_N}^*}{\longrightarrow} & H^2(W_N) \\
i_* \mbox{ }\downarrow & & & & i^* \mbox{ }\uparrow \\
H_2(U_N) & \cong & H^2(U_N, \partial U_N) & \overset{j_{U_N}^*}{\longrightarrow} & H^2(U_N) \\
\end{array} $$
where the isomorphisms are from Poincar{\'e} duality and the other maps are induced from inclusions. (The claimed commutativity is not quite as obvious as obvious as it seems, but is a reasonable exercise in diagram chasing.)  Given this, it follows that $j^*_{W_N}\left(\mbox{PD}[\widetilde{dF}]\right) = \big(\vec{0}^{\mbox{ }T}, d(I_0 - \vec{p}_0^{\mbox{ }T}G_{\infty}\vec{p}_0 + N)\big)_{U_N}|_{W_N}$, since $j^*_{U_N}$ is represented by multiplication by $G_N$.  It is easy to see that $(d\vec{k}^{\mbox{ }T}, d\vec{p}_0^{\mbox{ }T}G_{\infty}\vec{p}_0)_{U_N}|_{W_N} = 0$, noting that $(d\vec{k}^{\mbox{ }T}, d\vec{p}_0^{\mbox{ }T}G_{\infty}\vec{p}_0)_{U_N} = j^*_{U_N}\left(\mbox{PD}[S_1]\right)$ and that $\mbox{PD}[S_1]$ restricts to the trivial element of $H^2(W_N, Y_N) \cong \mathbb{Z}$.  Hence $j^*_{W_N}\left(\mbox{PD}[\widetilde{dF}]\right)$ can also be written as $\big(d\vec{k}^{\mbox{ }T}, d(I_0 + N)\big)_{U_N}|_{W_N}$, and so at least up to torsion,
\begin{equation}
\label{eq:21}
\mbox{PD}[F]|_{W_N} = (\vec{k}^{\mbox{ }T}, I_0 + N)_{U_N}|_{W_N}. 
\end{equation}
That this actually holds precisely follows from the fact that it restricts to $\mbox{PD}[K]$ in $H^2(Y)$, and the kernel of $H^2(W_N) \rightarrow H^2(Y)$ contains no nontrivial torsion elements.

Now, let us note the following facts.

\begin{lemma} \label{thm:A.2} The term $\kappa$ is equal to $I_0 - \vec{p}_0^{\mbox{ }T}G_{\infty}\vec{p}_0$, and is independent of the particular choice of $\vec{p}_0$.  
The longitude $N\mu + \lambda$ is special if and only if $\kappa = -N$.  The order of $\mu$ in $H_1(Y_0)$ is $d|\kappa |$.  The element $(\vec{b}^{\mbox{ }T}, c)_{U_N}$ of $H^2(U_N)$ restricts to a torsion element of $H^2(Y_N)$ if and only if $(\vec{b}^{\mbox{ }T})_{U_{\infty}}|_Y$ is torsion and either $\kappa \ne -N$ or $c = \vec{b}^{\mbox{ }T}\cdot\vec{p}_0$ for any choice of $\vec{p}_0$.  
\end{lemma}

\begin{proof} All of these follow from straightforward matrix algebra, using the long exact sequence (\ref{eq:20}).  Note that if $\kappa = -N$, then the quantity $c - \vec{b}^{\mbox{ }T}\cdot\vec{p}_0 = \langle (\vec{b}^{\mbox{ }T}, c)_{U_N}|_{W_N}, [\widetilde{dF}] \rangle$ depends only on $(\vec{b}^{\mbox{ }T}, c)_{U_N}|_{Y_N}$.\end{proof}

\noindent \textit{Proof of Proposition \ref{thm:2.5}.} Straightforward, utilizing Proposition \ref{thm:A.1} and Lemma \ref{thm:A.2}. \halbox  

\vs

We now describe how to compute squares of elements of $H^2(W'_N)$; we first compute $[\widetilde{dF'}]^2$.

\vs

\noindent \textit{Proof of Proposition \ref{thm:2.7}.} We can square elements of $H_2(W_N)$ using the intersection form on $U_N$, since $W_N$ is a submanifold of $U_N$.   Using our identification of $[\widetilde{dF}]$ from above, we compute that $[\widetilde{dF}]^2 = d\big(-d\vec{k}^{\mbox{ }T}\vec{p}_0 + d(I_0 + N)\big) = d^2(\kappa + N)$.  Hence $[\widetilde{dF'}]^2 = -d^2(\kappa + N)$ in $H_2(W'_N)$, due to the reversal of orientation. \halbox

\vs

Next, we compute squares of general boundary-torsion elements of $H^2(W_N)$ and $H^2(W'_N)$.
Denote by $\alpha$ the inclusion of $\mbox{PD}[\widetilde{dF}] \in H^2(W_N, \partial W_N)$ into $H^2(W_N)$, $\alpha = j_{W_N}^*\big(\mbox{PD}[\widetilde{dF}]\big)$. Then of course $\alpha^2 = d^2(\kappa + N)$.  

We will also need to know the evaluation $\langle \alpha, [\widetilde{dF}]\rangle$.  
According to the diagram $\alpha$ equals $i^*\Big(j_{U_N}^*\big(\mbox{PD}[i_*([\widetilde{dF}])]\big)\Big)$.  Thus,  $$\langle \alpha, [\widetilde{dF}]\rangle = \big\langle j_{U_N}^*\big(\mbox{PD}[i_*([\widetilde{dF}])]\big), i_*([\widetilde{dF}]) \big\rangle = \Big((-d\vec{p}_0^{\mbox{ }T}, d)\cdot G_N\Big)\cdot\left( \begin{array}{c} \mbox{-}d\vec{p}_0 \\
d \\
\end{array} \right) = d^2(\kappa + N).$$

By the universal coefficients theorem, we have that 
$$H^2(W_N)/\mbox{Torsion} \cong \mbox{Hom}\big(H_2(W_N), \mathbb{Z}\big) \cong \mbox{Hom}\big(H_2(Y), \mathbb{Z}\big) \oplus \mbox{Hom}\big(\mathbb{Z}\cdot [\widetilde{dF}], \mathbb{Z}\big).$$ 
Recalling the map $j:H^2(W_N; \mathbb{Z}) \rightarrow H^2(\partial W_N; \mathbb{Z}) \rightarrow H^2(\partial W_N;\mathbb{Q})$, if $\phi \in \mbox{Ker }j$ and $\langle \phi, [\widetilde{dF}]\rangle = i$, it then follows that $\phi = \frac{i}{d^2(\kappa + N)} \alpha$ in $H^2(W_N)/\mbox{Torsion}$ since elements of $\mbox{Ker }j$ evaluate trivially on $H_2(Y)$.  So $\phi^2 = \Big(\frac{i}{d^2(\kappa + N)}\Big)^2 d^2(\kappa + N) = \frac{i^2}{d^2(\kappa + N)}$.  

By excision and the long exact sequence of $(U_N, W_N)$, the restriction map $i^*:H^2(U_N) \rightarrow H^2(W_N)$ is surjective, and of course $\langle (\vec{b}^{\mbox{ }T},c)_{U_N}, i_*([\widetilde{dF}]) \rangle = \langle i^*\big((\vec{b}^{\mbox{ }T},c)_{U_N}\big), [\widetilde{dF}] \rangle$.  These facts, together with the above, allow us to compute the square of any element in $H^2(W_N)$.  

Of course, squaring elements of $H^2(W'_N)$ is exactly the same, except that every Poincar{\'e} dual gets a minus sign. So we arrive at the value $-\frac{(c - \vec{b}^{\mbox{ }T}\vec{p}_0)^2}{(\kappa + N)}$ for the square of the class $(\vec{b}^{\mbox{ }T},c)_{U_N}|_{W'_N}$.  Since the evaluation of this class on any lift $[\widetilde{dF'}]$ is equal to $d(c - \vec{b}^{\mbox{ }T}\vec{p}_0)$, we have the following.

\begin{lemma}
\label{thm:A.3}  
The square of $\alpha \in \mbox{Ker }j \subset H^2(W'_N)$ is given by 
$$\alpha^2 = -\frac{\big(\langle \alpha, [\widetilde{dF}] \rangle\big)^2}{d^2(\kappa + N)}$$
(recall that $j$ is the composition of obvious maps $H^2(W'_N; \mathbb{Z}) \rightarrow H^2(\partial W'_N; \mathbb{Z}) \rightarrow H^2(\partial W'_N; \mathbb{Q})$).
\end{lemma}

\noindent \textit{Proof of Proposition \ref{thm:2.9}.}  Let $\m{t}_0 = \langle \vec{b}^{\mbox{ }T}, c \rangle_{U_0}|_{Y_0}$, and set $\m{s}^i_N = \langle \vec{b}^{\mbox{ }T}, c + (2i-1)N \rangle_{U_N}|_{W'_N}$ and $\m{t}^i_N = \m{s}^i_N|_{Y_N}$.   

We want to compute $\m{s}_{K+}(\m{t}^i_N)$.  Note that we have 
$$\frac{[\widetilde{dF'}]^2}{d^2} = -N, \mbox{ }\mbox{ }\mbox{ } \frac{\langle c_1(\m{s}^i_N), [\widetilde{dF'}] \rangle}{d} = c - \vec{b}^{\mbox{ }T}\vec{p}_0 + (2i-1)N.$$
Therefore, the function $Q_K(j, \m{s}^i_N) = 0$ when 
$$j = \frac{c - \vec{b}^{\mbox{ }T}\vec{p}_0 + (2i-1)N}{2N}; $$
and so $\m{s}_{K+}(\m{t}^i_N) = \m{s}^i_N + \lfloor j \rfloor \mbox{PD}[F']|_{W'_N}.$

Now, 
$$ \lfloor j \rfloor = i+ \left\lfloor \frac{c - \vec{b}^{\mbox{ }T}\vec{p}_0}{2N} - \frac{1}{2} \right\rfloor,$$
and for large $N$ this will just equal $i-1$. 
So, 
$$\m{s}_{K+}(\m{t}^i_N) = \langle \vec{b}^{\mbox{ }T} - 2(i-1)\vec{k}^{\mbox{ }T}, c - 2(i-1)I_0 + N \rangle_{U_N}|_{W'_N}.$$  
Also, $\mbox{PD}[F']|_{W'_N}$ is represented by $\big(-\vec{k}^{\mbox{ }T}, -(I_0+N)\big)_{U_N}|_{W'_N}$, where we pick up a negative sign from (\ref{eq:21}) since we are taking duals with respect to the orientation of $W'_N$; hence, 
$$\m{s}_{K+}(\m{t}^i_N) + \mbox{PD}[F']|_{W'_N} = \langle \vec{b}^{\mbox{ }T} - 2i\vec{k}^{\mbox{ }T}, c - 2iI_0 - N \rangle_{U_N}|_{W'_N}.$$

Therefore, we can just use Lemma \ref{thm:A.3} to compute 
$$q_K\left(\m{s}_{K+}(\m{t}^i_N)\right) = \frac{c_1^2\left(\m{s}_{K+}(\m{t}^i_N) + \mbox{PD}[F']|_{W'_N}\right) - c_1^2\left(\m{s}_{K+}(\m{t}^i_N)\right)}{4} = c - \vec{b}^{\mbox{ }T}\vec{p}_0.$$
Notice that this value is independent of $\m{t}^i_N$.  In fact, it is not hard to see that there is a Seifert surface $dS$ for $dK$, which can be capped off in $Y_0$ to give a surface $\widehat{dS}$, which will be represented by
$$\left. \left( \begin{array}{c} d\vec{p}_0 \\
-d \\
\end{array} \right)_{U_0}\right|_{Y_0},$$
and so 
$$ \frac{\langle c_1(\m{t}_0), [\widehat{dS}] \rangle}{d} = -(c - \vec{b}^{\mbox{ }T}\vec{p}_0),$$
which completes the proof. \halbox

\end{document}